\documentclass{amsart}

\usepackage{amsmath,amssymb}
\usepackage{amsthm}
\usepackage{stmaryrd}
\usepackage{enumerate}
\usepackage{url}
\usepackage{wasysym}
\usepackage{color}
\usepackage[normalem]{ulem}

\usepackage{graphicx}
\makeatletter
\providecommand{\bigsqcap}{%
  \mathop{%
    \mathpalette\@updown\bigsqcup
  }%
}
\newcommand*{\@updown}[2]{%
  \rotatebox[origin=c]{180}{$\m@th#1#2$}%
}
\makeatother

\newtheorem{theorem}{Theorem}[section]
\newtheorem{lemma}[theorem]{Lemma}
\newtheorem{corollary}[theorem]{Corollary}
\newtheorem{proposition}[theorem]{Proposition}
\newtheorem{fact}[theorem]{Fact}

\theoremstyle{definition}
\newtheorem{definition}[theorem]{Definition}
\newtheorem{example}[theorem]{Example}
\newtheorem{remark}[theorem]{Remark}
\newtheorem{convention}[theorem]{Convention}

\newtheorem{question}[theorem]{Question}

\def\cU{\mathcal{U}}

\def\cL{\mathcal{L}}

\def\tp{\operatorname{tp}}

\def\Av{\operatorname{Av}}
\def\Th{\operatorname{Th}}
\def\SL{\operatorname{SL}}
\def\st{\operatorname{st}}
\def\id{\operatorname{id}}
\def\supp{\operatorname{supp}}
\def\ex{\operatorname{ex}}
\def\dcl{\operatorname{dcl}}

\def\Aut{\operatorname{Aut}}
\def\inv{\operatorname{inv}}
\def\ext{\operatorname{ext}}
\def\fs{\operatorname{fs}}
\def\stab{\operatorname{Stab}}
\def\conv{\operatorname{conv}}

\def\Ind{\setbox0=\hbox{$x$}\kern\wd0\hbox to 0pt{\hss$\mid$\hss}
\lower.9\ht0\hbox to 0pt{\hss$\smile$\hss}\kern\wd0}

\def\Notind{\setbox0=\hbox{$x$}\kern\wd0\hbox to 0pt{\mathchardef
\nn=12854\hss$\nn$\kern1.4\wd0\hss}\hbox to
0pt{\hss$\mid$\hss}\lower.9\ht0 \hbox to 0pt{\hss$\smile$\hss}\kern\wd0}

\title{Definable convolution and idempotent Keisler measures II}
\author{Artem Chernikov}

\author{Kyle Gannon}

\address{Department of Mathematics\\
University of California Los Angeles\\
Los Angeles, CA 90095-1555, USA}
\email{chernikov@math.ucla.edu}

\address{Department of Mathematics\\
University of California Los Angeles\\
Los Angeles, CA 90095-1555, USA}

\email{gannon@math.ucla.edu}

\begin{document}

\begin{abstract} 
We study convolution semigroups of invariant/finitely satisfiable Keisler measures in NIP groups. We show that the ideal (Ellis) subgroups are always trivial and describe minimal left ideals in the definably amenable case, demonstrating that they always form a Bauer simplex. Under some assumptions, we give an explicit construction of a minimal left ideal in the semigroup of measures from a minimal left ideal in the corresponding semigroup of types (this includes the case of SL$_{2}(\mathbb{R})$, which is not definably amenable). We also show that the canonical push-forward map is a homomorphism from definable convolution  on $\mathcal{G}$ to classical convolution on the compact group $\mathcal{G}/\mathcal{G}^{00}$, and use it to classify $\mathcal{G}^{00}$-invariant idempotent measures.
\end{abstract}

\dedicatory{With gratitude to Ehud Hrushovski, whose beautiful ideas have deeply influenced the authors.}

\maketitle

\section{Introduction}

This paper is a continuation of \cite{ChGan}, but with a focus on NIP groups and the dynamical systems associated to the definable convolution operation. It was demonstrated in \cite{ChGan} that when $T$ is an NIP theory expanding a group, $\mathcal{G}$ is a monster model of $T$, and $G \prec \mathcal{G}$,  the spaces of global $\Aut(\mathcal{G}/G)$-invariant Keisler measures  and Keisler measures which are finitely satisfiable in $G$ (denoted $\mathfrak{M}_{x}^{\inv}(\mathcal{G},G)$ and $\mathfrak{M}_{x}^{\fs}(\mathcal{G},G)$,  respectively) form left-continuous compact Hausdorff semigroups under definable convolution $\ast$ (see Fact \ref{fac: conv is a semigroup in NIP}). Equivalently, the semigroup $\left(\mathfrak{M}_{x}^{\fs}(\mathcal{G},G), \ast \right)$ can be described as the Ellis semigroup of the dynamical system given by the action of $\conv(G)$, the convex hull of $G$ in the space of global measures finitely satisfiable in $G$, on the space of measures $\mathfrak{M}_{x}^{\fs}(\mathcal{G},G)$ (see \cite[Theorem 6.10 and Remark 6.11]{ChGan}).
The main purpose of this paper is to study the structure of these semigroups, as well as to provide a description of idempotent measures via type-definable subgroups in some cases. 

In Section \ref{sec: Prelims} we review some preliminaries and basic facts on convolution in compact topological groups (Section \ref{sec: prelim class conv}), Keisler measures (Section \ref{sec: Keisl meas prelim}), definable convolution in NIP groups (Section \ref{sec: def conv NIP prelims}), Ellis semigroups (Section \ref{sec: Ellis theory prelims}) and Choquet theory (Section \ref{sec: Choquet theory prelims}).

In Section \ref{sec: def conv vs conv} we study the relationship between the semigroups $\mathfrak{M}_{x}^{\inv}(\mathcal{G},G)$ and $\mathfrak{M}_{x}^{\fs}(\mathcal{G},G)$ (under definable convolution)  and the classical convolution semigroup of regular Borel probability measures on the compact topological group $\mathcal{G}/\mathcal{G}^{00}$. We demonstrate that the push-forward along the quotient map is a surjective, continuous, semigroup homomorphisms from definable convolution to classical convolution on $\mathcal{G}/\mathcal{G}^{00}$ (Theorem \ref{invar:conv}), mapping  idempotent Keisler measures onto   idempotent Borel measures on $\mathcal{G}/\mathcal{G}^{00}$ (Corollaries \ref{cor:wendel} and \ref{cor: pushf idemp onto}).

We have shown in \cite[Theorem 5.8]{ChGan} that, by analogy to the classical theorem of Kawada-Ito/Wendel for compact groups (Fact \ref{compact:corr}), there is a one-to-one correspondence between idempotent measures on a stable group and its type-definable subgroups (namely, every idempotent measure is the unique translation-invariant measure on its type-definable stabilizer group). In NIP groups, this correspondence fails (Example \ref{example:pair}), but revised versions of this statement can be recovered in some cases.
In particular, using the results of Section \ref{sec: def conv vs conv}, we demonstrate in Section \ref{sec: G00 inv idempt class} that a $\mathcal{G}^{00}$-invariant idempotent measure in an NIP group $\mathcal{G}$ is a (not necessarily unique) invariant measure on its type-definable stabilizer group.
In future work, we examine further cases of the classification of idempotent measures in NIP groups, including the generically stable case. 

 In Section \ref{sec: struct of conv semigrps} we study the semigroups $\left(\mathfrak{M}^{\inv}_{x}(\mathcal{G},G),* \right)$ and $\left(\mathfrak{M}^{\fs}_{x}(\mathcal{G},G),* \right)$ for an NIP group $\mathcal{G}$ through the lens of Ellis Theory. We demonstrate that the ideal subgroups of any minimal left ideal (in either $\mathfrak{M}^{\fs}_{x}(\mathcal{G},G)$ or $\mathfrak{M}^{\inv}_{x}(\mathcal{G},G)$) are always trivial (Proposition \ref{trivial}). This is due to the presence of the convex structure, in contrast to the case of types in definably amenable NIP groups (where, due to the proof of the Ellis group conjecture in \cite{ChSi}, the ideal subgroups are isomorphic to $\mathcal{G}/\mathcal{G}^{00}$). 
 We also classify minimal left ideals in both $\mathfrak{M}^{\fs}_{x}(\mathcal{G},G)$ and $\mathfrak{M}^{\inv}_{x}(\mathcal{G},G)$ when $\mathcal{G}$ is definably amenable. In this case, any minimal left ideal in $\mathfrak{M}_{x}^{\fs}(\mathcal{G},G)$ is also trivial (Proposition \ref{fs:I}), while $\mathfrak{M}_{x}^{\inv}(\mathcal{G},G)$ contains a unique minimal left ideal (which is also two-sided). This unique ideal is precisely the collection of measures in $\mathfrak{M}_{x}^{\inv}(\mathcal{G},G)$ which are $\mathcal{G}$-right-invariant (Proposition \ref{ideal:inv}; this is in contrast to minimal left ideals in $\mathfrak{M}_{x}^{\fs}(\mathcal{G},G)$ corresponding to $G$-\emph{left}-invariant measures). It is also a compact convex set, and moreover a \emph{Bauer simplex} (see Definition \ref{def: Bauer simplex}). In particular the set of its extreme points is closed, and consists precisely of the lifts  $\mu_p$ of the Haar measure on $\mathcal{G}/\mathcal{G}^{00}$ for $p \in S_x^{\inv}(\mathcal{G},G)$ an $f$-generic type of $\mathcal{G}$ (Corollary \ref{cor: ex inv f-generic}). If the group $\mathcal{G}$ is fsg, this minimal ideal is also trivial (Corollary \ref{cor: fsg}). We also observe that if $\mathcal{G}$ is not definably amenable, then the minimal left ideal of $\mathfrak{M}_{x}(\mathcal{G},G)$ has infinitely many extreme points (Remark \ref{prop: many idemp in non def am}). 
 See Theorem \ref{thm: main props of min ideals} for a more precise summary of the results of the section.

In Section \ref{sec: CIGS} we isolate certain conditions on $\mathcal{G}$, applicable in particular to some non-definably amenable groups,  which allows us to 
describe a minimal left ideal of $\mathfrak{M}_{x}^{\dagger}(\mathcal{G},G)$ for $\dagger \in \{\fs, \inv\}$ in terms of a minimal ideal in the corresponding semigroup of types.
We prove the following two results.  Suppose that $I$ is a minimal left ideal of $S_{x}^{\dagger}(\mathcal{G},G)$ and $u$ is an idempotent in $I$ such that $u * I$ is a compact group under the induced topology (we refer to this condition as \emph{CIG1}, see Definition \ref{def: CIG1}). Then $\mathfrak{M}(I) * \mu_{u * I}$ is a minimal left ideal of $\mathfrak{M}_{x}^{\dagger}(\mathcal{G},G)$, where $\mu_{u * I}$ is the Keisler measure corresponding to the normalized Haar measure on $u*I$ and $\mathfrak{M}(I)$ is the set of Keisler measures supported on $I$ (Theorem \ref{thm: CIG1 min ideal}). 
Under a stronger assumption \emph{CIG2} on $\mathcal{G}$ (see Definition \ref{def: CIG2}), we show that a minimal left ideal of $\mathfrak{M}_{x}^{\dagger}(\mathcal{G},G)$ is affinely homeomorphic to a collection of regular Borel probability measures over a natural quotient of $I$, in particular it is a Bauer simplex (Theorem \ref{thm: min ideal in CIG2 affine}).
In particular, $\SL_{2}(\mathbb{R})$ falls into both of these categories (Example \ref{ex: SL2R}).

\subsection*{Acknowledgements}
We are very grateful to the referee for a particularly detailed and insightful report and many valuable suggestions on improving the paper.

We thank Sergei Starchenko and Anand Pillay for helpful conversations. Both authors were partially supported by the NSF CAREER grant DMS-1651321, and Chernikov was additionally supported by  a Simons fellowship.

\section{Preliminaries}\label{sec: Prelims}
Given $r_1, r_2 \in \mathbb{R}$ and $\varepsilon \in \mathbb{R}_{>0}$, we  write $r_1 \approx_{\varepsilon} r_2 $ if $|r_1 - r_2| < \varepsilon$. For $n \in \mathbb{N}_{\geq 1}$, $[n] = \{1, 2, \ldots, n \}$.

\subsection{The classical setting} \label{sec: prelim class conv}
Before discussing the model theoretic setting, we recall some classical facts concerning compact Hausdorff spaces, measures, and compact topological groups.

\begin{fact}\label{meas:facts} Let $X,Y$ be compact Hausdorff spaces and $f:X \to Y$.  
\begin{enumerate}[$(i)$]
\item Let $\mathcal{M}(X)$ be the set of all regular Borel probability measures on $X$. Then $\mathcal{M}(X)$ is a compact Hausdorff space under the weak--$^*$ topology, with the basic open sets of the form
\begin{equation*} \bigcap_{i=1}^{n} \left\{\mu \in \mathcal{M}(X): r_i < \int_{X} f_i d\mu < s_i \right\}
\end{equation*} 
for $n \in \mathbb{N}$, $r_i < s_i \in \mathbb{R}$, $f_i : X \to \mathbb{R}$ continuous for $i \in [n]$.
\item A net of measures $(\mu_i)_{i \in I}$ in $\mathcal{M}(X)$ converges to a measure $\mu$ if and only if for any continuous  $f: X \to \mathbb{R}$, 
\begin{equation*} \lim_{i \in I} \int_{X} f d\mu_{i} = \int_{X} f d\mu.
\end{equation*} 
\item A (Borel) measurable map $f: X \to Y$ induces the \emph{push-forward} map $f_{*}:\mathcal{M}(X) \to \mathcal{M}(Y)$ defined by: for any Borel  $A \subseteq Y$, $f_{*}(\mu)(A) = \mu(f^{-1}(A))$. Then for any Borel function $h: Y \to \mathbb{R}$ such that $h \in L^{1}(f_{*}(\mu))$,
\begin{equation*} \int_{Y} h df_{*}(\mu) = \int_{X} \left( h \circ f  \right) d\mu. 
\end{equation*} 
Moreover the map $f_{*}$ is \emph{affine}: for any $r_1, \ldots, r_n \in [0,1]$ with $\sum_{i=1}^{n} r_i = 1$  and $\mu_1, \ldots, \mu_n \in \mathcal{M}(X)$,
\begin{equation*} f_{*}\left( \sum_{i =1}^{n} r_i \mu_i \right) = \sum_{i=1}^{n} r_i f_{*}(\mu_i). 
\end{equation*}
\item If $f:X \to Y$ is continuous, then $f_{*}: \mathcal{M}(X) \to \mathcal{M}(Y)$ is continuous. If $f:X \to Y$ is also surjective, then $f_{*}$ is also surjective. 
\end{enumerate}
\end{fact}

\begin{remark} Let $X$ be any compact Hausdorff space and let $C(X)$ be the collection of continuous functions from $X$ to $\mathbb{R}$. 
We consider $C(X)$ as a normed vector space with the $||\cdot||_{\infty}$ norm, 
i.e. $||f||_\infty = \sup_{x \in X} |f(x)|$.   The dual of $C(X)$, denoted $C(X)^{*}$, is the space of all linear functionals, i.e.~maps from $C(X)$ to $\mathbb{R}$ which are continuous with respect to the norm topology on $C(X)$. The weak--$^*$ topology on $C(X)^{*}$ is the coarsest topology such that for any $a \in X$, the map $E_{a}:C(X) \to \mathbb{R}$ given by $E_{a}(f) = f(a)$ is continuous. We remark that $\mathcal{M}(X)$ can be naturally viewed as a subset of $C(X)^{*}$ via $\mu \mapsto \int - d\mu$. The topology induced from $C(X)^{*}$ on $\mathcal{M}(X)$ is both compact and Hausdorff. Moreover, $\mathcal{M}(X)$ forms a convex subset of $C(X)^{*}$. 
\end{remark} 

\begin{convention} If $f:X \to \mathbb{R}$ is a measurable function, we sometimes write $\int_{X}f d\mu$ simply as $\mu(f)$. 
\end{convention}

\begin{definition} Let $X$ be a compact Hausdorff space and $\mu \in \mathcal{M}(X)$. The \emph{support of $\mu$} is $\supp(\mu) := \{a \in X:$ for any open neighborhood $U$ of $a$, $\mu(U) > 0\}$. Then $\supp(\mu)$  is a non-empty closed subset of $X$. We remark that $\mu(\supp(\mu)) = 1$. 
\end{definition} 

By a \textit{compact group} we mean a compact Hausdorff topological group where both the multiplication $-\cdot - : G \times G \to G$ and inverse $^{-1}:G \to G$ maps are continuous.  
\begin{definition} Let $G$ be a compact group and  $\mu,\nu \in \mathcal{M}(G)$. Then their \emph{convolution product}\footnote{To stay consistent with the notation in \cite{ChGan}, we will use ``$*$'' to denote definable convolution (defined later in this section) and ``$\star$'' to denote classical convolution.} $\mu \star \nu$ is the unique regular 
Borel measure on $G$ such that for any continuous function $f: G \to \mathbb{R}$, 
\begin{equation*} 
\int_{G} f(z) d(\mu \star \nu)(z) = \int_{G} \int_{G} f(x \cdot y)d\mu(x) d\nu(y). 
\end{equation*} 
Equivalently, $\mu \star \nu$ is the unique regular Borel measure on $G$ such that for any Borel subset $E$ of $G$,
\begin{equation*} \mu \star \nu(E) = \int_{G} \mu(Ex^{-1}) d\nu(x).
\end{equation*} 
See e.g.~\cite{Stromberg} for a proof the this equivalence.
\end{definition}  

\begin{remark}\label{remark:conv} Let $G$ be a compact group. 
\begin{enumerate} 
\item If $a,b \in G$, then $\delta_{a \cdot b} = \delta_{a} \star \delta_{b}$ (where $\delta_a$ denotes the Dirac measure on $a$). 
\item The space $\mathcal{M}(G)$ is a compact topological semigroup under convolution. In particular, the map $\star: \mathcal{M}(G) \times \mathcal{M}(G) \to \mathcal{M}(G)$ is associative and continuous.
\item The map $\delta: G \to \mathcal{M}(G), a \mapsto \delta_a$ is an embedding of topological semigroups. 
\end{enumerate} 
\end{remark} 

\begin{definition} Suppose that $G$ is a compact group and $\lambda \in \mathcal{M}(G)$. We say that $\lambda$ is \emph{idempotent} if $\lambda \star \lambda = \lambda$. 
\end{definition} 

\noindent The following theorem classifies idempotent measures on compact groups. The first proof of this theorem is due to Kawada and Ito \cite[Theorem 3]{KIto} and uses representation theory of compact groups. This result was rediscovered a decade and a half later by Wendel using semigroup theory \cite[Theorem 1]{Wendel}. 


\begin{fact}\label{compact:corr} Suppose $G$ is a compact group and $\lambda \in \mathcal{M}(G)$. Then the following are equivalent: 
\begin{enumerate} 
\item $\lambda$ is idempotent. 
\item $\supp(\lambda)$ is a closed subgroup of $G$ and $\lambda|_{\supp(\lambda)}$ is the normalized Haar measure on $\supp(\lambda)$. 
\end{enumerate} 
\end{fact} 

\noindent We are interested in which ways this theorem  can be recovered for Keisler measures on definable groups. 
However, finding subgroups of a monster model is more difficult than directly applying this classification theorem since the support of an \textit{idempotent Keisler measure} is a collection of types and not a subgroup of the model. Instead, we will also need to take into account a measure's stabilizer. This distinction does not arise in the compact group setting since the stabilizer of an idempotent probability measure is the same as its support. We take a moment to be precise about this statement. 

\begin{definition} Suppose $G$ is a compact group and $\lambda \in \mathcal{M}(G)$. Its \emph{right stabilizer} is $\stab(\lambda) := \{a \in G:$ for any Borel set $B \subseteq G$, $\lambda(B\cdot a) = \lambda(B)\}$.  
\end{definition} 

\begin{lemma}\label{lem: sup equals stab in compact grp} Let $G$ be a compact group and $\lambda \in \mathcal{M}(G)$. If $\lambda$ is idempotent, then $ \supp(\lambda) =  \stab(\lambda)$.
\end{lemma} 

\begin{proof} Suppose  $a \in \supp(\lambda)$. By Fact \ref{compact:corr}, $\supp(\lambda)$ is a closed subgroup of $G$ and $\lambda|_{\supp(\lambda)}$ is the normalized Haar measure on $\supp(\lambda)$. Hence for any Borel subset $C$ of $\supp(\lambda)$, $\lambda(C \cdot a) = \lambda(C)$. Let $X$ be a Borel subset of $G$. Then 
\begin{equation*} \lambda(X \cdot a) = \lambda((X \cdot a) \cap \supp(\lambda)) = \lambda((X \cap \supp(\lambda)) \cdot a) = \lambda( X \cap \supp(\lambda)) = \lambda(X),  
\end{equation*} 
hence $a \in \stab(\lambda)$. 

Conversely, suppose $a \in \stab(\lambda)$, but $a \not \in \supp(\lambda)$. By Fact \ref{compact:corr}, this implies that $\left( \supp(\lambda) \cdot a \right) \cap \supp(\lambda)= \emptyset$. However $\lambda(\supp(\lambda)) = 1$ and also
\begin{equation*} \lambda(\supp(\lambda) \cdot a) = \lambda(\supp(\lambda) \cdot a \cap \supp(\lambda)) = \lambda (\emptyset) = 0,
\end{equation*}
a contradiction.
\end{proof} 

Finally, we recall a couple of facts on integrating functions over compact groups.
\begin{fact}\label{int:comp} Suppose that $G$ is a compact group and $H$ is a closed subgroup of $G$ with normalized Haar measure $\lambda_{H}$. Let $h \in H$ and $f:G \to \mathbb{R}$ be a Borel function such that $f|_{H} \in L^{1}(\lambda_{H})$, i.e. the restriction of $f$ to $H$ is integrable. Then 
\begin{equation*} \int_{G} f(x) d\lambda(x) = \int_{G} f(x \cdot h) d\lambda(x), 
\end{equation*} 
where $\lambda$ is the measure on $G$ defined by $\lambda(X) = \lambda_{H}(X \cap H)$.
\end{fact} 

\noindent The next fact appears hard to find explicitly stated in the literature, so we provide a proof for completeness. 

\begin{lemma}\label{cont:int} Let $G$ be a compact group and assume that $f: G \to \mathbb{R}$ is continuous. Let $\mu \in \mathcal{M}(G)$. Then the map $ b \mapsto \int_{G} f(x\cdot b)d\mu$  from $G$ to $\mathbb{R}$ is continuous. 
\end{lemma} 

\begin{proof} Define $h : G \to \mathbb{R}$ via $h(b) = \int_{G} f(x \cdot b)d\mu$. We first show that for every $\varepsilon > 0$ there exists an open neighborhood $U$ of the identity $e \in G$ such that for any $b \in U$, $\sup_{x \in G}|f(x) - f(x \cdot b)| < \varepsilon$. Fix $\varepsilon > 0$, and suppose the statement does not hold. Then for every neighborhood $U$ of $e$ there exist some $b_{U} \in U$ and $c_{U} \in G$ such that $|f(c_{U}) - f(c_{U} \cdot b_{U})| \geq \varepsilon$. Let  $N$ be the set of all open neighborhoods of $e$, then $N$ is a directed set under reverse inclusion and $(c_U \cdot b_U)_{U \in N}$ is a net.  Since $G$ is compact, we may pass to a convergent subnet $N'$ of $N$ so that  $(c_U \cdot b_U)_{U \in N'}$ converges. Note also that still $\lim_{U \in N'} b_{U} = e$. Since the nets $(c_{U} \cdot b_{U})_{U \in N'}$ and  $( b_{U})_{U \in N'}$ both converge and $G$ is a topological group, the net $( c_{U})_{U \in N'}$ also converges, let $c := \lim_{U \in N'} c_{U}$.  By continuity of $f$,
\begin{equation*} 
\lim_{U \in N'} f(c_U) = f(c) = \lim_{U \in N' } f(c_U \cdot b_U)
\end{equation*} 
Then $\lim_{U \in N'} |f(c_U) - f(c_U\cdot b_U)| = 0$, but $|f(c_U) - f(c_U \cdot b_U)| \geq \varepsilon$ for each $U \in N'$ by assumption,  a contradiction. 

We now show that $h$ is continuous. Let $(r,s) \subseteq \mathbb{R}$, $g_0 \in h^{-1}((r,s))$, and $\varepsilon := \min\{|h(g_0)-r|,|h(g_0)-s|\}$. By the paragraph above, there exists an open neighborhood of the identity $U$ such that for any $b \in U$, $\sup_{x \in G} |f(x) - f(x \cdot b)| < \frac{\varepsilon}{2}$. We will show that the open set $g_0 \cdot U$ is a subset of $h^{-1}((r,s))$ containing $g_0$. Note that $g_0 \in g_0 \cdot U$ since $e \in U$. Now suppose that $g_1 \in g_0 \cdot U$, so $g_1 = g_0 \cdot b_1$ for some $b_1 \in U$. Since for any $g \in G$, the action $k(x) \mapsto k(x \cdot g)$ of $G$ on the space  $C(G)$ of continuous functions from $G$ to $\mathbb{R}$ preserves the uniform norm,  acting on the right by $g_0$ derives $\sup_{x \in G}|f(x \cdot g_0) - f(x \cdot g_0 \cdot b_1)| < \frac{\varepsilon}{2}$. Therefore

\begin{equation*}
h(g_1) = \int_{G} f(x \cdot g_1) d\mu = \int_{G} f(x \cdot g_0 \cdot b_1) d\mu  \approx_{\frac{\epsilon}{2}} \int_{G} f(x \cdot g_0)d\mu = h(g_0). 
\end{equation*} 
Hence $h(g_1) \in (r,s)$ and so $g_0 \cdot U$ is an open subset of  $h^{-1}((r,s))$. Therefore $h^{-1}((r,s))$ is also open, and the map $h$ is continuous.  \end{proof} 


\subsection{Model theoretic setting}

For the most part, our notation is standard. Let $T$ be a complete first-order theory in a language $\mathcal{L}$ and assume that $\mathcal{U}$ is a sufficiently saturated and homogeneous model of $T$. While the rest of the paper is focused on the setting where $T$ expands the theory of a group, this section contains results about arbitrary theories. We write $x,y,z, \ldots$ to denote arbitrary finite tuples of variables. If $x$ is a tuple of variables and $A \subseteq \mathcal{U}$, then $\mathcal{L}_{x}(A)$ is the collection of formulas with free variables in $x$ and parameters from $A$, modulo logical equivalence. We write $\mathcal{L}_{x}$ for $\mathcal{L}_{x}(\emptyset)$. Given a partitioned formula $\varphi(x;y)$ with object variables $x$ and parameter variables $y$, we let $\varphi^*(y;x) := \varphi(x;y)$ be the partitioned formula with the roles of $x$ and $y$ reversed.

As usual, $S_x(A)$ denotes the space of types over $A$, and if $A \subseteq B \subseteq \mathcal{U}$ then 
$S^{\fs}_x(B,A)$ (respectively, $S^{\inv}_x(B,A)$) denotes the closed set of types in $S_x(B)$ that are finitely satisfiable in $A$ (respectively, invariant over $A$). Throughout this paper, we will want to discuss the spaces $S^{\inv}_{x}(B,A)$ and $S^{\fs}_{x}(B,A)$ simultaneously, so we let $S^{\dagger}_{x}(B,A)$ denote ``either $S^{\fs}_x(B,A)$ or $S_{x}^{\inv}(B,A)$''.  If $\varphi(x) \in \mathcal{L}_{x}(\mathcal{U})$, $[\varphi(x)] = \{p \in S_{x}(\mathcal{U}): \varphi(x) \in p\}$. Given a set $X \subset \mathcal{U}^{x}$ and $A \subseteq \mathcal{U}$ a small set of parameters, we say that $X$ is $\bigvee$-definable over $A$  (respectively, $\bigwedge$- or type-definable over $A$) if for some $\left\{ \psi_{i}(x) \right\}_{i \in I}$ with $\psi_{i}(x) \in \mathcal{L}_{x}(A)$ we have $X = \bigcup_{i \in I} \psi_{i}(\mathcal{U})$ (respectively, $X = \bigcap_{i \in I} \psi_{i}(\mathcal{U})$). And $X$ is $\bigvee$-definable (respectively, type-definable) if it is  $\bigvee$-definable (respectively, type-definable) over $A$ for some small $A \subseteq \cU$.

\begin{definition}\label{def: type def sets} If $X$ is a $\bigvee$-definable subset of $\mathcal{U}^{x}$, we let $[X] := \bigcup_{i \in I} [\psi_{i}(x)]$ where $\bigvee_{i \in I} \psi_i(x)$ is any $\bigvee$-definition of $X$. Likewise, if $X$ is a type-definable subset of $\mathcal{U}^{x}$, we let $[X] := \bigcap_{i \in I} [\phi_{j}(x)]$ where $\bigwedge_{i \in I} \phi_{j}(x)$ is any $\bigwedge$-definition of $X$. Note that $[X]$ does not depend on the choice of the small set of formulas defining $X$.
\end{definition} 

In the next fact, (1) follows by considering the preimages of half-open intervals, and for a proof of (2) see e.g.~\cite[Fact 2.10]{GanNIP}. 

\begin{fact}\label{fac: cont funcs on type spaces} Let $S$ be a topological space and $f:S \to \mathbb{R}$ a function.
\begin{enumerate}
	\item Assume $f$ is bounded and Borel. Then for every $\varepsilon > 0$ there exist $r_1, \ldots, r_n \in \mathbb{R}$ and Borel sets $B_1, \ldots, B_n$ such that $\{B_i\}_{i=1}^{n}$ partition $S$ and
\begin{equation*}
\sup_{a \in S} |f(a) - \sum_{i=1}^{n} r_i \mathbf{1}_{B_i}(a)| < \varepsilon. 
\end{equation*} 

\item Assume $S$ is a Stone space and $f$ is continuous. Then for every $\varepsilon >0$ there exists clopen sets $C_1, \ldots,C_n \subseteq S$ and $r_1,\ldots,r_n \in \mathbb{R}$ such that
\begin{equation*}
\sup_{a \in S} |f(a) - \sum_{i=1}^{n} r_i \mathbf{1}_{C_i}(a)| < \varepsilon.
\end{equation*}

\end{enumerate}

\end{fact}

\subsection{Keisler measures} \label{sec: Keisl meas prelim}

For any $A \subseteq \mathcal{U}$, a \emph{Keisler measure} over $A$ in variables $x$ is a finitely additive probability measure on $\mathcal{L}_{x}(A)$. We denote the space of Keisler measures over $A$ (in variables $x$) as $\mathfrak{M}_{x}(A)$. Every $\mu \in \mathfrak{M}_{x}(A)$ extends  uniquely to a regular Borel probability measure $\tilde{\mu}$ on the space $S_{x}(A)$, we will routinely use this correspondence. 
If $A \subseteq   B \subseteq \cU$, then there is an obvious restriction map $r_{0}: \mathfrak{M}_{x}(B) \to \mathfrak{M}_{x}(A)$ and we denote $r_{0}(\mu)$ simply as $\mu|_{A}$. Conversely, every $\mu \in \mathfrak{M}_{x}(A)$ admits an extension to some $\mu' \in \mathfrak{M}_{x}(B)$ (not necessarily a unique one). 

\begin{definition} Let $B \subseteq \mathcal{U}$ and $\mu \in \mathfrak{M}_{x}(\mathcal{U})$. We say that $\mu$ is:
\begin{enumerate}
\item \emph{invariant over $B$} if for any formula $\varphi(x,y) \in \mathcal{L}_{x,y}(B)$ and elements $a, b \in \mathcal{U}^{y}$ such that $a \equiv_{B} b$ we have $\mu \left(\varphi(x,b) \right) = \mu \left(\varphi(x,a) \right)$;
\item \emph{finitely satisfiable in $B$} if for any formula $\varphi(x) \in \mathcal{L}_{x}(\mathcal{U})$ such that $\mu(\varphi(x)) > 0$, there exists some $b \in B$ such that $\models \varphi(b)$. 
\end{enumerate} 
We let $\mathfrak{M}_{x}^{\fs}(\mathcal{U},B)$ (respectively, $\mathfrak{M}_{x}^{\inv}(\mathcal{U},B)$) denote the closed set of Keisler measures in $\mathfrak{M}_{x}(\mathcal{U})$ that are finitely satisfiable in $B$ (respectively, invariant over $B$).
\end{definition} 
Just as with types, we let $\mathfrak{M}^{\dagger}_{x}(\mathcal{U},B)$ mean ``$\mathfrak{M}^{\fs}_x(\mathcal{U},B)$ or $\mathfrak{M}_{x}^{\inv}(\mathcal{U},B)$". 
The \textit{support of $\mu \in \mathfrak{M}_{x}(B)$} is the non-empty closed set of types $\sup(\mu) = \{p \in S_{x}(B):$ for any $\varphi(x) \in p$, $\mu(\varphi(x)) > 0\}$. 
Given $\bar{p} = (p_1, \ldots, p_n)$ with $p_i \in S_x(A)$, we let $\Av(\bar{p}) \in \mathfrak{M}_x(A)$ be defined by $\Av(\bar{p})(\varphi(x)) := \frac{|\{i \in [n] : \varphi(x) \in p_i\}|}{n}$, and given $\bar{a} = (a_1, \ldots, a_n) \in \cU^x$, we let $\Av(\bar{a}) := \Av \left( \tp(a_1/\cU), \ldots, \tp(a_n/\cU)  \right)$. We refer to e.g.~\cite[Section 2]{ChGan} for a more detailed discussion of the aforementioned notions.

\begin{definition}\label{def: meas on a set of types} Let $\mathbb{X} \subseteq S_{x}(\mathcal{U})$. We let $\mathfrak{M}(\mathbb{X}):= \{ \mu \in \mathfrak{M}_{x}(\mathcal{U}): \sup(\mu) \subseteq \mathbb{X}\}$ be the set of Keisler measures supported on $\mathbb{X} $. If $\mathbb{X}$ is a closed subset of $S_{x}(\mathcal{U})$, we let $\mathcal{M}(\mathbb{X})$ denote the set of regular Borel probability measures on $\mathbb{X}$, with the topology on $\mathbb{X}$ induced from $S_{x}(\mathcal{U})$. When we consider $\mathcal{M}(\mathbb{X})$ as a topological space, we will always consider it with the weak--$^*$ topology.
\end{definition} 
\noindent The space of Keisler measures $\mathfrak{M}_{x}(A)$ is a (closed convex) subset of a real locally convex topological vector space of bounded charges on $\mathcal{L}_{x}(A)$ (see e.g.~\cite{Charges} for the details). 

%
%

\begin{lemma}\label{closed} Assume that $\mathbb{X}$ is a closed subset of $S_{x}(\mathcal{U})$. Then $\mathfrak{M}(\mathbb{X})$ is a closed convex subset of $\mathfrak{M}_{x}(\mathcal{U})$.
\end{lemma} 

\begin{proof} Suppose $\mathfrak{M}(\mathbb{X})$ is not closed, then $\lim_{i \in I} \mu_{i} = \mu$ for some $\mu \not \in \mathfrak{M}(\mathbb{X})$ and some net $(\mu_i)_{i \in I}$ with $\mu_i \in \mathfrak{M}(\mathbb{X})$. 
Then there exists a type $p \in \sup(\mu) \setminus \mathbb{X}$. Since $\mathbb{X}$ is closed, the set $U := S_{x}(\mathcal{U}) \backslash \mathbb{X}$ is open. Hence $U = \bigcup_{j \in J} [\varphi_{j}(x)]$ for some set of formulas $\{\varphi_j\}_{j \in J}$ and there is some $j \in J$ such that $\varphi_{j}(x) \in p$. Then $[\varphi_{j}(x)] \cap \mathbb{X} = \emptyset$ and $\mu(\varphi_{j}(x)) > 0$ (since $p \in \sup(\mu)$). Therefore
$\lim_{i \in I} \mu_i (\varphi_{j}(x)) = \lim_{i \in I} 0 = 0 < \mu(\varphi_{j}(x))$, 
a contradiction. 

The space $\mathfrak{M}(\mathbb{X})$ is convex since if $r,s \in \mathbb{R}_{>0}$, $r + s =1$, and $\mu,\nu \in \mathfrak{M}(\mathbb{X})$, then
$\sup(r\mu + s\nu) = \sup(\mu) \cup \sup(\nu) \subseteq \mathbb{X}$.
\end{proof} 

In the later sections, we will need to discuss maps from the space of Keisler measures to other spaces of measures. The following definition is an appropriate notion of an isomorphism in this context (and will be denoted by $\cong$). 

\begin{definition}\label{def: affine homeo} Let $V_1,V_2$ be two locally convex topological vector spaces. Suppose that $C_1$ and $C_2$ are closed convex subsets of $V_1$ and $V_2$, respectively. A map $f:C_1 \to C_2$ is an \textit{affine homeomorphism} if $f$ is a homeomorphism from $C_1$ to $C_2$ (with the induced topologies) and for any $a_1, \ldots, a_n \in C_1$ and $r_1, \ldots, r_n \in \mathbb{R}_{ \geq 0} $ with $\sum_{i=1}^{n} r_i = 1$ we have 
\begin{equation*} 
f \left(\sum_{i=1}^{n} r_i a_i \right) = \sum_{i=1}^{n} r_i f(a_i). 
\end{equation*} 
\end{definition} 

\begin{definition} Let $A$ be a subset of a locally convex topological vector space, $V$ and let $b \in V$. We say that $b$ is \emph{extreme in $A$} (or an \emph{extreme point of $A$}) if $b \in A$ and $b$ cannot be written as $r c_1 + (1-r)c_2$ for $c_1,c_2 \in A$ where $c_1 \neq c_2$ and $r \in (0,1)$. We let $\ex(A) := \{ c \in A: c \text{ is extreme in } A\}$.
\end{definition} 

\begin{fact}[Krein-Milman theorem] Let $A$ be a convex compact subset of a locally convex topological vector space $V$. Then the convex hull of $\ex(A)$ is a dense subset of $A$.  
\end{fact}

\begin{proposition}\label{closed:homeomorphism} Let $\mathbb{X} \subseteq S_{x}(\mathcal{U})$ be a closed set. Then  there exists an affine homeomorphism $\gamma : \mathfrak{M}(\mathbb{X}) \to \mathcal{M}(\mathbb{X})$ such that for any $\varphi(x) \in \mathcal{L}_{x}(\mathcal{U})$ and $\mu \in \mathfrak{M}(\mathbb{X})$, 
\begin{equation*} \mu(\varphi(x)) = \gamma(\mu)([\varphi(x)] \cap \mathbb{X}). 
\end{equation*}
Moreover, $\sup(\mu) = \supp(\gamma(\mu))$. 
\end{proposition} 

\begin{proof} This follows directly from the fact that every Keisler measure $\mu$ in $\mathfrak{M}_{x}(\mathcal{U})$ extends uniquely to a regular Borel probability measure $\tilde{\mu}$ on $S_{x}(\mathcal{U})$. We let $\gamma(\mu) := \tilde{\mu}\restriction_{\mathbb{X}}$, i.e.~the restriction of the measure $\tilde{\mu}$ to the collection of Borel subsets of $\mathbb{X}$. See e.g.~\cite[Page 99]{Sibook} for the details.
\end{proof}

\noindent For a proof of the following fact, see \cite[Lemma 2.10]{ChGan}. 
\begin{fact}\label{sup:inv}
\begin{enumerate}
\item $\mu \in \mathfrak{M}_{x}^{\fs}(\mathcal{U},M)$ if and only if $p \in S_{x}^{\fs}(\mathcal{U},M)$ for all $p \in \sup(\mu)$. 
	\item (T is NIP) $\mu \in \mathfrak{M}_{x}^{\inv}(\mathcal{U},M)$ if and only if  $p \in S_{x}^{\inv}(\mathcal{U},M)$ for every $p \in \sup(\mu)$. 
\end{enumerate}

\end{fact} 

\noindent Combining Proposition \ref{closed:homeomorphism} and Fact \ref{sup:inv} we have the following.
\begin{corollary}\label{cor:homeomorphism} 
\begin{enumerate} 
\item If $T$ is any theory, then $\mathfrak{M}^{\fs}_{x}(\mathcal{U},M) = \mathfrak{M}(S^{\fs}_{x}(\mathcal{U},M)) \cong \mathcal{M}(S^{\fs}_{x}(\mathcal{U},M))$. 
\item If $T$ is NIP, then  $\mathfrak{M}^{\inv}_{x}(\mathcal{U},M) = \mathfrak{M}(S^{\inv}_{x}(\mathcal{U},M)) \cong \mathcal{M}(S^{\inv}_{x}(\mathcal{U},M))$. 
\end{enumerate} 
\end{corollary}
\begin{remark}
	It is not true that $\mathfrak{M}(S^{\inv}_{x}(\mathcal{U},M)) = \mathfrak{M}_{x}^{\inv}(\mathcal{U},M)$ in an arbitrary theory, see \cite[Lemma 2.10(4)]{ChGan}.
\end{remark}

\begin{lemma}\label{lem: limit}  For any $\mu \in \mathfrak{M}_{x}(\mathcal{U})$, there exists a net of measures $(\nu_{j})_{j \in J}$ in $\mathfrak{M}_{x}(\mathcal{U})$ such that: 
\begin{enumerate}
\item for each $j \in J$, $\nu_j = \Av(\overline{p}_j)$ for some $\overline{p}_j = (p_{j_1}, \ldots, p_{j_m})$ with $p_{j_i} \in \sup(\mu)$;
\item $\lim_{j \in J} \nu_j = \mu$. 
\end{enumerate}  
Moreover, if $\mu$ is finitely satisfiable in $M \preceq \cU$, then we can take $\nu_j$ of the form $\Av(\bar{a}_j)$ for some $\bar{a}_j \in \left( M^x \right)^{<\omega}$.
\end{lemma} 

\begin{proof} Consider a basic open subset $O$ of $\mathfrak{M}_{x}(\mathcal{U})$, of the form
\begin{equation*} O  = \bigcap_{i=1}^{n}\left\{\nu \in \mathfrak{M}_{x}(\mathcal{U}) : r_i < \nu(\theta_i(x)) < s_i \right\}. 
\end{equation*} 
Suppose that $\mu \in O$. Let $\mathcal{B}$ be the (finite) Boolean algebra generated by the sets $\{\theta_1(x),\ldots,\theta_n(x)\}$, let $\{\sigma_{1}(x), \ldots, \sigma_{m}(x)\}$ be the set of its atoms. For each atom $\sigma_{i}(x)$ such that $\mu(\sigma_i(x)) > 0$, there exists some $p_i \in \sup(\mu)$ such that $\sigma_{i}(x) \in p_i$. Consider the measure
\begin{equation*} 
\lambda := \sum_{\{i \in [n]: \mu(\sigma_i(x)) > 0\}} \mu(\sigma_i(x)) \delta_{p_i}. 
\end{equation*} 
Then $\lambda (\theta_i(x)) = \mu(\theta_i(x))$ for every $i \in [n]$, hence $\lambda \in O$. We can choose a sufficiently large $t \in \mathbb{N}$ and $s_i \in \mathbb{N}$ so that $\frac{s_i}{t}$ is sufficiently close to $\mu(\sigma_i(x))$, so that 
 $\nu_{O} := \sum_{\{i \in [n]: \mu(\sigma_i(x)) > 0\}} \frac{s_i}{t} \delta_{p_i} \in O$ (taking $\bar{p}_O$ to be the tuple of types of length $t$ with $p_i$ repeated $s_i$ times, we see that $\nu_O = \Av(\bar{p}_O)$). Then we can take the net $(\nu_O)_{\mu \in O}$. 
 
 And if $\mu$ is finitely satisfiable in $M$ and $\mu(\sigma_i(x)) >0$, then $\models \sigma_i(a_i)$ for some $a_i \in M^x$, and we can take $p_i := \tp(a_i/\cU)$ (see \cite[Proposition 2.11]{ChGan}).
\end{proof}

\subsection{Definable convolution in NIP groups}\label{sec: def conv NIP prelims}
In this section, we assume that $T$ is an $\cL$-theory expanding a group, $\mathcal{G}$ denotes a sufficiently saturated model of $T$, and $G$ denotes a small elementary submodel; $x,y, \ldots$ denote singleton variables; and for any $\varphi(x) \in \cL_x(\mathcal{G})$, we let $\varphi'(x,y) := \varphi(x \cdot y)$.

\begin{definition}\label{def: def conv}[T is NIP] Suppose that $\mu, \nu \in \mathfrak{M}^{\inv}_{x}(\mathcal{G},G)$. Then we define $\mu * \nu$ to be the unique Keisler measure in $\mathfrak{M}^{\inv}_{x}(\mathcal{G},G)$ such that for any formula $\varphi(x) \in \mathcal{L}_{x}(\mathcal{G})$,
\begin{equation*} \mu * \nu(\varphi(x)) = \mu_{x} \otimes \nu_{y}(\varphi(x \cdot y)) = \int_{S_{y}(G')} F_{\mu,G'}^{\varphi'} d(\nu_{G'}),
\end{equation*} 
where $G'$ is a small model contains $G$ and all parameters from $\varphi$, $F_{\mu,G'}^{\varphi'}:S_{y}(G') \to [0,1]$ is given by 
$F_{\mu,G'}^{\varphi'}(q) = \mu(\varphi(x \cdot b))$ for some (equivalently, any) $b \in \mathcal{G}$ with $b \models q$, and $\nu_{G'}$ is the regular Borel probability measure on $S_{y}(G')$ corresponding to the  Keisler measure $\nu|_{G'}$. We will routinely suppress notation and write this integral as $\int_{S_{y}(G')} F_{\mu}^{\varphi'} d\nu$.
\end{definition} 

\begin{remark}
	This integral is well-defined since invariant measures in NIP  are Borel-definable, so the maps which are being integrated are measurable, and does not depend on the choice of $G'$. For more details about definable convolution and its basic properties we refer the reader to \cite[Section 3.2]{ChGan}, in particular \cite[Proposition 3.14]{ChGan} will be used freely throughout the article. 
\end{remark}
\noindent The following is well-known, see e.g.~\cite[Fact 3.11]{ChGan}.
\begin{fact}\label{fac: prod of types Ellis}
	Both $(S^{\inv}_{x}(\mathcal{G},G),*)$ and $(S_{x}^{\fs}(\mathcal{G},G),*)$ are left continuous (i.e.~$p \mapsto p \ast q$ is a continuous map for every $q$) compact Hausdorff semigroups.
\end{fact}

\noindent The next fact is from \cite[Proposition 6.2(3) + Proposition 6.4]{ChGan}.
\begin{fact}\label{fac: conv is a semigroup in NIP}[$T$ is NIP] Both $(\mathfrak{M}^{\inv}_{x}(\mathcal{G},G),*)$ and $(\mathfrak{M}_{x}^{\fs}(\mathcal{G},G),*)$ are left continuous (i.e.~$\mu \mapsto \mu \ast \nu$ is a continuous map for every $\nu$) compact Hausdorff semigroups.

Moreover, for any fixed $\nu$ and $\varphi(x) \in \mathcal{L}_x(\mathcal{G})$, the map $\mu \mapsto \left( \mu \ast \nu  \right)(\varphi(x)) \in [0,1]$ is continuous.
\end{fact} 

\noindent We also have right continuity when multiplying by a \emph{definable} measure (but not in general).
\begin{lemma}\label{lem: right cont for def meas}
	If $\nu \in \mathfrak{M}^{\inv}_{x}(\mathcal{G},G)$ is a definable measure, then the map $\mu \mapsto \nu \ast \mu$ from $\mathfrak{M}^{\inv}_{x}(\mathcal{G},G)$ to $\mathfrak{M}^{\inv}_{x}(\mathcal{G},G)$ is continuous.
\end{lemma}

\begin{proof} Let $O$ be a basic open subset of $\mathfrak{M}^{\inv}_{x}(\mathcal{G},G)$, that is 
\begin{equation*} 
O = \bigcap_{i=1}^{n} \left\{\mu \in \mathfrak{M}^{\inv}_{x}(\mathcal{G},G): r_i < \mu(\varphi_i(x)) < s_i\right\}
\end{equation*} 
for some $r_i, s_i \in  \mathbb{R}$ and $\varphi_i(x) \in \mathcal{L}_{x}(\mathcal{G})$.
We have
\begin{gather*}
(\nu * - )^{-1}(O) = \bigcap_{i=1}^{n} \left\{\mu \in \mathfrak{M}^{\inv}_{x}(\mathcal{G},G): r_i < \left( \nu * \mu \right) (\varphi_i(x)) < s_i\right\} \\
= \bigcap_{i=1}^{n} \left\{\mu \in \mathfrak{M}^{\inv}_{x}(\mathcal{G},G): r_i < \left( \nu_{z} \otimes \mu_{x} \right) \left(\varphi_i(z \cdot x) \right) < s_i\right\} \\
= \bigcap_{i=1}^{n} \left( \big(\left(\nu_{z} \otimes - \right)(\varphi_i(z \cdot x)) \big)^{-1}(r_i,s_i) \right).
\end{gather*} 
where $\nu_z$ is simply $\nu_x$ with change of variables to $z$ and $(r_i,s_i)$ is an open subinterval of $[0,1]$. By e.g.~\cite[Lemma 5.4]{CGH}, the map $\mu_x \in  \mathfrak{M}_{x}(\mathcal{G}) \mapsto \left(\nu_z \otimes \mu_x \right) \left( \varphi_i(z \cdot x) \right) \in [0,1]$ is continuous, so its restriction to $\mathfrak{M}^{\inv}_{x}(\mathcal{G},G)$ remains continuous. Thus $O$ is open, as the intersection of finitely many open sets. 
\end{proof} 

\begin{definition}
	A measure $\mu \in \mathfrak{M}^{\inv}_{x}(\mathcal{G},G)$ is \emph{idempotent} if $\mu * \mu = \mu$.  
\end{definition}

 The following simple observation will be frequently used in computations. 

\begin{fact}\label{fact:monst} Let $\mu \in \mathfrak{M}^{\inv}_{x}(\mathcal{G},G)$ and $f: S_{x}(G) \to \mathbb{R}$ be a bounded Borel function. Let $r: S_{x}(\mathcal{G}) \to S_{x}(G), p \mapsto p|_G$ be the restriction map. Then 
\begin{equation*} 
\int_{S_{x}(G)} f d\mu_{G} = \int_{S_{x}(\mathcal{G})} (f \circ r) d\mu. 
\end{equation*}
\end{fact} 

\subsection{Some facts from Ellis semigroup theory} \label{sec: Ellis theory prelims}
\begin{definition} Suppose that $(X,*)$ is a semigroup. A non-empty subset $I$ of $X$ is a \emph{left ideal} if $XI = \{x \ast i : x \in X, i \in I \} \subseteq I$. We say that $I$ is a \emph{minimal left ideal} if $I$ does not properly contain any other left ideal. 
\end{definition} 

\noindent The next fact summarizes the results that we will need from  the theory of Ellis semigroups. See \cite[Proposition 4.2]{Ellis} and \cite[Proposition 2.4]{Glasner}.

\begin{fact}\label{fact:Ellis} Suppose that $X$ is a compact Hausdorff space and $(X,*)$ is a left continuous semigroup, i.e.~for each $q \in X$, the map $-*q: X \to X$ is continuous.  Then there exists a minimal left ideal $I$, and any minimal left ideal is closed. We let $\id(I) = \left\{u \in I: u^{2} = u\right\}$ be the set of idempotents in $I$. 
\begin{enumerate}
\item $\id(I)$ is non-empty. 
\item For every $p \in I$ and $u \in \id(I)$, $p * u = p$.
\item For every $u \in \id(I)$, $u * I = \{ u *p : p \in I\} = \{p \in I: u *p = p\}$ is a subgroup of $I$ with identity element $u$. For every $u' \in \id(I)$, the map $\rho_{u,u'} := (u' * -)|_{u*I}$ is a group isomorphism from $u*I$ to $u'*I$. In view of this, we refer to $u \ast I$ as \emph{the ideal group}.
\item $I = \bigcup \{u * I: u \in \id(I)\}$, where the sets in the union are pairwise disjoint, and each set $u \cdot I$ is a subgroup of $I$ with identity $u$.
\item For any $q \in X$, $I *q$ is a minimal left ideal; and if $p \in I$, then $X * p = I$. 
\item Let $J$ be another minimal left ideal of $X$ and $u \in \id(I)$. Then there exists a unique idempotent $u' \in \id(J)$ such that $u * u' = u'$ and $u' * u = u$. The map $\rho_{I,J}:=(-*u')|_{I}$ is a homeomorphism from $I$ to $J$ (with the induced topologies) mapping $u \ast I$ to $u' \ast J$. 
\end{enumerate}
\end{fact} 

\noindent The following is a celebrated theorem of Ellis \cite[Theorems 1 and 2]{Ellis2} (see also \cite[Corollary 5.2]{lawson1974joint}).

\begin{fact}[Ellis' Joint Continuity Theorem]\label{Ellis:2}
\begin{enumerate}
\item Let $G$ be a locally compact Hausdorff semitopological group (i.e.~$G$ is equipped with a group structure such that the maps $x \mapsto y \cdot x$ and $x \mapsto x \cdot y$ from $G$ into $G$ are continuous for any fixed $y \in G$), and let $X$ be a locally compact Hausdorff topological space. Then every separately continuous action of $G$ on $X$ is (jointly) continuous.
	\item If $G$ is a locally compact Hausdorff semitopological group, then $G$ is a topological group.
\end{enumerate} 
\end{fact} 

\subsection{Some facts from Choquet theory}\label{sec: Choquet theory prelims}

We recall some notions and facts from Choquet theory for \emph{not necessarily metrizable}  compact Hausdorff spaces (we use \cite{phelps2001lectures} as a general reference). 	Let $E$ be a locally convex real topological vector space. The following generalizes the usual notion of a simplex in $\mathbb{R}^n$ to the infinite dimensional context.

\begin{definition}\cite[Section 10]{phelps2001lectures}
	\begin{enumerate}
		\item A set $P \subseteq E$ is a \emph{convex cone} if $P+P \subseteq P$ and $\lambda P \subseteq P$ for every scalar $\lambda > 0$ in $ \mathbb{R}$.
		\item A set $X \subseteq P$ is the \emph{base} of a  convex cone $P$ if for every $y \in P$ there exists a unique scalar $\lambda \geq 0$ in $\mathbb{R}$ and $x \in X$ such that $y = \lambda x$ (not all convex cones have a base).
		\item A convex cone $P$  in $E$ induces a translation invariant partial ordering on $E$: $x \geq y$ if and only if $x-y \in P$. When $P$ admits a base, then $P \cap \left(-P \right) = \{0\}$, hence $x \geq y \land y \geq x \implies x = y$.
		\item 	A non-empty compact convex set $X \subseteq E$ is a \emph{Choquet simplex}, or just \emph{simplex}, if $X$ is the base of a convex cone $P \subseteq E$ such that $P$ is a lattice with respect to the ordering induced by $P$. That is, for every $x,y \in P$ there exists a greatest lower bound $z \in P$ (i.e.~$z \leq x, z \leq y$ and for every $z' \in P$ with $z' \leq x, z'\leq y$, $z' \leq z$). The greatest lower bound $z$ of $x,y$ is unique and denoted $x \land y$.
	\end{enumerate}
\end{definition}


%
%

We could not find a direct quote for the following fact, so we provide a short  argument combining several standard results in the literature.
\begin{fact}\label{fac: inv measures simplex}
	Let $S$ be a compact Hausdorff space and $\mathcal{T}$ a family of continuous functions from $S$ into $S$.
Then the set of all regular $\mathcal{T}$-invariant (that is, $\mu \left( T^{-1}(A) \right) = \mu \left(A \right)$ for every Borel $A \subseteq S$ and $T \in \mathcal{T}$) Borel probability measures on $S$, denoted $\mathcal{M}_{\mathcal{T}}(S)$,  is a Choquet simplex (assuming it is non-empty).	
\end{fact}
\begin{proof}
By the Riesz representation theorem, we can view the set $\mathcal{M}^{+}(S)$  of all regular Borel non-negative finite measures on $S$ as a subset of $C(S)^{*}$, the dual (real topological vector) space of the topological vector space of continuous functions on $S$, with the weak-$^*$ topology.
Let $\mathcal{M}_{\mathcal{T}}(S)$ (respectively, $\mathcal{M}^{+}_{\mathcal{T}}(S)$) be the set  of regular Borel $\mathcal{T}$-invariant probability (respectively, finite non-negative) measures on $S$. 
Then $\mathcal{M}_{\mathcal{T}}(S) \subseteq \mathcal{M}^{+}_{\mathcal{T}}(S) \subseteq \mathcal{M}^{+}(S)$ are compact convex subsets (by Borel measurability of the maps in $\mathcal{T}$, see \cite[page 76]{phelps2001lectures}).
Moreover, $\mathcal{M}^{+}_{\mathcal{T}}(S)$ is a convex cone with the base $\mathcal{M}_{\mathcal{T}}(S)$. It is well-known that $\mathcal{M}^{+}(S)$ forms a lattice: for $\mu,\nu \in \mathcal{M}^{+}(S)$, their greatest lower bound $\mu \land \nu \in \mathcal{M}^{+}(S)$ can be defined via $\left( \mu \land \nu \right)(A) = \inf_{B \in \mathcal{S}, B \subseteq A} \left\{ \mu \left(B\right) + \nu \left(A \setminus B \right) \right\}$ (see e.g.~\cite[page 111]{dales2016banach};  it is easy to verify from this definition that if $\mu,\nu$ are regular, then $\mu \land \nu$ is also regular). Finally, \cite[Proposition 12.3]{phelps2001lectures} shows that if $\mu,\nu \in  \mathcal{M}^{+}(S)$ are $\mathcal{T}$-invariant,  then  $ \mu \land \nu $ is also $\mathcal{T}$-invariant (using an equivalent definition of the lower bound in terms of the Radon-Nikodym derivative). Hence $\mathcal{M}^{+}_{\mathcal{T}}(S)$ is a lattice, and so $\mathcal{M}_{\mathcal{T}}(S)$ is a Choquet simplex.
\end{proof}

\begin{definition}(see \cite[Section 11]{phelps2001lectures} or \cite[Chapter 2, \S 4]{alfsen2012compact})\label{def: Bauer simplex}
	A compact convex set $X \subseteq E$ is a \emph{Bauer simplex} if $X$ is a Choquet simplex and $\ex(X)$ is closed.
\end{definition}

\begin{definition}
	A point $x \in E$ is the \emph{barycenter}  of a regular Borel probability measure $\mu$ on $X$ if $f(x) = \mu(f) := \int_{X} f d\mu$ for any continuous linear function $f: E \to \mathbb{R}$.
\end{definition}

\begin{remark}\label{rem: aff hom simplex}
	Both the property of being a Choquet simplex and a Bauer simplex are preserved under affine homeomorphisms (see e.g.~\cite[pages 52-53]{phelps2001lectures}).
\end{remark}

\begin{fact}\label{fac: props of Bauer simplex}
	\begin{enumerate}
		\item \cite[Proposition 11.1]{phelps2001lectures} $X$ is a Bauer simplex if and only if the map sending a regular Borel probability measure $\mu$ on $\overline{\ex(X)}$ (the closure of the extreme points) to its barycenter is an affine homeomorphism of $\mathcal{M}\left(\overline{\ex(X) }\right)$ and $X$ (hence a posteriori of $\mathcal{M}\left(\ex(X) \right)$ and $X$).
		\item \cite[Corollary II.4.4]{alfsen2012compact} Up to affine homeomorphisms, Bauer simplices are exactly the sets of the form $\mathcal{M}(X)$ for $X$ a compact Hausdorff space (where $\ex \left( \mathcal{M}(X) \right) = \left\{\delta_x : x \in X \right\}$).
	\end{enumerate}
\end{fact}

%

\section{Definable convolution on $\mathcal{G}$ and convolution on $\mathcal{G}/\mathcal{G}^{00}$}\label{sec: def conv vs conv}

Throughout the rest of the paper, $T$ is a complete NIP theory expanding a group, $\mathcal{G}$ is a monster model of $T$, $G$ is a small elementary submodel of $\mathcal{G}$, $x,y, \ldots$ denote singleton variables, and for any $\varphi(x) \in \cL_x(\mathcal{G})$, $\varphi'(x,y) = \varphi(x \cdot y)$.
We define and study a natural push-forward map from $\mathfrak{M}_{x}(\mathcal{G})$ to $\mathcal{M}(\mathcal{G}/\mathcal{G}^{00})$. We demonstrate that this map is a  homomorphism from the semigroup $(\mathfrak{M}_{x}^{\inv}(\mathcal{G},G),*)$ of invariant Keisler measures with definable convolution onto the semigroup $(\mathcal{M}(\mathcal{G}/\mathcal{G}^{00}), \star)$ of regular Borel probability measures on the compact group $\mathcal{G}/\mathcal{G}^{00}$ with classical convolution. In particular, the image of an idempotent invariant Keisler measure on $\mathcal{G}$ is an idempotent measure on the compact group $\mathcal{G}/\mathcal{G}^{00}$. The proofs of these theorems are primarily analytic, and the NIP assumption is used to ensure that $\mathcal{G}^{00}$ exists and definable convolution is well-defined. We begin by recalling some properties of $\mathcal{G}/\mathcal{G}^{00}$ and define the corresponding push-forward map.

\begin{fact}\label{group:fact} Suppose that $T$ is NIP. 
\begin{enumerate}[$(i)$]
\item There exists a smallest type-definable subgroup of $\mathcal{G}$ of bounded index, denoted $\mathcal{G}^{00}$. Moreover, $\mathcal{G}^{00}$ is a normal subgroup of $\mathcal{G}$ type-definable over $\emptyset$. Let $\pi:\mathcal{G} \to \mathcal{G}/\mathcal{G}^{00}$ be the  quotient map, i.e.~$\pi(a) = a\mathcal{G}^{00}$.
\item $\mathcal{G}/\mathcal{G}^{00}$ is a compact group with the \emph{logic topology}: a subset $B$ of $\mathcal{G}/\mathcal{G}^{00}$ is closed if and only if $\pi^{-1}(B)$ is type-definable over some/any small submodel of $\mathcal{G}$. 
\item The map $\pi: \mathcal{G} \to \mathcal{G}/\mathcal{G}^{00} $ induces a continuous map $\hat{\pi}:S_{x}(\mathcal{G}) \to \mathcal{G}/\mathcal{G}^{00}$ via $\hat{\pi}(q) := \pi(a)$,  where $a \models q|_G$ and $G$ is some/any elementary submodel of $\mathcal{G}$. Therefore, we can consider the push-forward $\hat{\pi}_{*}: \mathcal{M}(S_{x}(\mathcal{G})) \to \mathcal{M}(\mathcal{G}/\mathcal{G}^{00})$. By Proposition \ref{closed:homeomorphism}, $\mathfrak{M}_{x}(\mathcal{G})$ is affinely homeomorphic to $\mathcal{M}(S_{x}(\mathcal{G}))$ and so (formally) we let $\pi_{*}: \mathfrak{M}_{x}(\mathcal{G}) \to \mathcal{M}(\mathcal{G}/\mathcal{G}^{00})$ be the composition of $\hat{\pi}_{*}$ and this homeomorphism. We will primarily work with $\pi_{*}$, and usually identify $\hat{\pi}_{*}$ and $\pi_{*}$ without comment. 
\item  The map $\pi_{*}: \mathfrak{M}_{x}(\mathcal{G}) \to \mathcal{M}(\mathcal{G}/\mathcal{G}^{00})$ is continuous, affine, and surjective. 
\end{enumerate}
\end{fact} 

\begin{proof} $(i)$ is a theorem of Shelah \cite{shelah2008minimal}, and $(ii)$ is from \cite{pillay2004type} (see also \cite[Section 8]{Sibook}). 

$(iii)$ First, $\hat{\pi}$ is well-defined. Indeed, let $G_1, G_2 \prec \mathcal{G}$ be small elementary submodels, $q \in S_{x}(\mathcal{G})$, $a_i \models q|_{G_i}$ for $i \in \{1,2\}$. It suffices to show $\pi(a_1) = \pi(a_2)$. Let $U$ be an open subset of $\mathcal{G}/\mathcal{G}^{00}$ such that $\pi(a_1) \in U$, and we show that then also $\pi(a_2) \in U$. Since $U$ is open, $\pi^{-1}(U)$ is $\bigvee$-definable over both $G_1$ and $G_2$, let $\bigvee_{j \in I_i} \psi^i_j(x)$ be a definition of $\pi^{-1}(U)$ over $G_i$. Hence there is some $j_{1} \in I_1$ such that $\mathcal{U} \models \psi_{j_1}(a_{1})$, so $\psi_{j_1}(x) \in q$. As $\bigcup_{j \in I_1} [\psi^1_{j}(x)] = \bigcup_{j \in I_2} [\psi^2_{j}(x)]$ (see Definition \ref{def: type def sets}), there exists some $j_{2} \in I_2$ so that $\psi_{j_{2}}(x) \in q$. Now
\begin{equation*} 
a_2 \in \psi^2_{j_{2}}(\mathcal{U})  \subseteq \bigcup_{j \in I_2} \psi^2_{j}(\mathcal{U}) = \pi^{-1}(U) \implies \pi(a_2) \in U. 
\end{equation*}
Since $\mathcal{G}/\mathcal{G}^{00}$ is Hausdorff and $\pi(a_1)$ and $\pi(a_2)$ are in the same open sets, we conclude that $\pi(a_1) = \pi(a_2)$.  

By the previous paragraph, $\hat{\pi} = f \circ r_{G}$ where $G$ is any small submodel, the map $r_G: S_{x}(\mathcal{G)} \to S_{x}(G)$ is the restriction map, and $f:S_{x}(G) \to \mathcal{G}/\mathcal{G}^{00}$ is defined via $f(q) = \pi(a)$, where $a \models q$. Both $f$ and $r_{G}$ are continuous maps and so $\hat{\pi}$ is a continuous map (the map $f$ is continuous by $(ii)$).

(iv) By Fact \ref{meas:facts}(iii),(iv) and Proposition \ref{closed:homeomorphism}.
\end{proof} 

\begin{definition} We let $\pi^{\fs}_{G,*} := \pi_{*}\restriction_{\mathfrak{M}_{x}^{\fs}(\mathcal{G},G)}$ and $\pi^{\inv}_{G,*} := \pi_{*} \restriction_{\mathfrak{M}_{x}^{\inv}(\mathcal{G},G)}$. We will typically write $\pi^{\inv}_{G,*}$ simply as $\pi^{\inv}_{*}$ when $G$ is clear from the context, and $\pi_{*}^{\dagger}$ to mean ``either $\pi_{*}^{\inv}$ or $\pi_{*}^{\fs}$".
\end{definition} 

\begin{remark}\label{rem: pi star is cont} Both $\pi_{*}^{\inv}$ and $\pi_{*}^{\fs}$ are continuous and affine since these maps are restrictions of $\pi_{*}$ to a closed convex subspace.
\end{remark} 

\begin{proposition}\label{surject} The map $\pi_{*}^{\dagger}:\mathfrak{M}_{x}^{\dagger}(\mathcal{G},G) \to \mathcal{M}(\mathcal{G}/\mathcal{G}^{00})$ is surjective.
\end{proposition} 

\begin{proof} Since $\mathfrak{M}_{x}^{\fs}(\mathcal{G},G) \subseteq \mathfrak{M}_{x}^{\inv}(\mathcal{G},G)$, it suffices to show that $\pi^{\fs}_{*}$ is surjective. Fix  $\nu \in \mathcal{M}(\mathcal{G}/\mathcal{G}^{00})$. By the Krein-Milman theorem, the convex hull of the extreme points of $\mathcal{M}(\mathcal{G}/\mathcal{G}^{00})$ is dense inside $\mathcal{M}(\mathcal{G}/\mathcal{G}^{00})$. The extreme points of $\mathcal{M}(\mathcal{G}/\mathcal{G}^{00})$ are the Dirac measures concentrating on the elements of $\mathcal{G}/\mathcal{G}^{00}$ (see e.g.~\cite[Example 8.16]{simon2011convexity}). Thus there exists a net $(\nu_{i})_{i \in I}$ of measures in $\mathcal{M}(\mathcal{G}/\mathcal{G}^{00})$ such that $\lim_{i \in I} \nu_i = \nu$ and for each $i \in I$, $\nu_{i} = \sum_{j = 1}^{n_{i}} r^i_{j} \delta_{b^{i}_{j}}$ for some $n_i \in \mathbb{N}$, $b^i_{j} \in \mathcal{G}/\mathcal{G}^{00}$ and $r^i_j \in \mathbb{R}_{>0}$ with  $\sum_{j = 1}^{n_i} r^i_j = 1$. Since the map $\pi$ is surjective, for each $b^i_j$ there exists some $a^i_j \in \mathcal{G}$ such that $\pi \left(a^i_j \right) = b^i_j$. Let $p^i_j \in S^{\fs}_x(\mathcal{G},G)$ be a global coheir of $\tp(a^i_{j}/G)$,  and let $\mu_{i} := \sum_{j = 1}^{n_i} r^i_{j} \delta_{p^i_{j}}$, then $\pi_{*}(\mu_{i}) = \nu_i$. Now $(\mu_{i})_{i \in I}$ is a net in the compact space $\mathfrak{M}_{x}^{\fs}(\mathcal{G},G)$, hence passing to a subnet we may assume that it converges, and let $\mu := \lim_{i \in I} \mu_i$. Then 
\begin{equation*}
\pi_{*}(\mu) = \pi_{*} \left(\lim_{i \in I} \mu_i \right) = \lim_{i \in I} \pi_{*}(\mu_i) = \lim_{i \in I} \nu_{i} = \nu,
\end{equation*} 
where the second equality follows from continuity of $\pi_{*}$. Hence $\pi^{\fs}_{*}$ is surjective. 
\end{proof} 

\begin{lemma}\label{basic:lem} Let $p,q \in S^{\inv}_{x}(\mathcal{G},G)$. Then:
\begin{enumerate}[$(i)$]
\item $\hat{\pi}(p) \cdot \hat{\pi}(q) = \hat{\pi}(p*q)$, 
\item $\pi_{*}(\delta_{p}) = \delta_{\hat{\pi}(p)}$,
\item $\pi_{*}(\delta_{p} * \delta_{q}) = \pi_{*}(\delta_{p}) \star \pi_{*}(\delta_{q})$. 
\end{enumerate} 
\end{lemma} 

\begin{proof} 
(i) Let $b \models q|_{G}$ and $a \models p|_{Gb}$. By defintion $\left( a \cdot b \right)\models p*q|_{G}$, hence 
\begin{equation*} 
\hat{\pi}(p * q) = \pi(a \cdot b) = \pi(a) \cdot \pi(b) = \hat{\pi}(p) \cdot \hat{\pi}(q). 
\end{equation*} 

\noindent (ii) Let $f:\mathcal{G}/\mathcal{G}^{00} \to \mathbb{R}$ be a continuous function. Then
\begin{equation*} \pi_{*}(\delta_{p})(f) = \int (f \circ \hat{\pi}) d\delta_{p} = f(\hat{\pi}(p)) = \int f d\delta_{\hat{\pi}(p)} = \delta_{\hat{\pi}(p)}(f).
\end{equation*} Since $\pi_{*}(\delta_{p})$ and $\delta_{\hat{\pi}(p)}$ agree on all continuous functions, by Fact \ref{meas:facts}(i) they belong to the same open sets in a Hausdorff space, hence $\pi_{*}(\delta_{p}) = \delta_{\hat{\pi}(p)}$.

\noindent (iii) We have
\begin{equation*} \pi_{*}(\delta_{p} * \delta_{q}) = \pi_{*}(\delta_{p *q}) = \delta_{\hat{\pi}(p *q)} = \delta_{\hat{\pi}(p) \cdot \hat{\pi}(q)} = \delta_{\hat{\pi}(p)} \star \delta_{\hat{\pi}(q)}.
\end{equation*} 
Here the first equality follows from \cite[Proposition 3.12]{ChGan}, the second and third equalities follows from $(ii)$ and $(i)$ respectively, and the last equality is by Remark \ref{remark:conv}. 
\end{proof} 

To show that $\pi^{\inv}_{*}$ is a homomorphism, we first observe some basic properties of the action of $\mathcal{G}$ on its space of types and in turn, on the space of continuous functions from $S_{x}(\mathcal{G})$ to $\mathbb{R}$. 

\begin{definition}\label{acts:naturally} Let $G$ be a model of $T$. For $a \in G$ and $p \in S_x(G)$, let $p \cdot a := \{\varphi(x\cdot a^{-1}):\varphi(x)\in p\} \in S_x(G)$ and $a \cdot p = \{\varphi(a^{-1}\cdot x):\varphi(x)\in p\} \in S_x(G)$. This defines right (respectively, left) action of $G$ on $S_x(G)$ by homeomorphisms.
\end{definition} 

\begin{lemma}\label{quot:triv} For any $a \in \mathcal{G}$ and $q \in S_{x}(\mathcal{G})$ we have $\pi(a) \cdot \hat{\pi}(p) = \hat{\pi}(a \cdot p)$ and $\hat{\pi}(p) \cdot \pi(a) =\hat{\pi}(p \cdot a)$.
\end{lemma} 

\begin{proof} We notice that 
\begin{equation*}  \hat{\pi}(p) \cdot \pi(a) =   \hat{\pi}(p) \cdot \hat{\pi}(\tp(a/\mathcal{G})) = \hat{\pi}(p* \tp(a/\mathcal{G}) ) = \hat{\pi}(p \cdot a),
\end{equation*}
where the second equality is by Lemma \ref{basic:lem}(i). The other computation is similar.
\end{proof} 

\begin{lemma}\label{cont:lemma} Let $G$ be any model of $T$. Let $h:S_{x}(G) \to \mathbb{R}$ be a function, $\left\{[\psi_{i}]\right\}_{i \in [n]}$ a partition of $S_{x}(G)$ with $\psi_i \in \cL_x(G)$, $\varepsilon \in \mathbb{R}_{>0}$ and $r_1,\ldots,r_n \in \mathbb{R}$  such that 
$ \sup_{q \in S_{x}(G)} \left \lvert h(q) - \sum_{i=1}^n r_i \mathbf{1}_{[\psi_{i}]}(q) \right \rvert < \varepsilon.
$
For $a \in G$, we define the functions $h \cdot a, a \cdot h: S_x(G) \to \mathbb{R}$ via $(h \cdot a)(p) = h(p \cdot a)$ and $(a \cdot h)(p) = h(a \cdot p)$. Then
\begin{gather*}
\sup_{q\in S_{x}(\mathcal{G})}|(h \cdot a)(q)-\sum_{i=1}^{n}r_{i}\mathbf{1}_{[\psi_{i}(x\cdot a)]} (q)|<\varepsilon \textrm{, and } \\
\sup_{q\in S_{x}(\mathcal{G})}|(a \cdot h)(q)-\sum_{i=1}^{n}r_{i}\mathbf{1}_{[\psi_{i}(a \cdot x)]}(q)|<\varepsilon.	
\end{gather*}
In particular, if $h$ is continuous, then $h \cdot a$ and $a \cdot h$ are both continuous maps from $S_{x}(\mathcal{G})$ to $\mathbb{R}$ (as uniform limits of continuous functions, using Fact \ref{fac: cont funcs on type spaces}(2)).
\end{lemma}
\begin{proof} We only prove the lemma for $h \cdot a$ (the case of $a \cdot h$ is similar). Assume the conclusion fails, then there exists some $q \in S_{x}(G)$ such that 
$|(h \cdot a)(q) - \sum_{i=1}^{n} r_i \mathbf{1}_{[\psi_{i}(x \cdot a)]} (q)| > \varepsilon$. Since $\{[\psi_{i}(x)]\}_{i \in [n]}$ is a partition, so is $\{[\psi_{i}(x \cdot a)]\}_{i \in [n]}$. For precisely one $k \in [n]$, we have that $\psi_{k}(x \cdot a) \in q$ and $\sum_{i=1}^{n} r_i \mathbf{1}_{[\psi_{i}(x\cdot a)]}(q) = r_k$. So $\psi_{k}(x\cdot a^{-1} \cdot a) \in q \cdot a$, hence $\psi_{k}(x) \in q \cdot a$. Since $\{[\psi_{i}(x)]\}_{i \in [n]}$ forms a partition, we have that $\sum_{i=1}^{n} r_{i}\mathbf{1}_{[\psi_{i}(x)]}(q \cdot a) = r_k$. Then 
$\varepsilon > |h(q \cdot a) - \sum_{i=1}^{n} r_i \mathbf{1}_{[\psi_{i}(x)]} (q \cdot a)| = |(h \cdot a) (q) - r_k| > \varepsilon$ by assumption, a contradiction. 
\end{proof}

\begin{remark} The previous lemma follows also from the more general observation that both the left and right action of $\mathcal{G}$ on $(\mathbb{R}^{S_{x}(\mathcal{G})},||\cdot||_{\infty})$ is by isometries, where $\mathbb{R}^{S_{x}(\mathcal{G})}$ is the space of all functions from $S_{x}(\mathcal{G})$ to $\mathbb{R}$ with the uniform norm.
\end{remark} 

\begin{theorem}\label{invar:conv} Suppose $\mu, \nu \in \mathfrak{M}^{\inv}_{x}(\mathcal{G},G)$. Then
 $\pi_{*}(\mu * \nu) = \pi_{*}(\mu) \star \pi_{*}(\nu)$. 
\end{theorem}
\begin{proof} It suffices to show that for any continuous function $f:\mathcal{G}/\mathcal{G}^{00} \to \mathbb{R}$ we have $\pi_{*}(\mu * \nu)(f) = \pi_{*}(\mu) \star \pi_{*}(\nu)(f)$. Fix a continuous $f: \mathcal{G}/\mathcal{G}^{00} \to \mathbb{R}$. Let $r: S_{x}(\mathcal{G}) \to S_{x}(G), p \mapsto p|_{G}$ be the restriction map.  
Fix $\varepsilon >0$. Then $f \circ \hat{\pi}$ is a continuous function from $S_{x}(\mathcal{G})$ to $\mathbb{R}$ (which factors through $S_{x}(G)$), so by Fact \ref{fac: cont funcs on type spaces}(2) there exists a partition $\{[\psi_{i}(x)]\}_{i \in [n]}$ of $S_{x}(\mathcal{G})$ with $\psi_i(x) \in \mathcal{L}_{x}(G)$ and $r_1, \ldots,r_n \in \mathbb{R}$ such that
\begin{equation*} 
\sup_{p \in S_{x}(\mathcal{G})} \left \lvert (f \circ \hat{\pi})(p) - \sum_{i=1}^{n} r_i \mathbf{1}_{[\psi_{i}(x)]}(p) \right \rvert < \varepsilon.
\end{equation*} 

\noindent We now have the following computation for  $\pi_{*}(\mu * \nu)(f)$:
\begin{gather*}
	\pi_{*}(\mu * \nu)(f) = \int_{\mathcal{G}/\mathcal{G}^{00}} f d\pi_{*}(\mu*\nu) = \int_{S_{x}(\mathcal{G})} (f \circ \hat{\pi}) d(\mu*\nu)\\
	\approx_{\varepsilon} \int_{S_{x}(\mathcal{G})} \left( \sum_{i=1}^{n} r_i \mathbf{1}_{[\psi_i(x)]}\right) d(\mu * \nu) = \sum_{i=1}^{n} r_i \big((\mu * \nu)(\psi_i(x))\big)  \\
	=\sum_{i=1}^{n} r_i \big((\mu_x \otimes \nu_y)(\psi_{i}(x \cdot y)\big) = \sum_{i=1}^{n}r_i \int_{S_{y}(G)} F_{\mu,G}^{\psi_i'} d(\nu_{G}) \\
	\overset{(*)}{=} \sum_{i=1}^{n} r_i\int_{S_{y}(\mathcal{G})} \left( F_{\mu,G}^{\psi_i'} \circ r \right) d\nu  = \int_{S_{y}(\mathcal{G})} \left( \left( \sum_{i=1}^{n}r_i F_{\mu,G}^{\psi_i'} \right)  \circ r \right) d\nu.
\end{gather*}
\noindent The equality $(*)$ is justified by Fact \ref{fact:monst}. 

 Next we will show that the convolution product $\left( \pi_{*}(\mu) \star \pi_{*}(\nu) \right)(f)$ in $\mathcal{M}(\mathcal{G}/\mathcal{G}^{00})$ is close to the final term in the above computation. Define $h: \mathcal{G}/\mathcal{G}^{00} \to \mathbb{R}$ via $h(\mathbf{a}) = \int_{\mathcal{G}/\mathcal{G}^{00}} f(x\cdot \mathbf{a}) d \pi_{*}(\mu)$. By Lemma \ref{cont:int}, $h$ is continuous. Fix $p \in S_{y}(\mathcal{G})$ and let $\mathbf{b} := \hat{\pi}(p) \in \mathcal{G}/\mathcal{G}^{00}$ and $b \models r(p) \in \mathcal{G}$. By definition, $\hat{\pi}(p) = \pi(b) = \mathbf{b}$.  By Lemmas \ref{basic:lem} and \ref{cont:lemma}, we have the following computation:
\begin{gather*}
	(h \circ \hat{\pi})(p) =  h(\mathbf{b}) =  \int_{\mathcal{G}/\mathcal{G}^{00}} f(x\cdot \mathbf{b}) d \pi_{*}(\mu) = \int_{q \in S_{x}(\mathcal{G})} f(\hat{\pi}(q) \cdot \mathbf{b})d\mu\\
	= \int_{q \in S_{x}(\mathcal{G})} f(\hat{\pi}(q) \cdot \pi(b))d\mu = \int_{q \in S_{x}(\mathcal{G})} f(\hat{\pi}(q \cdot b))d\mu = \int_{S_{x}(\mathcal{G})} \left((f \circ \hat{\pi}) \cdot b \right) d\mu\\
	\approx_{\varepsilon} \int_{S_{x}(\mathcal{G})} \sum_{i=1}^{n} r_i \mathbf{1}_{[\psi_{i}(x\cdot b)]} d\mu = \sum_{i=1}^{n} r_i \mu(\psi_i(x\cdot b)) = \left( \left( \sum_{i=1}^{n} r_i F_{\mu,G}^{\psi'_i} \right) \circ r \right)(p).
\end{gather*}
\noindent Since $p$ was arbitrary in $S_{y}(\mathcal{G})$, we conclude that
\begin{equation*} 
\sup_{p \in S_{y}(\mathcal{G})} \left \lvert (h \circ \hat{\pi})(p) - \left(\left( \sum_{i=1}^{n} r_i F_{\mu,G}^{\psi'} \right) \circ r \right)(p) \right \rvert < \varepsilon.
\end{equation*}
Therefore 
\begin{gather*} 
\left(\pi_{*}(\mu) \star \pi_{*}(\nu) \right)(f) = \int_{\mathcal{G}/\mathcal{G}^{00}} h d\pi_{*}(\nu) = \int_{S_{y}(\mathcal{G})} (h \circ \hat{\pi}) d\nu \\\approx_{\varepsilon} \int_{S_{y}(\mathcal{G})} \left( \left( \sum_{i=1}^{n} r_i F_{\mu,G}^{\psi'} \right) \circ r \right) d\nu \approx_{\varepsilon}	\pi_{*}(\mu * \nu)(f) . 
\end{gather*} 
Since $\varepsilon$ was arbitrary, we conclude that $\pi_{*}(\mu * \nu)(f) = \left( \pi_{*}(\mu) \star \pi_{*}(\nu) \right)(f)$. 
\end{proof} 

\begin{corollary}\label{cor:wendel} If $\mu \in \mathfrak{M}^{\inv}_{x}(\mathcal{G},G)$ and $\mu$ is idempotent, then $\pi_{*}(\mu)$ is an idempotent measure on $\mathcal{G}/\mathcal{G}^{00}$.
\end{corollary} 

\begin{proof} By Theorem \ref{invar:conv} we have  $\pi_{*}(\mu) \star \pi_{*}(\mu) = \pi_{*}( \mu * \mu) = \pi_{*}(\mu)$.
\end{proof} 

\begin{corollary}\label{cor: pushf idemp onto} Let $\lambda \in \mathcal{M}(\mathcal{G}/\mathcal{G}^{00})$ and assume that $\lambda$ is idempotent. Then there exists a measure $\nu \in \mathfrak{M}^{\fs}_{x}(\mathcal{G},G)$ such that $\pi_{*}(\nu) = \lambda$ and $\nu$ is idempotent. 
\end{corollary} 

\begin{proof} By Proposition \ref{surject}, the set $A := \{\eta \in \mathfrak{M}^{\fs}_{x}(\mathcal{G},G): \pi_{*}(\eta) = \lambda\}$ is non-empty. Since $\pi_{*}$ is continuous by Fact \ref{group:fact}(iv), $A$ is a closed subset of $\mathfrak{M}^{\fs}_{x}(\mathcal{G},G)$. And for any $\eta_1, \eta_2 \in A$ we have $\eta_1 \ast \eta_2 \in A$,  as $\pi_{*}(\eta_{1} * \eta_{2}) = \pi_{*}(\eta_{1}) \star \pi_{*}(\eta_{2}) = \lambda \star \lambda = \lambda$ by Theorem \ref{invar:conv}. Hence $(A,*)$ is a compact left-continuous semigroup (using Fact \ref{fac: conv is a semigroup in NIP}). By Fact \ref{fact:Ellis}, $(A,*)$ contains an idempotent.
\end{proof} 

\section{$\mathcal{G}^{00}$-invariant idempotent measures and type-definable subgroups}  \label{sec: G00 inv idempt class}

In this section we use the properties of the push-forward map established in Section  \ref{sec: def conv vs conv} to prove that if $\mu$ is idempotent, $\mathcal{G}^{00}$-right-invariant, and automorphism invariant over a small model, then $\mu$ is a translation invariant measure on its type-definable stabilizer subgroup of $\mathcal{G}$. 
\begin{definition}\label{def: stab, def amen}
\begin{enumerate}
\item Let $\mu \in \mathfrak{M}_{x}(\mathcal{G})$. The \emph{right stabilizer of $\mu$}, denoted as $\stab(\mu)$, is the subgroup of $\mathcal{G}$ defined as follows:
\begin{equation*} \stab(\mu) := \bigcap_{\varphi \in \mathcal{L}_x(\mathcal{G})}\left\{g \in \mathcal{G}: \mu(\varphi(x)) = \mu(\varphi(x \cdot g))\right\}. 
\end{equation*} 
\item  Let $\mathcal{H}$ be a subgroup of $\mathcal{G}$ (not necessarily definable). We say that $\mu \in \mathfrak{M}_{x}(\mathcal{G})$ is \emph{$\mathcal{H}$-right-invariant} (respectively, \emph{$\mathcal{H}$-left-invariant}) if for every formula $\varphi(x) \in \mathcal{L}_{x}(\mathcal{G})$ and $h \in \mathcal{H}$ we have $\mu(\varphi(x \cdot h)) = \mu(\varphi(x))$ (respectively, $\mu(\varphi(h \cdot x)) = \mu(\varphi(x))$). We say that $\mu$ is \emph{$\mathcal{H}$-invariant} if $\mu$ is both $\mathcal{H}$-left-invariant and $\mathcal{H}$-right-invariant. 
\item Let $\mathcal{H}$ be a type-definable subgroup of $\mathcal{G}$. We say that $\mathcal{H}$ is \emph{definably amenable} if there exists some $\mu \in \mathfrak{M}_{x}(\mathcal{G})$ such that $\tilde\mu([\mathcal{H}]) = 1$ (where $\tilde{\mu}$ is the unique regular Borel probability measure extending $\mu$) and $\mu$ is $\mathcal{H}$-right-invariant.
 Moreover, in this case we say that $(\mathcal{H}, \mu)$ is an \emph{amenable pair}.
\end{enumerate} 
\end{definition} 

\noindent The next proposition shows that if a Keisler measure witnesses
the definable amenability of some type-definable subgroup of $\mathcal{G}$, then this
subgroup must be its stabilizer:
\begin{proposition}\label{pair} Suppose that $\mu \in \mathfrak{M}_{x}(\mathcal{G})$ and $\mathcal{H}$ is a type-definable subgroup of $\mathcal{G}$. Suppose that $\tilde{\mu}([\mathcal{H}]) = 1$ and $\mathcal{H} \subseteq \stab(\mu)$. Then $\mathcal{H} = \stab(\mu)$.
\end{proposition}

\begin{proof} Suppose that that $\mathcal{H} \neq \stab(\mu)$, and let $g \in \stab(\mu) \backslash \mathcal{H}$. The subsets $[\mathcal{H}]$ and $[\mathcal{H}] \cdot g$ of $S_x(\mathcal{G})$ are disjoint and $\tilde{\mu} ([\mathcal{H}] \cup ([\mathcal{H}] \cdot g)) = 2$, where $\tilde{\mu}$ is the unique regular Borel probability measure extending $\mu$ to $S_{x}(\mathcal{G})$. This is a contradiction. 
\end{proof} 

\begin{definition} 
An idempotent measure $\mu \in \mathfrak{M}^{\inv}_{x}(\mathcal{G},G)$ is said to be \emph{pairless} if there does not exist a type-definable subgroup $\mathcal{H}$ of $\mathcal{G}$ such that $\left( \mathcal{H}, \mu \right)$ is an amenable pair. 
\end{definition} 
\begin{remark}
	By Proposition \ref{pair}, if $\stab(\mu)$ is type-definable, then $\mu$ is pairless if and only if $\mu([\stab(\mu)]) \neq 1$. 
\end{remark}

We now give two examples of pairless idempotent measures (in fact, types) in NIP groups (one definable, the other finitely satisfiable). Our third example shows that there can be many measures forming an amenable pair with a given group.

\begin{example}\label{example:pair} Let $T$ be the (complete)  theory of divisible ordered abelian groups,  $G := (\mathbb{R}, +, <) \models T$ and $\mathcal{G} \succ G$ a monster model of $T$.
\begin{enumerate} 
\item Let $p_{0^{+}}$ be the unique global definable (over $\mathbb{R}$) type extending  $\{x < a: a > 0, a \in \mathcal{G} \} \cup \{x >a: a \leq 0, a \in \mathbb{R}\}$. Then $\delta_{p_{0^{+}}} \in \mathfrak{M}^{\inv}_x(\mathcal{G},G)$ is idempotent and pairless. 
\item Let $p_{\mathbb{R}^{+}}$ be the unique global type finitely satisfiable in  $\mathbb{R}$ and extending $\{x > a: a \in \mathbb{R}\}$. Then $\delta_{p_{\mathbb{R}^{+}}}  \in \mathfrak{M}_x^{\fs}(\mathcal{G},G)$ is idempotent and pairless. 
\item Let $p_{+\infty}$ and $p_{- \infty}$ be the unique global heirs (over $\mathbb{R}$) extending the types $\Theta_{+}(x) := \{x > a: a \in \mathbb{R}\}$ 
and $\Theta_{-}(x) := \{x < a: a \in \mathbb{R}\}$ respectively. Then $(\mathcal{G}, \mu_{r})$ is an amenable pair for any $r \in [0,1]$, where
\begin{equation*} 
\mu_{r} = r\delta_{p_{-\infty}} + (1-r) \delta_{p_{+ \infty}}. 
\end{equation*} 
\end{enumerate} 
\end{example} 

\begin{proof}
\begin{enumerate}
\item Note that $\stab \left(\delta_{p_{0^{+}}} \right) = \{0\}$ and $\delta_{p_{0^{+}}}(\{0\}) = 0$, hence $\delta_{p_0^+}$ is pairless by Proposition \ref{pair}. We now check that $\delta_{p_{0^{+}}}$ is idempotent. Fix some $a \in \mathcal{G}$, some small $G' \prec \mathcal{G}$ containing $a$ and $\mathbb{R}$, and a realization $c \models p_{0^{+}}|_{G'}$ in $\mathcal{G}$. Note that
\begin{equation*} \left( p_{0^{+}} * p_{0^{+}}  \right) (x < a) = \left( p_{0^{+}} \otimes p_{0^{+}}  \right) (x + y < a) = p_{0^{+}} (x < a - c). 
\end{equation*} 
We now have two cases: 
\begin{enumerate}
\item If $a > 0$, then $a - c > 0$ and so $\left( p_{0^{+}} * p_{0^{+}} \right)  (x < a) = 1$. 
\item If $a \leq 0$, then $a - c < 0$ and so $\left( p_{0^{+}} * p_{0^{+}} \right) (x < a) = 0$. 
\end{enumerate} 
Hence, using quantifier-elimination, $p_{0^{+}} * p_{0^{+}} = p_{0^{+}}$, and so $\delta_{p_{0^{+}}}* \delta_{p_{0^{+}}} = \delta_{p_{0^{+}}}$. 
\item The measure $\delta_{p_{\mathbb{R}^{+}}}$ is idempotent by a computation analogous to the one in (1). We have 
$$\stab \left(\delta_{p_{\mathbb{R}^{+}}} \right) = \{a \in \mathcal{G}: -n < a < n \text{ for some n} \in \mathbb{N}\}.$$
 We note that $\stab \left(\delta_{p_{\mathbb{R}^{+}}} \right)$ is a $\bigvee$-definable subset of $\mathcal{G}$, but is not definable, hence it is not type-definable. 
Now suppose that there exists a type-definable subgroup $\mathcal{H}$ of $\mathcal{G}$ such that $\left(\mathcal{H},\delta_{p_{\mathbb{R}^{+}}} \right)$ is an amenable pair. Then, by definition,  $\mathcal{H} \subseteq \stab \left(\delta_{p_{\mathbb{R}^{+}}} \right)$ and $\delta_{p_{\mathbb{R}^{+}}} \left([\mathcal{H}] \right) = 1$. By Proposition \ref{pair}, we conclude that $\mathcal{H} = \stab \left(\delta_{p_{\mathbb{R}^{+}}} \right)$. Hence $\stab \left(\delta_{p_{\mathbb{R}^{+}}} \right)$ is type-definable, a contradiction. Alternatively, we get a contradiction by regularity of the measure:
\begin{equation*} 
\delta_{p_{\mathbb{R}^{+}}} \left( \left[\stab \left(\delta_{p_{\mathbb{R}^{+}}} \right) \right] \right) = \sup \left(\left\{\delta_{p_{\mathbb{R}^{+}}}([-n < x < n]): n \in \mathbb{N}\right\} \right) = 0.
\end{equation*} 

\item Note that $p_{+ \infty}$ and $p_{- \infty}$ are (left and right) $\mathcal{G}$-invariant. Hence $\mu_{r} := r\delta_{p_{- \infty}} + (1-r)\delta_{p_{+ \infty}} \in \mathfrak{M}_x(\mathcal{G})$ is $\mathcal{G}$-invariant for any $r \in [0,1]$. Since $\mu_{r}$ is $\mathcal{G}$-invariant, $(\mathcal{G},\mu_{r})$ is an amenable pairing for every $r \in [0,1]$. \qedhere
\end{enumerate} 
\end{proof} 

In the rest of this section we show that in an NIP group $\mathcal{G}$,  for any $\mathcal{G}^{00}$-invariant idempotent $\mu \in \mathfrak{M}^{\inv}_x(\mathcal{G}, G)$, $\stab(\mu)$ is type-definable and $\left(\stab(\mu), \mu \right)$ is an amenable pair.

\begin{definition}\label{def:wendel} Assume that $\mu \in \mathfrak{M}^{\inv}_{x}(\mathcal{G},G)$ is idempotent. By Corollary \ref{cor:wendel}, the measure $\pi_{*}(\mu) \in \mathcal{M}(\mathcal{G}/\mathcal{G}^{00})$ is idempotent and by Fact \ref{compact:corr}, $\supp(\pi_{*}(\mu))$ is a closed subgroup of $\mathcal{G}/\mathcal{G}^{00}$ and $\pi_{*}(\mu) \restriction_{\supp(\pi_{*}(\mu))}$ is the normalized Haar measure on this closed subgroup. Then $\pi^{-1}(\supp(\pi_{*}(\mu)))$ is a type-definable subgroup of $\mathcal{G}$. We let $H_{\mathcal{L}}(\mu) := \pi^{-1}(\supp(\pi_{*}(\mu)))$. 
\end{definition} 

\begin{proposition}\label{useful:lemma} Suppose $\mu \in \mathfrak{M}^{\inv}_{x}(\mathcal{G},G)$ is idempotent and $\mathcal{G}^{00}$-right-invariant.
\begin{enumerate}[$(i)$]
\item If $p \in \sup(\mu)$, then $\hat{\pi}(p) \in \supp(\pi_{*}(\mu))$ (see Fact \ref{group:fact} for the definition of $\hat{\pi}$). 
\item If $p \in \sup(\mu)$, then $p \in [H_{\mathcal{L}}(\mu)]$. 
\item $\mu([H_{\mathcal{L}}(\mu)]) = 1$.
\item If $b \in \stab(\mu)$, then $\pi(b) \in \stab(\pi_{*}(\mu))$. 
\end{enumerate} 
\end{proposition} 

\begin{proof}
\begin{enumerate}[$(i)$]
\item Let $U$ be an open subset of $\mathcal{G}/\mathcal{G}^{00}$ containing $\hat{\pi}(p)$. Then $\pi^{-1}(U)$ is $\bigvee$-definable, so $\pi^{-1}(U) = \bigvee_{i \in I} \psi_{i}(x)$ for some $\psi_i \in \mathcal{L}_x(G)$. Hence there exists some $i \in I$ so that $\psi_{i}(x) \in p$. Since $p \in \sup(\mu)$, we have that $\mu(\psi_{i}(x)) > 0$. Then 
\begin{equation*} \pi_{*}(\mu)(U) = \tilde{\mu}([\hat{\pi}^{-1}(U)]) \geq \mu(\psi_{i}(x)) > 0,
\end{equation*} 
where $\tilde{\mu}$ is the unique regular Borel probability measures extending $\mu$. Therefore $\hat{\pi}(p) \in \supp(\pi_{*}(\mu))$.
\item Obvious by $(i)$.
\item Assume not. Then $\mu(S_{x}(\mathcal{G})\setminus [H_{\mathcal{L}}(\mu)]) > 0$. This set is open and so by regularity there exists some $[\psi(x)] \subset S_{x}(\mathcal{G}) \setminus [\mathcal{H}_{\mathcal{L}}(\mu)]$ such that $\mu(\psi(x)) > 0$. Then there exists some $p \in \sup(\mu)$ so that $\psi(x) \in p$. This contradicts $(ii)$. 
\item By Theorem \ref{invar:conv}, 
\begin{equation*} \pi_{*}(\mu) \cdot \pi(b)  =   \pi_{*}(\mu) \star \delta_{\pi(b)} = \pi_{*}( \mu \ast \delta_{b}) = \pi_{*}(\mu). \qedhere
\end{equation*}
\end{enumerate}
\end{proof}

\begin{lemma}\label{Borel:quotient} Assume that $f:S_{x}(\mathcal{G}) \to \mathbb{R}$ is Borel and factors through $\hat{\pi}:S_{x}(\mathcal{G}) \to \mathcal{G}/\mathcal{G}^{00}$, and let $f_{\star}:\mathcal{G}/\mathcal{G}^{00} \to \mathbb{R}$ be the factor map. Then $f_{\star}$ is Borel. 
\end{lemma} 

\begin{proof} The map $\hat{\pi}:S_{x}(\mathcal{G}) \to \mathcal{G}/\mathcal{G}^{00}$ is a continuous surjective map between compact Hausdorff spaces. If the map $f = f_{\star} \circ \hat{\pi}$ is Borel, then $f_{\star}$ is Borel by \cite[Theorem 10]{holicky2003perfect} (see \cite[Theorem 2.1]{CGH} for an explanation).
\end{proof} 

\begin{lemma}\label{G:comp} Assume that $\mu \in \mathfrak{M}_{x}^{\inv}(\mathcal{G},G)$ is idempotent and $\mathcal{G}^{00}$-right-invariant. Suppose that $p \in \sup(\mu|_{G})$ and $a \models p$ in $\mathcal{G}$. Then $\mu(\varphi(x)) = \mu(\varphi(x \cdot a))$ for any $\varphi(x) \in \mathcal{L}_x(\mathcal{G})$. 
\end{lemma}

\begin{proof} Fix $p \in \sup(\mu|_{G})$, $\varphi(x) \in \mathcal{L}_x(\mathcal{G})$ and $a \in \mathcal{G}$ such that $a \models p$. Fix a small model $G' \prec \mathcal{G}$ such that $G'$ contains $G$, $a$, and all of the parameters of $\varphi$. Let $r:S_{y}(\mathcal{G}) \to S_{y}(G'), q \mapsto q|_{G'}$ be the restriction map. Since $\mu$ is idempotent,
\begin{equation*} \mu(\varphi(x \cdot a)) = \mu * \mu (\varphi(x \cdot a)) = \int_{S_{y}(G')} F_{\mu,G'}^{\varphi_{a}'}d\mu_{G'} = \int_{S_{y}(\mathcal{G})} \left( F_{\mu,G'}^{\varphi_{a}'} \circ r \right) d\mu,
\end{equation*} 
\noindent where $\varphi_a(x) := \varphi(x \cdot a)$, hence $\varphi_a'(x,y) = \varphi(x \cdot y \cdot a)$ and $F_{\mu,G'}^{\varphi_{a}'}(q) = \mu(\varphi(x\cdot c \cdot a))$ for some/any $c \models q$ (see Definition \ref{def: def conv}).   Let $f := F_{\mu,G'}^{\varphi_{a}'} \circ r$ and $h := F_{\mu,G'}^{\varphi'} \circ r$.

~

\noindent \textbf{Claim 1.} Both $f$ and $h$ factor through $\hat{\pi}:S_{y}(\mathcal{G}) \to \mathcal{G}/\mathcal{G}^{00}$. 

\begin{proof} The proofs are essentially the same, so we only show that $f$ factors through $\hat{\pi}$. Fix $q_{1},q_{2} \in S_{y}(\mathcal{G})$ with $\hat{\pi}(q_{1}) = \hat{\pi}(q_{2})$, we want to show that then $f(q_1) = f(q_2)$. Let $b_1, b_2 \in \mathcal{G}$ be such that $b_1 \models r(q_1)$ and $b_2 \models r(q_2)$, then $\pi(b_1) = \pi(b_2)$. Since $\mathcal{G}^{00}$ is a normal subgroup of $\mathcal{G}$, we then have $b_{1} = d \cdot b_2$ for some $d \in \mathcal{G}^{00}$.
Hence 
\begin{equation*} f(q_1) = \left( F_{\mu,G'}^{\varphi_{a}'} \circ r \right)(q_1) = \mu(\varphi(x \cdot b_{1} \cdot a)) = \mu(\varphi(x \cdot d \cdot b_2 \cdot a)).
\end{equation*} 
And since $\mu$ is $\mathcal{G}^{00}$-right-invariant, we have $\mu(\psi(x \cdot d)) = \mu(\psi(x))$ for  $\psi(x) := \varphi(x \cdot b_2 \cdot a)$, that is 
\begin{equation*}  \mu(\varphi(x \cdot d \cdot b_2 \cdot a)) = \mu(\varphi(x \cdot b_2 \cdot a)) = \left( F_{\mu,G'}^{\varphi_{a}'} \circ r \right)(q_2) = f(q_2). \qedhere
\end{equation*} 
\end{proof}

\noindent We let $f_{\star}$ and $h_{\star}$ be the associated factor maps from $\mathcal{G}/\mathcal{G}^{00}$ to $\mathbb{R}$.

~

\noindent \textbf{Claim 2.} We have $h_{\star} \cdot \pi(a) =  f_{\star}$, where $h_{\star} \cdot \pi(a): \mathcal{G}/\mathcal{G}^{00} \to \mathbb{R}$ is the function defined by $(h_{\star} \cdot \pi(a))(\mathbf{b}) := h_{\star} (\mathbf{b} \cdot \pi(a))$ for any $\mathbf{b} \in \mathcal{G}/\mathcal{G}^{00}$. 

\begin{proof}
Fix $\mathbf{b} \in \mathcal{G}/\mathcal{G}^{00}$ and $b \in \mathcal{G}$ such that $\pi(b) = \mathbf{b}$. Then
\begin{gather*}(h_{\star} \cdot \pi(a))(\mathbf{b}) = (h_{\star})(\mathbf{b} \cdot \pi(a)) = (h_{\star})(\pi(b \cdot a)) = \left( F_{\mu,G'}^{\varphi'} \circ r \right) \left( \tp(b \cdot a/ \mathcal{G}) \right)\\
 = F_{\mu,G'}^{\varphi'}(\tp(b\cdot a/G')) = \mu(\varphi(x \cdot b \cdot a)) =  F_{\mu,G'}^{\varphi'_a} \left( \tp(b/G') \right) = f_{\star}(\mathbf{b}).\qedhere
\end{gather*} 
\end{proof}

\noindent \textbf{Claim 3.} $\mu(\varphi(x \cdot a)) = \mu(\varphi(x))$. 

\begin{proof}
The maps $f_{\star}, h_{\star}: \mathcal{G}/\mathcal{G}^{00} \to \mathbb{R}$ are Borel by Lemma \ref{Borel:quotient}. By assumption $a \models p$ with $p \in \sup(\mu|_G)$. Then there exists $\hat{p} \in \sup(\mu)$ such that $\hat{p}|_{G} = p$ (see e.g.~\cite[Proposition 2.8]{ChGan}). By Proposition \ref{useful:lemma}(i) we then have $\pi(a) = \hat{\pi} \left( \hat{p} \right) \in \supp(\pi_{*}(\mu))$. The measure $\pi_{*}(\mu)$ is idempotent by Corollary \ref{cor:wendel}. Applying Fact  \ref{int:comp} (to the compact group $\mathcal{G}/\mathcal{G}^{00}$ and its closed subgroup $\supp(\pi_{*}(\mu)) \ni \pi(a)$) we get 
$$\int_{\mathcal{G}/\mathcal{G}^{00}}  \left( h_{\star} \cdot \pi(a) \right) d\pi_{*}(\mu) = \int_{\mathcal{G}/\mathcal{G}^{00}} h_{\star} d\pi_{*}(\mu).$$
Using this and Claim 2 we have the following computation: 
\begin{gather*} \mu(\varphi(x \cdot a)) = \left( \mu * \mu \right)(\varphi(x \cdot a)) \\
= \int_{S_{y}(\mathcal{G})} f d\mu = \int_{\mathcal{G}/\mathcal{G}^{00}} f_\star d \pi_{*}(\mu) = \int_{\mathcal{G}/\mathcal{G}^{00}}  \left( h_{\star} \cdot \pi(a) \right) d\pi_{*}(\mu) = \int_{\mathcal{G}/\mathcal{G}^{00}} h_{\star} d\pi_{*}(\mu)\\
 = \int_{S_{y}(\mathcal{G})} h d\mu = \int_{S_{y}(G')} F_{\mu,G'}^{\varphi'} d\mu_{G'}  = \left( \mu * \mu \right) (\varphi(x)) = \mu(\varphi(x)).\qedhere
\end{gather*} 
\end{proof}
\noindent This concludes the proof of Lemma \ref{G:comp}.
\end{proof}

\begin{lemma}\label{lemma:sup} Suppose that $\mathbf{g} \in \supp(\pi_{*}(\mu))$. Then there exists some $p \in \sup(\mu|_{G})$ such that for any $b \models p$ we have $\pi(b) = \mathbf{g}$. 
\end{lemma} 

\begin{proof} We use the fact that $\pi_{*}: \mathcal{M}(S_{x}(\mathcal{G})) \to \mathcal{M}(\mathcal{G}/\mathcal{G}^{00})$ is a push-forward map. Let $\tilde{\mu}$ be the unique extension of $\mu$ to a regular Borel probability measure on $S_{x}(\mathcal{G})$.  Let $\mathbf{g} \in \supp(\pi_{*}(\mu))$ and let $U \subseteq \mathcal{G}/\mathcal{G}^{00}$ be an open set containing $\mathbf{g}$. Since $\mathbf{g} \in \supp(\pi_{*}(\mu))$, we have that $0 < \pi^{*}(\mu)(U) = \tilde{\mu} \left( [\pi^{-1}(U)] \right)$. Then there exists some $p_{U} \in \supp(\tilde{\mu})$ such that $p_{U} \in [\pi^{-1}(U)]$. The collection of open sets in $\mathcal{G}/\mathcal{G}^{00}$ containing $\mathbf{g}$  forms a directed family under reverse inclusion, and we can consider the net $(p_{U})_{\mathbf{g} \in U}$. Since $\supp(\tilde{\mu})$ is closed and hence compact, there exists a convergent subnet $(q_{i})_{i \in I}$ with a limit in $\supp \left(\tilde{\mu} \right)$. Let $q := \lim_{i \in I} q_i$. By continuity of $\hat{\pi}:S_{x}(\mathcal{G}) \to \mathcal{G}/\mathcal{G}^{00}$, we have that $\hat{\pi}(q) = \mathbf{g}$. Since $\sup(\mu) = \supp(\tilde{\mu})$ we conclude that $q \in \sup(\mu)$. By definition of $\hat{\pi}$ we have $\hat{\pi}(q) = \pi(b)$ for any $b \models q|_{G}$, so the lemma holds with $p := q|_{G}$. 
\end{proof} 

\begin{theorem} Suppose that $\mu \in \mathfrak{M}_{x}^{\inv}(\mathcal{G},G)$ is idempotent and $\mathcal{G}^{00}$-right-invariant. Then: 
\begin{enumerate}
\item $\stab(\mu) = H_{\mathcal{L}}(\mu)$ (see Definition \ref{def:wendel}); 
\item $\stab(\mu)$ is a type-definable subgroup of $\mathcal{G}$; 
\item $(\stab(\mu),\mu)$ is an amenable pair. 
\end{enumerate} 
\end{theorem} 

\begin{proof} (1) As $H_{\mathcal{L}}(\mu)$ is a type-definable subgroup of $\mathcal{G}$, by Proposition \ref{pair} it suffices to show that $\mu$ is $H_{\mathcal{L}}(\mu)$-right-invariant and $\mu \left( [H_{\mathcal{L}}(\mu)] \right) = 1$. By Proposition \ref{useful:lemma}(iii) we have $\mu([H_{\mathcal{L}}(\mu)]) = 1$, so it remains to show that $H_{\mathcal{L}}(\mu) \subseteq \stab(\mu)$. Fix $a \in H_{\mathcal{L}}(\mu)$. Then $\mathbf{g} := \pi(a) \in \supp(\pi_{*}(\mu))$ by Proposition \ref{useful:lemma}(i). By Lemma \ref{lemma:sup}, there exists some $p \in \sup(\mu|_{G})$ and $b \models p$ such that $\pi(b) = \mathbf{g}$. In particular $a \cdot \mathcal{G}^{00} = b \cdot \mathcal{G}^{00}$,  so $a = c \cdot b$ for some $c \in \mathcal{G}^{00}$. Now we have 
\begin{equation*} \mu(\varphi(x \cdot a)) = \mu(\varphi(x \cdot c \cdot b)) = \mu(\varphi(x \cdot b)) = \mu(\varphi(x)).
\end{equation*} 
The second equality follows from the fact that $\mu$ is $\mathcal{G}^{00}$-right-invariant and the fourth equality follows from Lemma \ref{G:comp}. 

Statement (2) follows from the fact that $\stab(\mu) = \mathcal{H}_{\mathcal{L}}(\mu)$ and $\mathcal{H}_{\mathcal{L}}(\mu)$ is type-definable. Statement (3) follows since $\mu([\stab(\mu)]) = \mu([\mathcal{H}_{\mathcal{L}}(\mu)]) = 1$. 
\end{proof}

\section{The structure of convolution semigroups}\label{sec: struct of conv semigrps}

By Fact \ref{fac: conv is a semigroup in NIP}, if $T$ is an NIP theory expanding a group, then both $(\mathfrak{M}^{\inv}_{x}(\mathcal{G},G),*)$ and $(\mathfrak{M}^{\fs}_{x}(\mathcal{G},G),*)$ are left-continuous compact Hausdorff  semigroups (hence satisfy the assumption of Fact \ref{fact:Ellis}). In this section we describe some properties of the minimal left ideals and ideal groups which arise in this setting. Unlike the better studied case of the semigroup $\left( S_{x}^{\fs}(\mathcal{G},G), \ast \right)$, we demonstrate that the ideal subgroups of any minimal left ideal (in either $\mathfrak{M}^{\fs}_{x}(\mathcal{G},G)$ or $\mathfrak{M}^{\inv}(\mathcal{G},G)$) are always trivial, i.e.~isomorphic to the group with a single element. The following theorem summarizes the properties that we will prove in this section. 

\begin{theorem}\label{thm: main props of min ideals} Assume that $\mathcal{G}$ is NIP, and let $I$ be a minimal left ideal of $\mathfrak{M}_{x}^{\dagger}(\mathcal{G},G)$ (exists by Fact \ref{fact:Ellis}). Then we have the following.
\begin{enumerate}
\item $I$ is a closed convex subset of $\mathfrak{M}_{x}^{\dagger}(\mathcal{G},G)$ (Proposition \ref{convex:compact}). 
\item For any $\mu \in I$, $\pi_{*}(\mu) = h$, where $h$ is the normalized Haar measure on $\mathcal{G}/\mathcal{G}^{00}$ (Proposition \ref{haar}). 
\item If $\mathcal{G}/\mathcal{G}^{00}$ is non-trivial, then $I$ does not contain any types (Proposition \ref{prop: GmodG00 non triv no types}).
\item For any idempotent $u \in I$, we have $u * I \cong (e, \cdot)$. In other words, the ideal group is always trivial (Proposition \ref{trivial}). 
\item Every element of $I$ is an idempotent (Proposition \ref{mult}). 
\item If $\mu, \nu \in I$ then $\mu * \nu = \mu$ (Proposition \ref{mult}). 
\item For any $\mu \in I$, $I = \left\{\nu \in \mathfrak{M}_{x}^{\dagger}(\mathcal{G},G): \nu * \mu = \nu\right\}$ (Corollary \ref{needed}). 
\item For any definable measure $\nu \in \mathfrak{M}_{x}^{\dagger}(\mathcal{G},G)$ there exists a measure $\mu \in I$ such that $\nu * \mu = \mu$. In particular, for any $g \in G$ there exists a measure $\mu \in I$ such that $\delta_{g} * \mu = g \cdot \mu = \mu$ (Proposition \ref{prop: G not free action}). 
\item Assume that $\mathcal{G}$ is definably amenable.
\begin{enumerate} 
\item If $\dagger = \fs$, then $I = \{\nu\}$, where $\nu \in \mathfrak{M}_{x}^{\fs}(\mathcal{G},G)$ is a $G$-left-invariant measure (Proposition \ref{fs:I}). 
\item If $\dagger = \inv$, then
\begin{equation*} I = \left\{\mu \in \mathfrak{M}_{x}^{\inv}(\mathcal{G},G): \mu \text{ is } \mathcal{G}\text{-right-invariant}\right\}. 
\end{equation*} 
Moreover, $I$ is a two-sided ideal, and is the unique minimal left ideal (Proposition \ref{ideal:inv}). The set $\ex(I)$ of extreme points of $I$ is closed and equal to $\left\{ \mu_p : p \in S_x^{\inv}(\mathcal{G},G) \textrm{ is right }f\textrm{-generic} \right\}$, and $I$ is a Bauer simplex (Corollary \ref{cor: ex inv f-generic}).
\end{enumerate} 
\item If $\mathcal{G}$ is fsg and $\mu \in \mathfrak{M}_x(\mathcal{G})$ is the unique $\mathcal{G}$-left-invariant measure, then  $I = \{\mu\}$ is the unique minimal left (in fact, two-sided) ideal in both  $\mathfrak{M}^{\inv}_x(\mathcal{G},G)$ and $\mathfrak{M}^{\fs}_x(\mathcal{G},G)$ (Corollary \ref{cor: fsg}).
\item If $\mathcal{G}$ is not definably amenable, then the closed convex set $I$ has infinitely many extreme points (Remark \ref{prop: many idemp in non def am}).

\end{enumerate} 
\end{theorem} 
\noindent We remark that $(5)$ and $(11)$ of Theorem \ref{thm: main props of min ideals} guarantee the existence of many idempotent measures in non-definably amenable NIP groups. All previous ``constructions" of idempotent measures either explicitly or implicitly use definable amenability or amenability of closed subgroups of $\mathcal{G}/\mathcal{G}^{00}$. A priori, the idempotent measures we find here have no connection to type-definable   subgroups.

\subsection{General structure}

Our first goal is to show that any minimal left ideal of $\mathfrak{M}_{x}^{\dagger}(\mathcal{G},G)$ is convex. We begin by showing that convolution is affine in both arguments, hence preserves convexity on both sides.
\begin{lemma}\label{convex:1} 
Assume $\mu, \lambda_1, \lambda_2 \in \mathfrak{M}_{x}^{\dagger}(\mathcal{G},G)$ and $r,s \in \mathbb{R}_{>0}$ with $r + s = 1$. We have:
\begin{enumerate}
\item $\left( r \lambda_1 + s \lambda_2 \right) \ast \mu = r (\lambda_1 \ast \mu) + s (\lambda_2 \ast \mu)$;
\item $\mu \ast \left( r \lambda_1 + s \lambda_2 \right) = r (\mu \ast \lambda_1) + s (\mu \ast \lambda_2)$;
	\item if $A \subseteq \mathfrak{M}_{x}^{\dagger}(\mathcal{G},G)$ is convex, then both $\mu * A$ and $A * \mu$ are convex. 

\end{enumerate}
\end{lemma} 

\begin{proof} Parts (1) and (2) were stated in \cite[Proposition 3.14(4)]{ChGan}, but no proof was provided there, so we take the opportunity to provide it here. 

(1) Fix a formula $\varphi(x) \in \mathcal{L}_{x}(\mathcal{G})$ and let $G'$ be a small model containing $G$ and the parameters of $\varphi$. Then 
\begin{gather*} 
\left( (r\lambda_1 + s\lambda_2) * \mu \right)(\varphi(x)) = \int_{S_{y}(G')} F_{r\lambda_1 + s\lambda_2}^{\varphi'} d\mu_{G'} =  \int_{S_{y}(G')} \left( rF_{\lambda_1}^{\varphi'} + sF_{\lambda_2}^{\varphi'} \right) d\mu_{G'} \\
= r \int_{S_{y}(G')} F_{\lambda_1}^{\varphi'}  d\mu_{G'} + s \int_{S_{y}(G')} F_{\lambda_2}^{\varphi'}  d\mu_{G'} = r\left(\lambda_1 * \mu\right)(\varphi(x)) +  s\left(\lambda_2 * \mu\right)(\varphi(x))\\
= \left(r (\lambda_1 \ast \mu) + s (\lambda_2 \ast \mu)\right)(\varphi(x)).
\end{gather*} 

(2) Fix a formula $\varphi(x) \in \mathcal{L}_{x}(\mathcal{G})$ and a small model $G'$ containing $G$ and the parameters of $\varphi$. Then the map $F_{\mu}^{\varphi'}:S_{y}(G') \to [0,1]$ is a bounded Borel function, hence for any $\varepsilon > 0$ there exist Borel subsets $B_1, \ldots, B_n$ of $S_{y}(G')$ and real numbers $k_1, \ldots, k_n$ such that 
\begin{equation*}
\sup_{q \in S_{y}(G)} \left \lvert F_{\mu}^{\varphi'}(q) - \sum_{i=1}^{n} k_i \mathbf{1}_{B_i}(q) \right \rvert < \varepsilon. 
\end{equation*} 
Now we compute the product:
\begin{gather*}
\left( \mu*(r\lambda_1 + s\lambda_2) \right) (\varphi(x)) = \int_{S_{y}(G')} F_{\mu}^{\varphi'} d(r\lambda_1 + s\lambda_2) \\
\approx_{\varepsilon} \int_{S_{y}(G')}  \left( \sum_{i=1}^{n} k_i\mathbf{1}_{B_i} \right) d(r\lambda_1 + s\lambda_2)
= r\sum_{i=1}^{n} k_i \lambda_1(B_i) + s\sum_{i=1}^{n} k_i\lambda_2(B_i) \\
= r\int_{S_{y}(G')}  \left( \sum_{i=1}^{n}k_i \mathbf{1}_{B_i} \right) d\lambda_1 + s\int_{S_{y}(G')} \left(  \sum_{i=1}^{n}k_i \mathbf{1}_{B_i} \right) d\lambda_2 \\
\approx_{\varepsilon} r \int_{S_{y}(G')}  F_{\mu}^{\varphi'} d\lambda_1 + s \int_{S_{y}(G')}  F_{\mu}^{\varphi'}d\lambda_2 =\left(r (\mu \ast \lambda_1) + s (\mu \ast \lambda_2) \right)(\varphi(x)).
\end{gather*} 

(3) We first prove that $A * \mu$ is convex. Let $\nu_1,\nu_2 \in A * \mu$ and $r,s \in \mathbb{R}_{>0}$ with $r + s =1$ be given, we need to show that $r\nu_1 + s\nu_2 \in A*\mu$. By assumption there exist some $\lambda_1,\lambda_2 \in A$ such that $\lambda_i * \mu = \nu_i$ for $i \in \{1,2\}$. Since $A$ is convex, we have that $r\lambda_{1} + s\lambda_{2} \in A$. It follows by (1) that $r \nu_1 + s \nu_2 = (r\lambda_1 + s\lambda_2) * \mu \in A \ast \mu$.

Now we prove that $\mu*A$ is convex. Similarly, let $\nu_1,\nu_2 \in \mu *A$ and $r,s \in \mathbb{R}_{>0}$ with $r + s =1$ be given, and let $\lambda_1,\lambda_2 \in A$ be such that $\mu * \lambda_i = \nu_i$. Consider the measure $r\lambda_1 + s\lambda_2 \in A$. It follows by (2) that $r\nu_1 + s\nu_2 = \mu*(r\lambda_1 + s\lambda_2) \in \mu \ast A$.
\end{proof} 

\begin{proposition}\label{convex:compact} If $I$ is a minimal left ideal in $\mathfrak{M}_{x}^{\dagger}(\mathcal{G},G)$, then $I$ is closed and convex. 
\end{proposition}

\begin{proof} Any minimal left ideal is closed by Fact \ref{fact:Ellis}. Choose $\mu \in I$. By Fact \ref{fact:Ellis}(5), we have $\mathfrak{M}_{x}^{\dagger}(\mathcal{G},G)*\mu = I$. By Lemma \ref{convex:1} and the convexity of $\mathfrak{M}_{x}^{\dagger}(\mathcal{G},G)$, $I$ is convex. 
\end{proof} 

\noindent We now consider the interaction between the push-forward map to $\mathcal{G}/\mathcal{G}^{00}$ and the minimal left ideal. The following lemma is standard.

\begin{lemma}\label{two-sided} Let $S$ be a semigroup, $L$ a minimal left ideal of $S$ and $H$ a two-sided ideal in $S$. Then $L \subseteq H$. 
\end{lemma} 
\begin{proof}
	Note that $L' := L \cap H$ is non-empty (for $l \in L$ and $h \in H$, $h\cdot l \in L \cap H$) and is a left ideal (as an intersection of two left ideals). As $L' \subseteq L$, by minimality $ L = L' \subseteq H$.
\end{proof}

\begin{proposition}\label{haar} Let $I$ be a minimal left ideal in $\mathfrak{M}_{x}^{\dagger}(\mathcal{G},G)$. Then for every $\nu \in I$ we have $\pi_{*}^{\dagger}(\nu) = h$, where $h$ is the normalized Haar measure on $\mathcal{G}/\mathcal{G}^{00}$. 
\end{proposition} 

\begin{proof} Since $\pi_{*}^{\dagger}$ is surjective (Proposition \ref{surject}) and continuous (Remark \ref{rem: pi star is cont}), the set $A := \left( \pi_{*}^{\dagger} \right)^{-1}(\{h\})$ is a non-empty closed subset of $\mathfrak{M}_{x}^{\dagger}(\mathcal{G},G)$. Moreover, $A$ is a two-sided ideal: since $\pi_{*}^{\dagger}$ is a homomorphism (Theorem \ref{invar:conv}) and $h$ is both left and right invariant, for any $\mu \in A$ and $\nu \in \mathfrak{M}_{x}^{\dagger}(\mathcal{G},G)$ we have
\begin{equation*} \pi_{*}^{\dagger}(\nu * \mu) = \pi_{*}^{\dagger}(\nu) \star \pi_{*}^{\dagger}(\mu) = \pi_{*}^{\dagger}(\nu) \star h= h,
\end{equation*} 
and a similar computation also shows that $A$ is a right ideal. By Lemma \ref{two-sided} we have $I \subseteq A$, which completes the proof. 
\end{proof} 

\begin{definition} Let $\mu \in \mathfrak{M}_{x}(\mathcal{G})$. We say that $\mu$ is \emph{strongly continuous} if for every $\varepsilon > 0$, there exists a finite partition $\left\{[\psi(x)]\right\}_{i < n}$ of $S_{x}(\mathcal{G})$ with $\psi_i \in \mathcal{L}_x \left(\mathcal{G} \right)$ such that $\mu(\psi(x)) < \varepsilon$ for all $i < n$. 
\end{definition}

\begin{proposition}\label{prop: GmodG00 non triv no types} Let $I$ be a minimal left ideal in $\mathfrak{M}_{x}^{\dagger}(\mathcal{G},G)$.
\begin{enumerate}
\item If $\mathcal{G}/\mathcal{G}^{00}$ is non-trivial, then $I$ does not contain any types.
\item If $\mathcal{G}/\mathcal{G}^{00}$ is infinite, then every measure in $I$ is strongly continuous.
\end{enumerate} 
\end{proposition} 

\begin{proof}
(1) By Lemma \ref{basic:lem}(2) we have $\pi_{*}^{\dagger}(\delta_{p}) = \delta_{\hat{\pi}(p)}$, which does not equal the normalized Haar measure on $\mathcal{G}/\mathcal{G}^{00}$ when it is non-trivial. This contradicts Proposition \ref{haar}.

(2) If $\mathcal{G}/\mathcal{G}^{00}$ is infinite then the normalized Haar measure $h$ on $\mathcal{G}/\mathcal{G}^{00}$ is zero on every point. Suppose that $\nu \in \mathfrak{M}_{x}^{\dagger}(\mathcal{G},G)$ is not strongly continuous. By compactness and \cite[Theorem 5.2.7]{Charges},  $\nu$ can be written as 
\begin{equation*} \nu = r_0 \mu_0 + \sum_{i \in \omega} r_i \delta_{p_i},
\end{equation*} 
where $\mu_0 \in \mathfrak{M}_{x}^{\dagger}(\mathcal{G},G)$ is strongly continuous, $r_i \in [0,1]$ and $p_i \in S_x^{\dagger} \left(\mathcal{G},G \right)$ for each $i \in \omega$, and $\sum_{i \in \omega} r_i = 1$.  We then must have $r_{i^*} > 0$ for some $i^* \in \omega \backslash \{0\}$. Since the push-forward map is affine (Remark \ref{rem: pi star is cont}), we have
\begin{equation*} 
\pi_{*}(\nu) = r_0\pi_{*}(\mu_0) + \sum_{i\in \omega} r_i \delta_{\hat{\pi}(p_i)}. 
\end{equation*}
Hence $\pi_{*}(\nu) \left(\left\{\hat{\pi}(p_{i^*})\right\} \right) = r_{i^*} > 0$, so $\pi_{*}(\nu) \neq h$,  contradicting Proposition \ref{haar}.
\end{proof} 

We now show that the ideal subgroup of any minimal left ideal is trivial. A related result appears in \cite[Theorem 3]{CC}, but we are working in a semigroup which is only left-continuous. Our proof is a generalization of the proof that there do not exist any non-trivial convex compact groups and  follows \cite[Lemmas 3.1 and 3.2]{Murphy}. In particular, compactness is used only to get an extreme point in \textit{some} ideal subgroup. Some elementary algebra is then used to show that the only possible ideal subgroups are isomorphic to a single point. 

\begin{lemma}\label{extreme} If $I$ is a minimal left ideal in $\mathfrak{M}_{x}^{\dagger}(\mathcal{G},G)$, then $\ex(I) \neq \emptyset$. 
\end{lemma} 

\begin{proof} By Proposition \ref{convex:compact}, $I$ is a compact convex set. By the Krein-Milman theorem, $I$ contains an extreme point. 
\end{proof} 

\begin{lemma}\label{exidempotent} If $I$ is a minimal left ideal in $\mathfrak{M}_{x}^{\dagger}(\mathcal{G},G)$, then there exists an idempotent $\mu$ in $I$ such that $\mu \in \ex(\mu * I)$. 
\end{lemma} 

\begin{proof} By Lemma \ref{extreme}, there exists a measure $\nu \in I$ which is extreme in $I$. By Fact \ref{fact:Ellis}(4), there exists an idempotent $\mu$ in $I$ such that $\nu \in \mu * I$. Towards a contradiction, suppose that $\mu \not \in \ex(\mu * I)$. Then there exist distinct $\eta_1, \eta_2 \in \mu * I$ and $r \in (0,1)$ such that $r \eta_1 + (1-r)\eta_2 = \mu$. 
As $\mu$ is the identity of the group $\mu * I$ by Fact  \ref{fact:Ellis}(3), we get
\begin{gather*}
	\nu = \nu * \mu = r \left( \nu * \eta_1 \right) +(1-r) \left( \nu * \eta_2 \right).
\end{gather*}
Since $\nu \in \ex(I)$ and $\nu \ast \eta_i \in I$ as $I$ is a left ideal, it follows that $\nu = \nu * \eta_1 = \nu * \eta_2$. Since $\nu,\eta_1, \eta_2 \in \mu *I$ and  $\mu *I$ is a group, this implies $\eta_1 = \eta_2$, contradicting the assumption. Hence $\mu \in \ex(\mu * I)$. 
\end{proof} 

\begin{proposition}\label{trivial} The ideal subgroup of $\mathfrak{M}_{x}^{\dagger}(\mathcal{G},G)$ is trivial. 
\end{proposition} 

\begin{proof} Let $I$ be a minimal left ideal of $\mathfrak{M}_{x}^{\dagger}(\mathcal{G},G)$. By Lemma \ref{exidempotent}, there exists an idempotent $\mu \in I$ such that $\mu$ is extreme in $\mu *I$. Let $\eta_1, \eta_2 \in \mu *I$. We will show that $\eta_1 = \eta_2$. By Lemma \ref{convex:1} and Proposition \ref{convex:compact}, $\mu * I$ is convex. Hence $\alpha := \frac{\eta_1 + \eta_2 }{2} \in \mu * I$. Since $\mu *I$ is a group with identity $\mu$, $\mu *I$ contains $\alpha^{-1}$ (i.e. $\alpha^{-1} * \alpha =\alpha * \alpha^{-1} =  \mu$). Then
\begin{equation*}
\mu =\alpha^{-1} * \alpha= \alpha^{-1} * \left( \frac{1}{2} \eta_1 + \frac{1}{2} \eta_2 \right) = \frac{1}{2} \left( \alpha^{-1} * \eta_1 \right)  + \frac{1}{2} \left( \alpha^{-1} * \eta_2 \right). 
\end{equation*} 
Since $\mu$ is extreme in $\mu *I$ and $\alpha^{-1} * \eta_i \in \mu \ast I$, we get $\mu = \alpha^{-1} * \eta_1 = \alpha^{-1} * \eta_2$, hence $\eta_1 = \eta_2$. 
\end{proof} 

We have shown that any ideal subgroup of $\mathfrak{M}_{x}^{\dagger}(\mathcal{G},G)$ is trivial. Since the minimal left ideals can be partitioned into their ideal subgroups, it follows that the convolution operation is trivial when restricted to a minimal left ideal.

\begin{proposition}\label{mult} Let $I$ be a minimal left ideal in $\mathfrak{M}_{x}^{\dagger}(\mathcal{G},G)$. Then every element of $I$ is an idempotent. Moreover, for any elements $\mu, \nu \in I$, we have that $\mu * \nu = \mu$. 
\end{proposition} 

\begin{proof} By Fact \ref{fact:Ellis}(4) and Proposition \ref{trivial},
\begin{equation*}
I = \bigsqcup_{\mu \in \id(I)} \mu * I = \bigsqcup_{\mu \in \id(I)} \{\mu\} = \id(I). 
\end{equation*} 
The ``moreover'' part also follows from the observation that $\mu * I = \{\mu\}$.  
\end{proof} 

\begin{corollary}\label{needed} Let $I$ be a minimal left ideal of $\mathfrak{M}_{x}^{\dagger}(\mathcal{G},G)$ and assume that $\mu \in I$. Then  $I = \{\nu \in \mathfrak{M}_{x}^{\dagger}(\mathcal{G},G): \nu * \mu = \nu\}$. 
\end{corollary} 

\begin{proof} By Proposition \ref{mult} we have $I \subseteq \{\nu \in \mathfrak{M}_{x}^{\dagger}(\mathcal{G},G): \nu * \mu = \nu\}$. And since $I$ is a left ideal and $\mu \in I$, we have $\{\nu \in \mathfrak{M}_{x}^{\dagger}(\mathcal{G},G): \nu * \mu = \nu\} \subseteq I$. 
\end{proof} 

\noindent We also observe that the action of the underlying group $G$ on the minimal left ideal is far from being a free action (this is of course trivial in the definably amenable case, but is meaningful when $\mathcal{G}$ is not definably amenable). 

\begin{proposition}\label{prop: G not free action} Let $I$ be a minimal left ideal of $\mathfrak{M}_{x}^{\dagger}(\mathcal{G},G)$. For any definable measure $\nu \in \mathfrak{M}^{\dagger}(\mathcal{G},G)$ there exists a measure $\mu \in I$ such that $\nu * \mu= \mu$. In particular, for every element $g \in G$, there exists a measure $\mu \in I$ such that $\delta_{g} * \mu = \mu$. 
\end{proposition} 

\begin{proof} Consider the map $\nu \ast -: \mathfrak{M}_{x}^{\dagger}(\mathcal{G},G) \to \mathfrak{M}_{x}^{\dagger}(\mathcal{G},G)$ sending $\lambda$ to $\nu \ast \lambda$.
Since $I$ is a minimal left ideal, the image of $(\nu*-)|_{I}$ is contained in $I$. Since $\nu$ is definable, the map $(\nu*-)|_{I}: I \to I$ is continuous by Lemma \ref{lem: right cont for def meas}. By Lemma \ref{convex:1}, this  map is also affine. 
By the Markov-Kakutani fixed point theorem, there exists some $\mu \in I$ such that $\nu * \mu = \mu$. The ``in particular'' part of the statement follows since $\delta_{g}, g \in G$ is a definable measure. 
\end{proof} 

\subsection{Definably amenable groups}

We now shift our focus to the dividing line of definable amenability. We first describe all minimal left ideals in both $\left( \mathfrak{M}^{\fs}_{x}(\mathcal{G},G), \ast \right)$ and $\left(\mathfrak{M}^{\inv}_{x}(\mathcal{G},G), \ast \right)$ when $\mathcal{G}$ is definably amenable. We then make an observation about what happens outside of the definably amenable case.
Recall that $T$ is a complete NIP theory expanding a group, $\mathcal{G}$ is a monster model of $T$, $G$ is a small elementary submodel of $\mathcal{G}$. The group $\mathcal{G}$ is \emph{definably amenable} if  there exists $\mu \in \mathfrak{M}_{x}(\mathcal{G})$ such that $\mu$ is $\mathcal{G}$-left-invariant.
\begin{remark}\label{rem: def amenable el equiv}
\begin{enumerate}
\item 	The group $\mathcal{G}$ is definably amenable if and only if for some $G' \models T$ there exists a $G'$-left-invariant $\mu \in \mathfrak{M}_x(G')$, if and only if for every $G' \models T$ there exists a $G'$-left-invariant $\mu \in \mathfrak{M}_x(G')$ (see \cite[Section 5]{hrushovski2008groups}).
\item If $G' \preceq \mathcal{G}$ and $\mu \in \mathfrak{M}_x(G')$ is $G'$-left-invariant, then the measure $\mu^{-1} \in \mathfrak{M}_x(G')$ defined by $\mu^{-1}(\varphi(x)) = \mu(\varphi(x^{-1}))$ for any $\varphi(x) \in \mathcal{L}_{x}(G')$ is a $G'$-right-invariant, and vice versa.
If $\mu \in \mathfrak{M}_{x}^{\dagger}(\mathcal{G},G)$, then also $\mu^{-1} \in \mathfrak{M}_{x}^{\dagger}(\mathcal{G},G)$ (see \cite[Lemma 6.2]{ChSi}).
\end{enumerate}

\end{remark} 
\noindent We will need the following fact.

\begin{fact}\label{meas:exist} Assume that $\mathcal{G}$ is definably amenable and NIP.
\begin{enumerate}[$(i)$] 
\item \cite[Proposition 3.5]{ChPiSi} For any $G$-left-invariant measure $\mu_0 \in \mathfrak{M}_x(G)$ (which exists by Remark \ref{rem: def amenable el equiv}(1)) there exists $\mu \in \mathfrak{M}_{x}^{\inv}(\mathcal{G},G)$ such that $\mu$ is $\mathcal{G}$-left-invariant and extends $\mu_0$. The same holds for right-invariant measures by Remark \ref{rem: def amenable el equiv}(2).
\item \cite[Theorem 3.17]{ChPiSi}  There exists $\nu \in \mathfrak{M}_{x}^{\fs}(\mathcal{G},G)$ such that $\nu$ is $G$-left-invariant (but not necessarily $\mathcal{G}$-left-invariant). 
\end{enumerate} 
\end{fact} 

 We remark that Fact \ref{meas:exist}(ii) follows from \cite[Theorem 3.17]{ChPiSi} as $S_{x}^{\fs}(\mathcal{G},G) = S_{x}(G^{\ext}) $ (where $G^{\ext}$ is the Shelah's expansion of $G$ by all externally definable subsets) and $\mathfrak{M}\left(S_{x}^{\fs}(\mathcal{G},G) \right) = \mathfrak{M}_{x}^{\fs}(\mathcal{G},G)$ (see Corollary \ref{cor:homeomorphism}). We now compute the minimal left ideals in definably amenable NIP groups, first in the finitely satisfiable case and then in the invariant case.

\begin{proposition}\label{fs:I}  The group $\mathcal{G}$ is definably amenable if and only if $|I| = 1$ for some (equivalently, every) minimal left ideal $I$ in $\mathfrak{M}_{x}^{\fs}(\mathcal{G},G)$. And if $\mathcal{G}$ is definably amenable, then the minimal left ideals of $\mathfrak{M}_{x}^{\fs}(\mathcal{G},G)$ are precisely of the form $\{\nu\}$ for $\nu$  a $G$-left-invariant measure in $\mathfrak{M}_{x}^{\fs}(\mathcal{G},G)$. 


\end{proposition} 
\begin{proof}
Let $I$ be a minimal left ideal, and assume that $I = \{ \mu \}$. Then for any $g \in G$ we have $g \cdot \mu = \delta_{g} \ast \mu = \mu$, so $\mu$ is $G$-left-invariant. In particular, $\mu|_{G}$ is a $G$-left-invariant measure on $\mathfrak{M}_x(G)$, hence $\mathcal{G}$ is definably amenable by Remark \ref{rem: def amenable el equiv}(1). And all minimal left ideals have the same cardinality by Fact \ref{fact:Ellis}(6).

Conversely, assume that $\mathcal{G}$ is definably amenable. By Fact \ref{meas:exist}(2) there exists some $\mu \in \mathfrak{M}_{x}^{\fs}(\mathcal{G},G)$ such that $\mu$ is $G$-left-invariant. We claim that for any such $\mu$, $\{\mu\}$ is a minimal left ideal of $\mathfrak{M}_{x}^{\fs}(\mathcal{G},G)$.  Let $\nu$ be any measure in $\mathfrak{M}^{\fs}_{x}(\mathcal{G},G)$. Since $\nu$ is finitely satisfiable in $G$, by Lemma \ref{lem: limit} there exists a net of measures in $\mathfrak{M}^{\fs}_x(\mathcal{G},G)$ of the form $\left(\Av(\overline{a}_{i})\right)_{i \in I}$, such that each $\overline{a}_{i} = (a_{i,1}, \ldots, a_{i, n_i}) \in (G^{x})^{n_i}$ for some $n_i \in \mathbb{N}$ and $\lim_{i \in I} \left(\Av(\overline{a}_{i})\right) = \nu$. Fix any $\varphi(x) \in \mathcal{L}_x(\mathcal{G})$. By the ``moreover'' part of Fact \ref{fac: conv is a semigroup in NIP}, the map $\lambda \in \mathfrak{M}^{\fs}_{x}(\mathcal{G},G) \mapsto \left(\lambda \ast \mu \right)(\varphi(x)) \in [0,1]$ is continuous. Therefore 
\begin{gather*} \left(\nu * \mu \right)(\varphi(x)) = \lim_{i \in I} \big( \left(\Av(\overline{a}_{i}) * \mu \right)(\varphi(x))\big)\\
 = \lim_{i \in I} \left( \frac{1}{n_i}\sum_{j=1}^{n_{i}} \mu(\varphi(a_{i,j} \cdot x))\right) \overset{(a)}{=} \lim_{i \in I} \mu(\varphi(x)) = \mu(\varphi(x)). 
\end{gather*}
 Equality $(a)$ follows as $\mu$ is $G$-left-invariant and each $a_{i,j}$ is in $G$. It follows that $\nu \ast \mu = \mu$, hence $\{\mu\}$ is a left ideal.
\end{proof} 

We now compute the minimal left ideals in the invariant case, but first we record an auxiliary lemma. 

\begin{lemma}\label{quick1} Assume that $f:S_{x}(G) \to [0,1]$ is a Borel function. For any $b \in G$, we define the function $f\cdot b: S_x(G) \to [0,1]$ via $(f \cdot b)(p) := f(p \cdot b)$ (recall Lemma \ref{cont:lemma}). If $\mu \in \mathfrak{M}_x(G)$ is $G$-right-invariant then 

\begin{equation*} 
\int_{S_{x}(G)} f d\mu = \int_{S_{x}(G)} (f \cdot b) d\mu.
\end{equation*} 
\end{lemma} 

\begin{proof} For $b \in G$, consider the map $\gamma_{b}:S_{x}(G) \to S_{x}(G)$ defined by $\gamma_{b}(p) := p \cdot b$. The map $\gamma_{b}$ is a continuous bijection. Hence we can consider the push-forward map $(\gamma_{b})_{*}:\mathfrak{M}_{x}(G) \to \mathfrak{M}_{x}(G)$. Denote $(\gamma_{b})_{*}(\mu)$ as $\mu_{b}$. Fix a formula $\varphi(x) \in \mathcal{L}_{x}(G)$. We claim that $\gamma_{b}^{-1}([\varphi(x)]) = [\varphi(x\cdot b)]$. 

We first show that $(\gamma_{b})^{-1}([\varphi(x)]) = [\varphi(x \cdot b)]$. Assume that $p \in [\varphi(x \cdot b)]$, then $\varphi(x) \in p \cdot b$ and  so $p \cdot b \in [\varphi(x)]$. Hence $(\gamma_{b})^{-1}(p \cdot b) \in (\gamma_{b}^{-1})([\varphi(x)])$. Since $\gamma_{b}$ is a bijection, we have that $p = \gamma_{b}^{-1}(p \cdot b)$, which implies that $p \in (\gamma_{b}^{-1})([\varphi(x)])$. So $[\varphi(x \cdot b)] \subseteq (\gamma_{b})^{-1}([\varphi(x)])$. Now assume that $p \in (\gamma_{b})^{-1}([\varphi(x)])$. Then $\gamma_{b}(p) \in [\varphi(x)]$, hence $p \cdot b \in [\varphi(x)]$ and so $\varphi(x) \in p \cdot b$. By definition $\varphi(x\cdot b) \in (p \cdot b) \cdot b^{-1}$,  and since $(p \cdot b) \cdot b^{-1} = p$, we conclude that $\varphi(x\cdot b) \in p$. Hence $p \in [\varphi(x \cdot b)]$ and $(\gamma_{b})^{-1}([\varphi(x)]) = [\varphi(x\cdot b)]$.

Now we show that $\mu_b = \mu$. Indeed, by $G$-right invariance of $\mu$ and the previous paragraph we have 
\begin{equation*} \mu_{b}(\varphi(x)) = \mu(\gamma_{b}^{-1}[\varphi(x)]) = \mu(\varphi(x \cdot b)) = \mu(\varphi(x)). 
\end{equation*} 
And so by Fact \ref{meas:facts}(iii) we have 
\begin{equation*}\int_{S_{x}(G)} f d\mu =  \int_{S_{x}(G)} f d\mu_{b} = \int_{S_{x}(G)} \left( f \circ \gamma_{b} \right) d\mu = \int_{S_{x}(G)} (f \cdot b) d\mu. \qedhere
\end{equation*} 
\end{proof}

%

\begin{proposition}\label{ideal:inv} Assume that $\mathcal{G}$ is definably amenable. Let 
$$I_{G}^{\inv} := \left\{\mu \in \mathfrak{M}^{\inv}_{x}(\mathcal{G},G): \text{ $\mu$ is $\mathcal{G}$-right-invariant}\right\}.$$
 Then $I_{G}^{\inv}$ is a closed, non-empty, two-sided ideal. Moreover, $I_{G}^{\inv}$ is the unique minimal left ideal in $\mathfrak{M}_{x}(\mathcal{G},G)$. 
\end{proposition}

\begin{proof} The set $I_{G}^{\inv}$ is closed since it is the complement of the  union of basic open sets in $\mathfrak{M}^{\inv}_{x}(\mathcal{G},G)$:
\begin{gather*}
	\mathfrak{M}^{\inv}_{x}(\mathcal{G},G) \setminus I_G^{\inv} = \\
	\bigcup_{\varphi(x) \in \mathcal{L}_x(\mathcal{G})} \bigcup_{s < t \in [0,1]} \bigcup_{g \in \mathcal{G}}\left( \{\mu :  \mu(\varphi(x)) < s\} \cap  \{\mu :  \mu(\varphi(x \cdot g)) > t\} \right).
\end{gather*}
By Fact \ref{meas:exist}(1), we know that the set $I_{G}^{\inv}$ is non-empty. We first show that $I_{G}^{\inv}$ is a left ideal. Let $\mu \in I_{G}^{\inv}$ and $\nu \in \mathfrak{M}^{\inv}_{x}(\mathcal{G},G)$. It suffices to show that the measure $\nu *\mu$ is $\mathcal{G}$-right-invariant. That is, we need to show that for any $\varphi(x) \in \mathcal{L}_{x}(\mathcal{G})$ and $b \in \mathcal{G}$ we have  $\left( \nu * \mu \right)(\varphi(x \cdot b)) = \left(\nu * \mu \right)(\varphi(x))$. Let $G' \prec \mathcal{G}$ be a small model containing $G$, $b$ and the parameters of $\varphi$. For any $q \in S_{y}(G')$ and $a \models q$ in $\mathcal{G}$, letting $\varphi_b(x) := \varphi(x \cdot b)$ and noting that $a \cdot b \models q \cdot b$, we have
\begin{equation*} F_{\nu,G'}^{\varphi'_{b}}(q) = \nu(\varphi(x\cdot a \cdot b)) = F_{\nu,G'}^{\varphi'}(q \cdot b) = \left(F_{\nu,G'}^{\varphi'}(q) \right) \cdot b. 
\end{equation*}  
Hence, by Lemma \ref{quick1},
\begin{gather*} \left(\nu * \mu \right)(\varphi(x \cdot b)) = \int_{S_{y}(G')} F_{\nu}^{\varphi_{b}'} d\mu_{G'} = \\
\int_{S_{y}(G')} \left( \left(F_{\nu}^{\varphi'}\right) \cdot b \right)d\mu_{G'} = \int_{S_{y}(G')} F_{\nu}^{\varphi'} d\mu_{G'} = \left( \nu * \mu \right)(\varphi(x)).
\end{gather*} 

We now argue that $I_{G}^{\inv}$ is a right ideal. Again let $\mu \in I_{G}^{\inv}$, $\nu \in \mathfrak{M}^{\inv}_{x}(\mathcal{G},G)$, and fix $\varphi(x) \in \mathcal{L}_{x}(\mathcal{G})$ and $G' \prec \mathcal{G}$  containing $G$ and the parameters of $\varphi$. Using $\mathcal{G}$-right-invariance of $\mu$ we have
\begin{equation*} 
\left( \mu * \nu \right)(\varphi(x)) = \int_{S_{y}(G')} F_{\mu}^{\varphi'} d\nu_{G'} = \int_{S_{y}(G')} \mu(\varphi(x)) d\nu_{G'} =  \mu(\varphi(x)). 
\end{equation*} 
Hence $I_{G}^{\inv}$ is a two-sided ideal.

Note that the previous computation shows that $\mu * \nu = \mu$ for any $\mu \in I_{G}^{\inv}$ and $\nu \in \mathfrak{M}_{x}^{\inv}(\mathcal{G},G)$. So if $J$ is any minimal left ideal of $\mathfrak{M}_{x}^{\inv}(\mathcal{G},G)$, then $I_{G}^{\inv} \subseteq J$. Since $I_{G}^{\inv}$ is two-sided, we have that $J \subseteq I_{G}^{\inv}$ (by Lemma \ref{two-sided}). Hence $J = I_{G}^{\inv}$ and $I_{G}^{\inv}$ is the unique minimal left ideal. 
\end{proof} 


We recall some terminology and results from \cite{ChSi} (switching from the action on the left to the action on the right everywhere).
\begin{definition}
	\begin{enumerate}
		\item A type $p \in S_x(\mathcal{G})$ is \emph{right $f$-generic} if  for every $\varphi(x) \in p$ there is some small model $G \prec \mathcal{G}$ such that for any $g\in \mathcal{G}$, $\varphi(x \cdot g)$ does not fork over $G$.
		\item A type $p \in S_x(\mathcal{G})$ is \emph{strongly right $f$-generic} if there exists some small $G \prec \mathcal{G}$ such that $p \cdot g \in S_x^{\inv}(\mathcal{G},G)$ for all $g \in \mathcal{G}$. This is equivalent to the  definition in \cite{ChSi} since in NIP theories, a global type $p$ does not fork over a model $M$ if and only if $p$ is $M$-invariant (see e.g.~\cite[Proposition 2.1]{NIP2}). 
		\item Given a right $f$-generic $p$, let $\mu_p$ be defined via 
		$$\mu_p(\varphi(x)) := h \left( \left\{ \pi(g) \in \mathcal{G}/\mathcal{G}^{00} : g \in \mathcal{G}, \varphi(x) \in  p \cdot g \right\} \right),$$
		where $\pi: \mathcal{G} \to \mathcal{G}/\mathcal{G}^{00}$ is the quotient map and $\varphi(x) \in \mathcal{L}_x(\mathcal{G})$.
		Then $\mu_p \in \mathfrak{M}_x(\mathcal{G})$ and, assuming additionally that $\mathcal{G}$ is definably amenable, $\mu_p$ is $\mathcal{G}^{00}$-right-invariant (see \cite[Definition 3.16]{ChSi} for the details).
	\end{enumerate}
\end{definition}
\begin{fact}\label{fac: erg meas in def am NIP} Assume that $\mathcal{G}$ is definably amenable, NIP.
	\begin{enumerate}
		\item If $p \in S_x^{\inv}(\mathcal{G},G)$ is right $f$-generic then $p$ is strongly right $f$-generic over $G$ and $\mu_p \in \mathfrak{M}_x^{\inv}(\mathcal{G}, G)$. The set of all right $f$-generic types in $S_x(\mathcal{G})$ (and hence in $S_x^{\inv}(\mathcal{G},G)$) is closed.
		\item Let $\mathfrak{I}(\mathcal{G})$ be the (closed convex) set of all $\mathcal{G}$-right invariant measures in $\mathfrak{M}_x(\mathcal{G})$.
		Then the set $\ex(\mathfrak{I}(\mathcal{G}))$ of the extreme points of $\mathfrak{I}(\mathcal{G})$ is the set of all measures of the form $\mu_p$ for some right $f$-generic $p \in S_x(\mathcal{G})$.
		\item The map $p \mapsto \mu_p$ from the (closed) set of global right $f$-generic types to the (closed) set of global $\mathcal{G}$-right-invariant measures is continuous.
	\end{enumerate}
\end{fact}
\begin{proof}
	(1) Any $f$-generic $p \in S_x^{\inv}(\mathcal{G},G)$ is strongly $f$-generic over $G$ by \cite[Proposition 3.9]{ChSi}. For any $f$-generic $p$, $\sup(\mu_p) \subseteq \overline{p \cdot \mathcal{G}}$, where $\overline{X}$ is the topological closure of $X$ in $S_x(\mathcal{G})$ and $p \cdot \mathcal{G} = \{ p \cdot g \in S_x(\mathcal{G}): g \in \mathcal{G} \}$ is the orbit of $p$ under the right action of $\mathcal{G}$ (by \cite[Remark 3.17(2)]{ChSi}). As $p$ is strongly $f$-generic over $G$, we have $p \cdot \mathcal{G} \subseteq S_x^{\inv}(\mathcal{G},G)$, hence $\sup(\mu_p) \subseteq \overline{S_x^{\inv}(\mathcal{G},G)} = S_x^{\inv}(\mathcal{G},G)$. Hence $\mu_p \in \mathfrak{M}_x^{\inv}(\mathcal{G}, G)$ by Fact \ref{sup:inv}(2).
	
	(2) is \cite[Theorem 4.5]{ChSi}, and (3) is \cite[Proposition 4.3]{ChSi}.
\end{proof}

\noindent Adapting the proof of \cite[Theorem 4.5]{ChSi}, we can describe the extreme points of the minimal ideal $I_{G}^{\inv}$.
\begin{corollary}\label{cor: ex inv f-generic}
Assume that $\mathcal{G}$ is definably amenable NIP. Then:
\begin{enumerate}
	\item $\ex \left(I_{G}^{\inv} \right)= \left\{ \mu_p : p \in S_x^{\inv}(\mathcal{G},G) \textrm{ is right }f\textrm{-generic} \right\}$;
	\item $\ex \left(I_{G}^{\inv} \right)$ is a closed subset of $I_{G}^{\inv}$, and $I_{G}^{\inv}$ is a Bauer simplex.
\end{enumerate}
\end{corollary}
\begin{proof}
If $p \in S_x^{\inv}(\mathcal{G},G)$ is right $f$-generic, then $\mu_p$ is $\mathcal{G}$-right-invariant and $\mu_p \in \mathfrak{M}_x^{\inv}(\mathcal{G}, G)$ by Fact \ref{fac: erg meas in def am NIP}(1), hence $\mu_p \in I_{G}^{\inv}$. By Fact \ref{fac: erg meas in def am NIP}(2), $\mu_p$ is extreme in $\mathfrak{I}(\mathcal{G})$, hence in particular it is extreme in $I_{G}^{\inv} \subseteq \mathfrak{I}(\mathcal{G})$.
 
Conversely, assume that $\mu \in \ex\left(I_{G}^{\inv} \right)$, and let 
$$S := \left\{ \mu_p : p \in S_x^{\inv}(\mathcal{G},G) \textrm{ is right }f\textrm{-generic} \right\}.$$ Let $\overline{\conv}(S)$ be the closed convex hull of $S$, then $\overline{\conv}(S) \subseteq I_{G}^{\inv}$ by Propositions \ref{convex:compact} and \ref{ideal:inv}. As $\mu$ is $\mathcal{G}$-right-invariant, by \cite[Lemma 3.26]{ChSi}, for any $\varepsilon > 0$ and $\varphi_1(x), \ldots, \varphi_k(x) \in \mathcal{L}_x(\mathcal{G})$, there exist some right $f$-generic $p_1, \ldots, p_n \in \sup(\mu)$ such that $\mu(\varphi_j(x)) \approx_{\varepsilon} \frac{1}{n}\sum_{i=1}^{n} \mu_{p_i}\left( \varphi_j(x) \right)$ for all $j \in [k]$. 
While \cite[Lemma 3.26]{ChSi} is stated for a single formula, it also applies to finitely many formulas by encoding them as appropriate instances of a single formula --- formally, we apply \cite[Lemma 3.26]{ChSi} to the formula 
\begin{equation*} 
\theta(x ; y_0, \ldots, y_k) : = \bigvee_{i=1}^{k} \left( y_0 = y_i \land  \varphi_{k}(x) \right). 
\end{equation*} 
 As we have $p_i \in S^{\inv}_x(\mathcal{G},G)$ for all $i \in [n]$, by Fact \ref{sup:inv}(2), it follows that $\mu \in \overline{\conv}(S)$, and it is still an extreme point of $\overline{\conv}(S) \subseteq I_{G}^{\inv}$. It follows that $\mu \in \overline{S}$, by the (partial) converse to the Krein-Milman theorem (see e.g.~\cite[Fact 4.1]{ChSi} applied to $C:= \overline{\conv}(S)$). By Fact \ref{fac: erg meas in def am NIP}(3), the map $p \mapsto \mu_p$ from $S_x^{\inv}(\mathcal{G},G)$ to $\mathfrak{M}_x^{\inv}(\mathcal{G}, G)$ is a continuous map from a compact to a Hausdorff space, hence also a closed map. It follows that $S = \overline{S}$, so $\mu \in S$.

By Corollary \ref{cor:homeomorphism}(2), we have an affine homeomorphism  between $\mathfrak{M}_x^{\inv}(\mathcal{G}, G)$ and $\mathcal{M} \left( S_x^{\inv}(\mathcal{G}, G) \right)$, which restricts to an affine homeomorphism between $I^{\inv}_{G}$ and the set $\mathcal{M}_{\mathcal{G}} \left( S_x^{\inv}(\mathcal{G}, G) \right)$ of all right-$\mathcal{G}$-invariant regular Borel probability measures on $S_x^{\inv}(\mathcal{G}, G)$. By Fact \ref{fac: inv measures simplex}, $\mathcal{M}_{\mathcal{G}} \left( S_x^{\inv}(\mathcal{G}, G) \right)$ is a Choquet simplex, hence $I^{\inv}_{G}$ is a Bauer simplex (using Remark \ref{rem: aff hom simplex}).
\end{proof}

\begin{question}
	Can every Bauer simplex of the form $\mathcal{M}(X)$ with $X$ a compact Hausdorff totally disconnected space be realized as a minimal left ideal of $\left( \mathfrak{M}^{\inv}_{x}(\mathcal{G},G), \ast \right)$ for some definably amenable NIP group $\mathcal{G}$?
\end{question}

\begin{example} Let $G := (\mathbb{R};<,+)$ and $\mathcal{G} \succ G$ a monster model. As $\mathcal{G}$ is abelian, it is amenable as a discrete group and hence definably amenable. By Proposition \ref{ideal:inv}, $\mathfrak{M}_{x}^{\inv}(\mathcal{G},\mathbb{R})$ has a unique minimal left ideal $I_{G}^{\inv}$. One checks directly that $p_{-\infty}$ (the unique type extending $\{x < a: a \in \mathcal{G}\}$) and $p_{+\infty}$ (the unique type extending $\{x > a: a \in \mathcal{G}\}$) are the right $f$-generics in $S^{\inv}_x \left(\mathcal{G},G \right)$, and $\mu_{p_{+\infty}} = \delta_{p_{+\infty}}, \mu_{p_{-\infty}} = \delta_{p_{-\infty}}$. Hence, by Corollary \ref{cor: ex inv f-generic}, $|\ex(I_{G}^{\inv})| = 2$ and 
\begin{equation*} I_{G}^{\inv} = \left\{r\delta_{p_{+\infty}} + (1-r) \delta_{p_{-\infty}}: r \in [0,1]\right\}.
\end{equation*} 
(See also Example \ref{ex: CIG2 examples ideal}(1).)
\end{example} 

Recall that $\mathcal{G}$ is \emph{uniquely ergodic} if it admits a unique $\mathcal{G}$-left-invariant measure $\mu \in \mathfrak{M}_x(\mathcal{G})$ (see \cite[Section 3.4]{ChSi}). Recall that $\mathcal{G}$ is \emph{fsg} if there exists a small $G \prec \mathcal{G}$ and $p \in S_x(\mathcal{G})$ such that $g \cdot p$ is finitely satisfiable in $G$ for all $g \in \mathcal{G}$. All fsg groups are uniquely ergodic (see e.g.~\cite[Proposition 8.32]{Sibook}), but there exist uniquely ergodic NIP groups which are not fsg (see \cite[Remark 3.38]{ChSi}).

\begin{corollary}\label{cor: fsg}
\begin{enumerate}
	\item If $\mathcal{G}$ is uniquely ergodic, then  $I^{\inv}_G = \{ \mu \}$, where $\mu$ is the unique $\mathcal{G}$-left-invariant measure. 
	\item If $\mathcal{G}$ is moreover \emph{fsg}, letting $\mu \in \mathfrak{M}_{x}(\mathcal{G})$ be the unique $\mathcal{G}$-left-invariant measure, $\{\mu \}$ is the unique minimal left ideal of $\mathfrak{M}_{x}^{\fs}(\mathcal{G},G)$ (which is also two-sided).
\end{enumerate}
	
\end{corollary}
\begin{proof}
(1) For any $\mathcal{G}$-left-invariant measure $\mu$, the measure $\mu^{-1}$ is $\mathcal{G}$-right-invariant (see Remark \ref{rem: def amenable el equiv}(2)), and vice versa. Moreover, from the definition, $\mu_1 = \mu_2$ if and only if $\mu_1^{-1} = \mu_2^{-1}$. It follows that if there exists a unique $\mathcal{G}$-left-invariant measure $\mu$, then there exists a unique $\mathcal{G}$-right-invariant measure $\mu^{-1}$.
	By \cite[Lemma 6.2]{ChSi} there also exists a measure $\nu$ which is simultaneously $\mathcal{G}$-left-invariant and $\mathcal{G}$-right-invariant. But then $\mu = \nu = \mu^{-1}$, hence $\mu$ is also $\mathcal{G}$-right-invariant. And $\mu \in \mathfrak{M}_x^{\inv}(\mathcal{G}, G)$ by Fact \ref{meas:exist} and uniqueness, hence $I_G^{\inv} = \{ \mu\}$.
	
	(2) By e.g.~\cite[Propositions 8.32, 8.33]{Sibook}, $\mathcal{G}$ is fsg if and only if there exists a $\mathcal{G}$-left-invariant generically stable measure $\mu \in \mathfrak{M}_{x}(\mathcal{G})$, and then $\mathcal{G}$ is uniquely ergodic, hence $\mu$ is also the unique $\mathcal{G}$-right-invariant measure. By Fact \ref{meas:exist}(i) and uniqueness of $\mu$ it follows that $\mu$ is invariant over $G$, hence generically stable over $G$ (in fact, over an arbitrary small model). In particular $\mu \in \mathfrak{M}_{x}^{\fs}(\mathcal{G},G)$, and it is the unique measure in $\mathfrak{M}_{x}^{\inv}(\mathcal{G},G)$ extending $\mu|_G$ (by \cite[Proposition 3.3]{NIP3}). 	Now assume that $\nu \in \mathfrak{M}_{x}^{\fs}(\mathcal{G},G)$ is an arbitrary $G$-left-invariant measure. We have $\nu|_{G} = \mu|_G$, as by Fact \ref{meas:exist}(i) there exists some $\mathcal{G}$-left-invariant $\nu'$ extending $\nu|_G$, hence $\nu' = \mu$, so $\nu|_G = \nu'|_{G} = \mu|_{G}$. But as $\mu$ is the unique measure in $\mathfrak{M}_{x}^{\inv}(\mathcal{G},G)$ extending $\mu|_G$, it follows that $\nu=\mu$. If follows by Proposition \ref{fs:I} that $\{\mu\}$ is the unique minimal left ideal of $\mathfrak{M}_{x}^{\fs}(\mathcal{G},G)$. 
	Finally, in any semigroup, if the union of its minimal left ideals is non-empty, then it is a two sided ideal \cite{clifford1948semigroups}. Hence in our case $\{\mu\}$ is a two-sided ideal.
\end{proof}

%

%

\begin{question}
	Can the fsg assumption be relaxed to unique ergodicity in Corollary \ref{cor: fsg}(2)?
\end{question}

Our final observation in this section deals with non-definably amenable groups.  

\begin{remark}\label{prop: many idemp in non def am} Assume that $\mathcal{G}$ is not definably amenable. Let $I$ be a minimal left ideal in $\mathfrak{M}^{\dagger}(\mathcal{G},G)$. Then $\ex(I)$ is infinite. 
\end{remark} 

\begin{proof} For any  $g \in G$, the map $\delta_{g} * -: \ex(I) \to \ex(I)$ is a bijection. Towards a contradiction, assume that $\ex(I)$ is finite, say $\ex(I) = \{\mu_1, \ldots, \mu_n\}$. Consider the measure $\lambda \in \mathfrak{M}_{x}^{\dagger}(\mathcal{G},G)$ defined by $\lambda = \sum_{i=1}^{n} \frac{1}{n}\mu_i$. Then for any $g \in G$ we have $\delta_{g} * \lambda = \lambda$. Hence the measure $\lambda|_{G}$ is in $\mathfrak{M}_{x}(G)$ and is $G$-left-invariant. This contradicts the assumption that $\mathcal{G}$ is not definably amenable by (1) and (2) of Remark \ref{rem: def amenable el equiv}. 
\end{proof}

\section{Constructing minimal left ideals}\label{sec: CIGS}


In this section, under some assumptions on the semigroup $(S_{x}^{\dagger}(\mathcal{G},G),*)$ (applicable to some non-definably amenable groups, e.g. $\textrm{SL}_2(\mathbb{R})$), we construct a minimal left ideal of $\left( \mathfrak{M}_{x}^{\dagger}(\mathcal{G},G), \ast \right)$ using a minimal left ideal and an ideal subgroup of $(S_{x}^{\dagger}(\mathcal{G},G),*)$, and demonstrate that this minimal left ideal is parameterized by a space of regular Borel probability measures over a compact Hausdorff space.

\subsection{Basic lemmas}  We will need some  auxiliary lemmas connecting convolution and left ideals. We assume that $T = \Th(\mathcal{G})$ is NIP throughout. 

\begin{lemma}\label{eats:measures} Let $\mu, \nu \in \mathfrak{M}_{x}^{\dagger}(\mathcal{G},G)$. If $\mu * \delta_{p} = \mu$ for every $p \in \sup(\nu)$, then $\mu * \nu = \mu$.
\end{lemma} 

\begin{proof} Fix a formula $\varphi(x) \in \mathcal{L}_{x}(\mathcal{G})$. Let $G' \prec \mathcal{G}$ be a small model containing $G$ and the parameters of  $\varphi$. We have
\begin{equation*} 
\left( \mu * \nu \right)(\varphi(x)) = \int_{\sup \left(\nu|_{G'} \right)} F_{\mu,G'}^{\varphi'} d(\nu_{G'}). 
\end{equation*} 
By Fact \ref{sup:inv}, $\sup(\nu)$ is a subset of $S_{x}^{\dagger}(\mathcal{G},G)$. For any $q \in \sup(\nu)$ we have  $F_{\mu,G'}^{\varphi'}(q) = \mu(\varphi(x \cdot b)) = \left( \mu * \delta_{p} \right) (\varphi(x)) = \mu(\varphi(x))$, where $b \models q$. Hence
\begin{equation*} 
\int_{\sup(\nu|_{G'})} F_{\mu,G'}^{\varphi'} d(\nu_{G'}) = \int_{S_{y}(G')} \mu(\varphi(x)) d(\nu_{G'})  = \mu(\varphi(x)),
\end{equation*} 
so $\mu * \nu = \mu$. 
\end{proof} 

\begin{lemma}[$T$ is NIP]\label{ideal} Assume that $I$ is a left ideal of $\left( S_{x}^{\dagger}(\mathcal{G},G), \ast \right)$. Then $\mathfrak{M}(I)$ (see Definition \ref{def: meas on a set of types}) is a left ideal of $\left(\mathfrak{M}_{x}^{\dagger}(\mathcal{G},G), \ast \right)$. 
\end{lemma} 

\begin{proof} Let $p \in S_{x}^{\dagger}(\mathcal{G},G)$ and $\mu \in \mathfrak{M}(I)$. We first argue that $\delta_{p} * \mu \in \mathfrak{M}(I)$. Assume towards a contradiction that $\delta_{p} * \mu \not \in \mathfrak{M}(I)$. Then there exists some $q \in \sup(\delta_{p} * \mu)$ such that $q \not \in I$. Then there exists $\psi(x) \in \mathcal{L}_x(\mathcal{G})$ such that $\psi(x) \in q$ and $[\psi(x)] \cap I = \emptyset$. Since $\psi(x) \in q$ and $q \in \sup(\delta_{p} * \mu)$, we have $\left( \delta_{p} * \mu \right)(\psi(x)) > 0$. Let now $G' \prec \mathcal{G}$ be a small model containing $G$ and the parameters of $\psi$. Then 
\begin{equation*}  
\left( \delta_{p} * \mu \right)(\psi(x)) = \int F_{\delta_{p},G'}^{\psi'} d(\mu_{G'}) > 0, 
\end{equation*} 
so there exists some $t \in \sup(\mu|_{G'})$ such that $F_{\delta_{p},G'}^{\psi'}(t) =1$. Fix $\hat{t} \in \sup(\mu)$ such that $\hat{t}|_{G'} = t$ (exists by e.g.~\cite[Proposition 2.8]{ChGan}), as $\mu \in \mathfrak{M}(I)$ we have $\hat{t} \in \supp(\mu) \subseteq I \subseteq S_{x}^{\dagger}(\mathcal{G},G)$. Unpacking the notation, we conclude that $\psi(x) \in p * \hat{t}$. Since $\hat{t} \in  I$, it also follows that $p * \hat{t} \in I$. Hence $[\psi(x)] \cap I \neq \emptyset$, a contradiction. 

Now let $\nu \in \mathfrak{M}_{x}^{\dagger}(\mathcal{G},G)$, we want to show that $\nu * \mu \in \mathfrak{M}(I)$. By Lemma \ref{lem: limit}, we have that $\nu = \lim_{i \in I} \Av(\overline{p}_i)$ for some net $I$, with $\overline{p}_i = (p_{i,1}, \ldots, p_{i,n_i}) \in I^{n_i}, n_i \in \mathbb{N}$ for each $i \in I$. By left continuity of convolution (Fact \ref{fac: conv is a semigroup in NIP}) we have 
\begin{equation*}
\nu * \mu = \lim_{i \in I} \left( \Av(\overline{p}_i) * \mu \right) = \lim_{i \in I} \left( \frac{1}{n_{i}} \sum_{j=1}^{n_{i}} \left( \delta_{p_{j}} * \mu \right) \right).
\end{equation*} 
By the previous paragraph $\delta_{p_{j_{i}}} *\mu \in \mathfrak{M}(I)$ for each $i \in I$. Then by convexity of $\mathfrak{M}(I)$ (Lemma \ref{closed}), also  $\Av(\overline{p}_{i}) * \mu \in \mathfrak{M}(I)$ for each $i \in I$. Since $\mathfrak{M}(I)$ is closed (again, Lemma \ref{closed}), $ \nu \ast \mu =\lim_{i \in I} \left(\Av(\overline{p}_{i}) * \mu \right) \in \mathfrak{M}(I)$. Therefore $\mathfrak{M}(I)$ is a left ideal. 
\end{proof} 

\begin{remark}\label{min:circle} We remark that minimality of the left ideal need not be preserved in Lemma \ref{ideal}. Indeed, let $G := \left(S^{1},\cdot, ^{-1},C(x,y,z) \right)$ be the standard unit circle group over $\mathbb{R}$, with $C$ the cyclic clockwise ordering, and let $T_{O}$ be the corresponding theory. If $\mathcal{G}$ is a monster model of $T_{O}$, then the semigroup $\left(S_{x}^{\fs}(\mathcal{G},S^{1}),* \right)$ has a unique proper (and hence minimal) left ideal $I := S_{x}^{\fs}(\mathcal{G},S^{1}) \setminus \left\{ \tp \left(a/\mathcal{G} \right) : a \in S^1 \right\}$. Let $\lambda$ be the Keisler measure corresponding to the  normalized Haar measure on $S^{1}$, $\lambda$ is smooth and right invariant, in particular $\mathcal{G}$ is fsg (see \cite[Example 4.2]{ChGan} and \cite[Proposition 8.33]{Sibook}). 
By Lemma \ref{ideal}, $\mathfrak{M}(I)$ is a left ideal of $\left(\mathfrak{M}_{x}^{\fs}(\mathcal{G},G), \ast \right)$. 
Note that $\mathfrak{M}_{x}^{\fs}\left(\mathcal{G},S^{1} \right)$ contains a unique minimal left ideal $\{\lambda\}$ by Corollary \ref{cor: fsg}(2), and $\{\lambda\}\subsetneq \mathfrak{M}(I)$ since the latter contains $\delta_{p}$ for every global type $p$ finitely satisfiable in $S^1$ but not realized in it.
\end{remark} 
We now recall how the ideal subgroups act on a minimal left ideal. The following is true in any compact left topological semigroup, we include a proof for completeness in our setting. 
\begin{corollary}\label{action}  Let $I$ be a minimal left ideal in $S_{x}^{\dagger}(\mathcal{G},G)$ and $u$ an idempotent in $I$. Let $p$ be any element in $I$. Then the map $\left(-*p \right)|_{u * I}: u*I \to u*I$ is a continuous bijection. Moreover, $\left( -*p \right)|_{u * I} = \left( -*(u*p) \right)|_{u*I}$.
\end{corollary} 
\begin{proof}
	We have $\left( u \ast I \right) \ast p =  u \ast \left(I \ast p \right) = u \ast I$ as $I \ast p = I$ by Fact \ref{fact:Ellis}(5) (using Fact \ref{fac: prod of types Ellis}).
Surjectivity: fix $r \in u \ast I$; as $u \ast p \in u \ast I$ and $u \ast I$ is a group with identity $u$, there exists some $s \in u \ast I$ such that $s \ast (u \ast p) = u$; then $r \ast s \in u \ast I$, and $(r \ast s) \ast p = (r \ast s \ast u) \ast p = r \ast (s \ast (u \ast p)) = r \ast u = r$. Injectivity: assume $r \ast p = t \ast p$ for some $r,t \in u \ast I$; as also $r \ast u = r, t \ast u = t$, we have $r \ast (u \ast p) = t \ast (u \ast p)$, hence taking inverses in the group $u \ast I$ we have $r \ast (u \ast p) \ast (u \ast p)^{-1} = t \ast (u \ast p) \ast (u \ast p)^{-1}$, so $r \ast u = t \ast u$, so $r = t$. Finally, the map is continuous as a restriction of a continuous map $-\ast p : S_{x}^{\dagger}(\mathcal{G},G) \to S_{x}^{\dagger}(\mathcal{G},G)$.
The ``moreover'' part follows directly from associativity. 
\end{proof}

\subsection{Compact ideal subgroups (CIG1)} 
We define CIG1 semigroups and show that under this assumption, we can describe a minimal left ideal of the semigroup of measures. 

\begin{definition}\label{def: CIG1} We say that the semigroup $\left(S_{x}^{\dagger}(\mathcal{G},G), \ast \right)$ is \emph{CIG1} (or ``\emph{admits compact ideal subgroups}'') if there exists some minimal left ideal $I$ and idempotent $u \in I$ such that $u * I$ is a compact group with the induced topology from $I$. We let $h_{u*I}$ denote the normalized Haar measure on $u *I$ and define the Keisler measure $\mu_{u*I} \in \mathfrak{M}_x(\mathcal{G})$ as follows: 
\begin{equation*} 
\mu_{u*I}(\varphi(x)) := h_{u*I} \left([\varphi(x)] \cap u*I \right).
\end{equation*} 
\end{definition} 


\begin{remark} Suppose that $\left(S_{x}^{\dagger}(\mathcal{G},G), \ast \right)$ is CIG1. Then any minimal left ideal witnesses this property, i.e.~for any minimal left ideal $J$ of $S_{x}^{\dagger}(\mathcal{G},G)$ there exists an idempotent $v \in J$ such that $v *J$ is a compact group with the induced topology.
\end{remark}
\begin{proof} Suppose $\left(S_{x}^{\dagger}(\mathcal{G},G), \ast \right)$ is CIG1. Fix a minimal left ideal $I$ and an idempotent $u$ in $I$ such that $u*I$ is a compact group. Let $J$ be any other minimal left ideal. By Fact \ref{fact:Ellis}(6) there exists an idempotent $v \in J$ such that $u*v = v, v \ast u = u$, and the map $\left(-*v\right)|_{I}:I \to J$ is a homeomorphism mapping $u \ast I$ to $v \ast J$. Note that the restriction to $u \ast I$ is a group homomorphism (indeed, for $p_1, p_2 \in u \ast I$, $\left( p_1 \ast v \right) \ast \left( p_2 \ast v \right) = p_1 \ast v \ast u \ast p_2 \ast v = p_1 \ast u \ast p_2 \ast v = \left( p_1 \ast p_2 \right) \ast v$), hence a continuous group isomorphism. Since is is also a homeomorphism onto its range $v \ast J$, as the restriction of a homeomorphism, it follows that $v \ast J$ is a compact group.
\end{proof}  

\begin{lemma}\label{def:ellis} The semigroup $\left(S_{x}^{\dagger}(\mathcal{G},G), \ast \right)$ is CIG1 if any of the following holds:
\begin{enumerate} 
\item for some minimal left ideal $I$, every $p \in I$ is definable;
\item the ideal group of $S_{x}^{\dagger}(\mathcal{G},G)$ is finite.  
\end{enumerate}
\end{lemma}  

\begin{proof} (1) Fix $p \in I$ and let $u \in I$ be the unique idempotent such that $p \in u*I$ (by Fact \ref{fact:Ellis}(4)). Since $p$ is definable, the map $\left( p*- \right)|_{I}:I \to I$ is continuous (by Lemma \ref{lem: right cont for def meas}), hence also closed. Since $I$ is compact, the image of $(p*-)|_{I}$ is compact and is equal to $u*I$. Hence $(u * I,*)$ is a compact Hausdorff space, an abstract group, and both left multiplication and right multiplication are continuous. By Fact \ref{Ellis:2}, $(u *I,*)$ is a compact group. 

(2) is obvious.
\end{proof} 

\begin{example}\label{example:CIG1}
\begin{enumerate}
\item  Let $G := (\mathbb{Z},+,<)$, and consider the sets
\begin{gather*} I^{+} := \left\{ q \in S_{x}^{\inv}(\mathcal{G},\mathbb{Z}): (a < x) \in q \textrm{ for all } a \in \mathbb{\mathcal{G}}\right\} \textrm{ and}\\
 I^{-} := \left\{ q \in S_{x}^{\inv}(\mathcal{G},\mathbb{Z}): (x < a) \in q  \textrm{ for all } a \in \mathcal{G} \right\}. 
\end{gather*} 
Then $I := I^{+} \cup I^{-}$ is the unique minimal left ideal of $\left(S_{x}^{\inv}(\mathcal{G},\mathbb{Z}),* \right)$. Note that every type in $I$ is definable (over $\mathbb{Z}$). By Lemma \ref{def:ellis}, the semigroup $\left(S_{x}^{\inv}(\mathcal{G},\mathbb{Z}),* \right)$ is CIG1. The ideal subgroups are $(I^{-},*)$ and $(I^{+},*)$, both isomorphic to $\widehat{\mathbb{Z}}$ as topological groups. 
\item Consider $G := \SL_{2}(\mathbb{R})$ as a definable subgroup in $(\mathbb{R}, \cdot, +)$. If $I$ is a minimal left ideal of $\left(S_{x}^{\fs}(\mathcal{G},\SL_{2}(\mathbb{R})),* \right)$ and $u$ is an idempotent in $I$, then $u * I \cong \mathbb{Z}/2\mathbb{Z}$ by \cite[Theorem 3.17]{GDP}, hence the semigroup is CIG1. Note that $\SL_{2}(\mathbb{R})$ is not definably amenable (\cite[Remark 5.2]{hrushovski2008groups} + \cite[Lemma 4.4(1)]{conversano2012connected}).
\item There exist fsg groups that are not CIG1. Consider the circle group from Remark \ref{min:circle}. The minimal left ideal of $(S^{\fs}(\mathcal{G},S^{1}),*)$ is precisely $S^{\fs}(\mathcal{G},S^{1})$. As in (1), this left ideal can be decomposed into two ideal subgroups as follows. Let $\st: \mathcal{G} \to S^{1}$ be the standard part map. Consider the sets
\begin{gather*} 
I^{R} := \left\{q \in S_{x}^{\fs}(\mathcal{G},G): \textrm{if } b \models q \textrm{, then } C(\st(b), b, a) \textrm{ for any } a \in S^{1} \right\},\\
I^{L} := \left\{q \in S_{x}^{\fs}(\mathcal{G},G): \textrm{if } b \models q \textrm{, then } C(a, b, \st(b)) \textrm{ for any } a \in S^{1} \right\}. 
\end{gather*} 
Then both $I^{R}$ and $I^{L}$ are ideal subgroups which are isomorphic (as abstract groups) to $S^{1}$, and $S_{x}^{\fs}(\mathcal{G},S^{1}) = I^{R} \sqcup I^{L}$. Moreover, $I^{R}$ and $I^{L}$ are dense subsets of $S_{x}^{\fs}(\mathcal{G},S^{1})$. Note that if $I^{R}$ was compact (with the induced topology), we would have $I^{R} = S_{x}^{\fs}(\mathcal{G},S^{1})$, a contradiction. The same argument applies to $I^{L}$. Therefore, $\left(S_{x}^{\fs}(\mathcal{G},\SL_{2}(\mathbb{R})),* \right)$ is not CIG1. 
\end{enumerate}
\end{example} 

\begin{lemma}\label{haar:semi} Assume that $\left( S_{x}^{\dagger}(\mathcal{G},G), \ast \right)$ is CIG1. Let $I \subseteq  S_{x}^{\dagger}(\mathcal{G},G)$ be a minimal left ideal and $u$ an idempotent in $I$ such that $u *I$ is a compact group. Then for any $p \in u*I$ we have $\mu_{u*I} * \delta_{p} = \mu_{u*I}$ and $ \delta_{p} * \mu_{u*I} = \mu_{u*I} $. 
\end{lemma} 

\begin{proof} Fix $p \in u *I$ and $\varphi(x) \in \mathcal{L}_{x}(\mathcal{G})$. Let $G' \prec \mathcal{G}$ be a small model containing $G$ and the parameters of $\varphi$. Let $a \models p|_{G'}$, and let $p^{-1}$ be the unique element of the group $u*I$ such that $p*p^{-1} = u$. 

~

\noindent \textbf{Claim 1.} $\left(\mu_{u*I} * \delta_{p} \right)(\varphi(x)) = \mu_{u*I}(\varphi(x))$.
\begin{proof}
We have the following computation, using right-invariance of the Haar measure $h_{u*I}$ on $u*I$: 
\begin{gather*} \left(\mu_{u*I}* \delta_{p} \right) (\varphi(x))=\int_{S_{y}(G')}F_{\mu_{u*I}}^{\varphi'}d(\delta_{p}|_{G'}) = F_{\mu_{u*I}}^{\varphi'}(p|_{G'}) =  \mu_{u*I}(\varphi(x\cdot a)) \\
 = h_{u*I} \left([\varphi(x \cdot a)] \cap u*I \right) =h_{u*I}\left(\{q\in u*I:\varphi(x\cdot a)\in q\} \right)\\
 = h_{u*I}\left(\{q\in u*I:\varphi(x)\in q * p\} \right) = h_{u*I} \left(\{q\in u*I:\varphi(x)\in q\}*p^{-1} \right)\\
= h_{u*I} \left(\{q\in u*I:\varphi(x)\in q\} \right) = \mu_{u*I}(\varphi(x)). \qedhere
\end{gather*} 
\end{proof}

\noindent \textbf{Claim 2.} $\left(\delta_{p} * \mu_{u*I} \right) (\varphi(x)) = \mu_{u*I}(\varphi(x))$
\begin{proof}
Let $r:S_{y}(\mathcal{G}) \to S_{y}(G')$ be the restriction map. Let $\tilde{\mu}_{u*I}$ be the extension of $\mu_{u*I}$ to a regular Borel probability measure on $S_{x}(\mathcal{G})$.  By construction, $\supp(\tilde{\mu}_{u*I}) = \sup(\mu_{u*I}) = u*I$ and $\tilde{\mu}_{u*I}|_{u*I} = h_{u*I}$. Using left-invariance of $h_{u*I}$ we have:
\begin{gather*} \left( \delta_{p} * \mu_{u*I} \right) (\varphi(x)) = \int_{S_{y}(G')} F_{\delta_{p}}^{\varphi'} d( \mu_{u*I}|_{G'}) = \int_{S_{y}(\mathcal{G})} \left( F_{\delta_{p}}^{\varphi'}  \circ r \right) d \mu_{u*I}\\
= \tilde{\mu}_{u*I} \left( \left\{q \in S_{x}(\mathcal{G}): \left(F_{\delta_{p}}^{\varphi'} \circ r \right)(q) = 1\right\} \right) = \tilde{\mu}_{u*I} \left(\left\{q \in S_{x}(\mathcal{G}): \varphi(x) \in p * q \right\} \right)\\
= \tilde{\mu}_{u*I}(\{q \in u * I: \varphi(x) \in p * q \}) = h_{u*I} \left(p^{-1} * \{q \in u * I: \varphi(x) \in q \} \right)\\
 = h_{u*I}(\{q \in u * I: \varphi(x) \in q \}) = \mu_{u*I}(\varphi(x)). \qedhere
\end{gather*} 
\end{proof}

\noindent Hence the statement holds. 
\end{proof} 

\begin{lemma}\label{eats} Assume that $\left( S_{x}^{\dagger}(\mathcal{G},G), \ast \right)$ is CIG1. Let $I \subseteq  S_{x}^{\dagger}(\mathcal{G},G)$ be a minimal left ideal and $u$ an idempotent in $I$ such that $u *I$ is a compact group. Then for any $p \in I$ we have $\mu_{u*I} * \delta_{p} = \mu_{u*I}$. 
\end{lemma} 

\begin{proof} 
For any $p \in I$ we have
\begin{equation*}\mu_{u*I} * \delta_p = (\mu_{u*I} * \delta_{u}) * \delta_p = \mu_{u*I} *(\delta_{u} * \delta_p) = \mu_{u*I} * \delta_{u * p} = \mu_{u*I}, 
\end{equation*} 
where the first and the last equalities are by Lemma \ref{haar:semi}, as $u, u * p \in u* I$. 
\end{proof}

\begin{theorem}\label{thm: CIG1 min ideal} Assume that  $\left( S_{x}^{\dagger}(\mathcal{G},G), \ast \right)$ is CIG1. Let $I \subseteq S_{x}^{\dagger}(\mathcal{G},G)$ be a minimal left ideal and $u$ an idempotent in $I$ such that $u *I$ is a compact group. Then $\mathfrak{M}(I) * \mu_{u*I}$ is a minimal left ideal of $\left( \mathfrak{M}^{\dagger}_{x}(\mathcal{G},G), \ast \right)$, containing an idempotent $ \mu_{u*I}$. 
\end{theorem}

\begin{proof} 
We first argue that $\mu_{u*I}$ is an element of some minimal left ideal of $\mathfrak{M}_{x}^{\dagger}(\mathcal{G},G)$. We know that $\mathfrak{M}(I)$ is a closed (by Fact \ref{fact:Ellis}  and Lemma \ref{closed}) left ideal of $\left(\mathfrak{M}_{x}^{\dagger}(\mathcal{G},G), \ast \right)$ (by Proposition \ref{ideal}). Hence there exists some $L \subseteq \mathfrak{M}(I)$ such that $L$ is a minimal left ideal of $ \left( \mathfrak{M}_{x}^{\dagger}(\mathcal{G},G), \ast \right)$, and we show that $\mu_{u*I} \in L$. 
Let $\nu \in \mathfrak{M}(I)$ be arbitrary. If $p \in \sup(\nu)$, then $p \in I$. By Lemma \ref{eats}, we then have $\mu_{u*I} * \delta_{p} = \mu_{u*I}$ for every $p \in \sup(\nu)$. By Lemma \ref{eats:measures} this implies  $\mu_{u*I} * \nu = \mu_{u*I}$, hence $ \mu_{u*I} * \mathfrak{M}(I) = \{\mu_{u*I}\}$. In particular $\mu_{u * I} * L = \{\mu_{u*I}\}$, and since $L$ is a left ideal this implies $\mu_{u*I} \in L$ (and also that $\mu_{u*I}$ is an idempotent). 

Then $\mathfrak{M}^{\dagger}_{x}(\mathcal{G},G) * \mu_{u*I} = L$ by Fact \ref{fact:Ellis}(5). We also have that $L * \mu_{u*I} = L$ since $\mu_{u*I} \in L$ and $L$ is a minimal left-ideal. Thus 
\begin{equation*}
L = L * \mu_{u*I} \subseteq \mathfrak{M}(I) * \mu_{u*I} \subseteq \mathfrak{M}_{x}^{\dagger}(\mathcal{G},G) * \mu_{u*I} = L. 
\end{equation*} 
So $\mathfrak{M}(I) * \mu_{u*I} = L$, hence  $\mathfrak{M}(I) * \mu_{u*I}$  is a minimal left ideal of $\left(\mathfrak{M}_{x}^{\dagger}(\mathcal{G},G), \ast \right)$. 
\end{proof}

\begin{corollary}\label{cor: all ideals in CIG1} Suppose that $\left( S_{x}^{\dagger}(\mathcal{G},G), \ast \right)$ is CIG1. Let $I$ be a minimal left ideal and $u$ an idempotent in $I$ such that $u * I$ is a compact group. Let $J$ be any minimal left ideal of $\left(\mathfrak{M}_{x}^{\dagger}(\mathcal{G},G), \ast \right)$. Then   $J$ and $\mathfrak{M}(I) * \mu_{u*I}$ are affinely homeomorphic. 
\end{corollary} 
\begin{proof}
By Fact \ref{fact:Ellis}	(6), Lemma \ref{convex:1}, and Theorem \ref{thm: CIG1 min ideal}. 
\end{proof}
\subsection{Compact ideal subgroups in minimal ideals with Hausdorff quotients (CIG2)} In this section we define \emph{CIG2 semigroups} and show that under this stronger assumption, any minimal left ideal of $\left( \mathfrak{M}_{x}^{\dagger}(\mathcal{G},G), \ast \right)$ is affinely homeomorphic to the space of regular Borel probability measures over a certain compact Hausdorff space given by  a  quotient of a minimal left ideal in $\left( S_{x}^{\dagger}(\mathcal{G},G), \ast \right)$.

\begin{definition} Let $I$ be a minimal left ideal in $\left(S_{x}^{\dagger}(\mathcal{G},G), \ast \right)$. We define the quotient space $K_{I} := I/\sim$, where $p \sim q$ if and only if $p$ and $q$ are elements of the same ideal subgroup of $I$, i.e.~there exists some idempotent $u \in I$ such that $p,q \in u *I$. We endow $K_{I}$ with the induced quotient topology and write elements of $K$ as $[u * I]$, where $u$ is an idempotent in $I$.  
\end{definition} 
\noindent The quotient topology on $K_{I}$ is automatically compact, but may not be Hausdorff. CIG2 stipulates that this quotient is Hausdorff.
\begin{definition}\label{def: CIG2} We say that the semigroup $\left( S_{x}^{\dagger}(\mathcal{G},G), \ast \right)$ is \emph{CIG2} if there exists a minimal left ideal $I$ such that:
\begin{enumerate}[($i$)]
\item for any idempotent $u \in I$, $u*I$ is compact; 
\item for any  $p \in I$ and $u' \in \id(I)$, the map $\left(p*-\right)|_{u' \ast I}$ is continuous (note that the range of this map is $u \ast I$, where $u\in \id(I)$ is such that $p \in u \ast I$);
\item $K_{I}$ is Hausdorff. 
\end{enumerate} 
\end{definition}

We remark that in the above definition, $(i)$ follows from $(iii)$ since each $u *I$ is a preimage of a point (hence a  closed set) in $K_I$ under the quotient map. 

%

\begin{lemma}\label{FIG:CIG2}\footnote{We thank the referee for pointing out a more general version of Lemma \ref{FIG:CIG2}, as well as Remark \ref{rem: CIG2 char}.}
The semigroup $\left( S_{x}^{\dagger}(\mathcal{G},G), \ast \right)$ is CIG2 if any of the following holds:
\begin{enumerate}
	\item the ideal group of $\left( S_{x}^{\dagger}(\mathcal{G},G), \ast \right)$ is finite;
	\item for some minimal ideal $I \subseteq S_{x}^{\dagger}(\mathcal{G},G)$, every $p \in I$ is definable.\end{enumerate} 
\end{lemma} 

\begin{proof} 
(1) Assume that the ideal group of $\left( S_{x}^{\dagger}(\mathcal{G},G), \ast \right)$ is finite, then the first two conditions of CIG2 are clearly satisfied, and we show (iii) from Definition \ref{def: CIG2}. Suppose that $I$ is a minimal left ideal in $\left( S_{x}^{\dagger}(\mathcal{G},G), \ast \right)$, and let $u$ be an idempotent in $I$.  Let us denote elements of $u*I$ as $\mathfrak{g}$. Then $u*I$ acts on $I$ on the right via $ p \cdot \mathfrak{g} := p * \mathfrak{g}$, and the orbit equivalence relation under this group action is the same as the equivalence relation $\sim$ in the definition of $K_{I}$. Indeed, $u$ is the identity of $u \ast I$ and $p \ast u = p$ for all $p \in I$ by Fact \ref{fact:Ellis}(2); if $p \cdot \mathfrak{g} = q$ and $p \in u' \ast I$ for some $u' \in \id(I)$, then $q = p \ast \mathfrak{g} \in \left( u' \ast I \right) \ast \mathfrak{g} = u' \ast \left( I \ast \mathfrak{g} \right) \subseteq u' \ast I$; and conversely, if $p,q \in u' \ast I$, using that $u' \ast I$ is a group and Fact \ref{fact:Ellis}(2), we have  $p =\left( q \ast q^{-1} \right) \ast p = q \ast \left( q^{-1} \ast p \right) = q \ast \left( u' \ast r \right) = q \ast (u ' \ast u) \ast r = (q \ast u') \ast (u \ast r) = q \ast (u \ast r) = q \cdot \mathfrak{g}$ for some $r \in I$ and $\mathfrak{g} := u \ast r$. This action is continuous by left continuity of convolution.

So $K_{I} = I/(u*I)$, and the quotient of any Hausdorff space by a continuous finite group action remains Hausdorff. Hence $K_{I}$ is Hausdorff. 

\noindent (2) The conditions (i) and (ii) of CIG2 hold since every type in the minimal left ideal $I$ is definable, as in the proof of Lemma \ref{def:ellis}(1). Let $u \in I$ be an idempotent. Arguing as in (1) we get $K_I = I/\left(u \ast I \right)$. The right action of the group $u \ast I$ on $I$ is continuous on the right, and by the assumption and  Lemma \ref{lem: right cont for def meas} it is also continuous on the left, hence continuous by the Ellis' joint continuity theorem (Fact \ref{Ellis:2}). Thus $I / \left( u \ast I \right)$ is Hausdorff, as the quotient of a Hausdorff space by the continuous action of a compact group.
\end{proof} 

\noindent The next fact follows directly from the definitions and Fact \ref{Ellis:2}. 

\begin{remark}\label{rem: CIG2 implies CIG1} If $\left( S_{x}^{\dagger}(\mathcal{G},G), \ast \right)$ is CIG2, then it is CIG1. Moreover, if $I$ is a minimal left ideal of $\left( S_{x}^{\dagger}(\mathcal{G},G), \ast \right)$ witnessing CIG2, then for any idempotent $u \in I$, $u * I$ is a compact group with the induced topology. Thus for every idempotent $u$ in $I$, the measure $\mu_{u *I}$ is well-defined. 
\end{remark} 

\begin{remark}\label{rem: CIG2 char}
\begin{enumerate}
	\item In the proof of Lemma \ref{FIG:CIG2}(2), it suffices to assume that for some idempotent $u \in I$, $u \ast I$ is closed and that for all $p \in I$, the map $p \ast -|_{u \ast I}$ is continuous.
	\item We also have the following equivalence: CIG2 holds if and only if CIG1 holds, and the map $u' \ast -|_{u \ast I}$ is continuous for some $u$ witnessing CIG1 and every idempotent $u' \in I$.
	
	Indeed, since $u \ast I$ is compact, it follows that each $u' \ast I$ is compact, and thus closed and $u' \ast - |_{u \ast I}$ is a homeomorphism. Since it is also a group isomorphism, each $u' \ast I$ is a compact group. Now, given any $p \in u' \ast I$, we have $p = u' \ast p = u' \ast u \ast p$, so left multiplication by $p$ of elements of $u \ast I$ is the composition of left multiplication by $u \ast p \in u \ast I$ (continuous since $u \ast I$ is a topological group) and left multiplication by $u'$ (continuous by assumption), hence it is continuous and we conclude by (1).
\end{enumerate}
\end{remark}

\begin{example}\label{ex: CIG2 one} Both examples $(1)$ and $(2)$ from Example \ref{example:CIG1} are CIG2. 
\begin{enumerate} 
\item The semigroup $\left( S_{x}^{\inv}(\mathcal{G},\mathbb{Z}), \ast \right)$ is CIG2 by Lemma \ref{FIG:CIG2}(2) as all types in $I$ are definable (note that we have $|K_{I}| = 2$).  
\item The ideal group of $\left( S_{x}^{\fs}(\mathcal{G},\SL_{2}(\mathbb{R})), \ast \right)$ is finite  ($\cong \mathbb{Z}/2\mathbb{Z}$), so it is CIG2 by Lemma \ref{FIG:CIG2}(1). 
\end{enumerate} 
\end{example} 

\begin{lemma}\label{need4surjective} Assume that $\left(S_{x}^{\dagger}(\mathcal{G},G), \ast \right)$ is CIG2, and let $I$ be a minimal left ideal witnessing it. Then for any $p \in I$ and $u \in \id(I)$ we have $\delta_{p}* \mu_{u*I} = \mu_{u'*I}$, where $u'$ is the unique idempotent in $I$ such that $p \in u'*I$. 
\end{lemma} 

\begin{proof} Fix $u,u' \in \id(I)$. Then the transition map $\rho_{u,u'}: =(u'*-)|_{u*I}: u \ast I \to u' \ast I$ is an isomorphism of topological groups (it is a group isomorphism by Fact \ref{fact:Ellis}(3), continuous by (ii) in CIG2, and $\rho_{u',u} \circ \rho_{u,u'} = \id_{u*I}$).
 Let $\Phi_{u,u'}: \mathcal{M}(u*I) \to \mathcal{M}(u'*I)$ be the corresponding push-forward map. Note that $\Phi_{u',u} \circ \Phi_{u,u'} = \id_{\mathcal{M}(u*I)}$. Moreover, $\Phi_{u,u'}(h_{u*I}) = h_{u'*I}$ because $\Phi_{u,u'}(h_{u*I})$ is a regular Borel probability measure on $u'*I$ which is right-invariant, and this property characterizes the normalized Haar measure. By a computation similar to the proof of Claim (2) in Lemma \ref{haar:semi}, for any $\varphi(x) \in \mathcal{L}_x(\mathcal{G})$ we have 
\begin{gather*} \left( \delta_{u} * \mu_{u'*I} \right) (\varphi(x)) = h_{u' * I}\left(\left\{ q \in u'*I: \varphi(x) \in u * q \right\} \right)\\
 = \left( \Phi_{u,u'}(h_{u*I}) \right) \left(\left\{ q \in u'*I: \varphi(x) \in u * q \right\}\right) \\
 =  h_{u*I}\left(\rho^{-1}_{u,u'} \left(\left\{q \in u' * I: \varphi(x) \in u * q\right\}\right)\right) \\
= h_{u*I}\left(u *\left\{q \in u' * I: \varphi(x) \in u * q\right\}\right)\\
= h_{u*I}(\{q \in u * I: \varphi(x) \in q\}) = \mu_{u*I}(\varphi(x)),
\end{gather*} 
hence $\delta_{u} * \mu_{u'*I}  = \mu_{u*I}$. Now let  $p \in u' * I$. By Lemma \ref{haar:semi} and the above computation, using that $p = p \ast u'$ by Fact \ref{fact:Ellis}(2), we have
\begin{equation*} \delta_p * \mu_{u*I} = \delta_{p \ast u'} \ast \mu_{u*I} = (\delta_{p} * \delta_{u'}) * \mu_{u*I} = \delta_{p} *( \delta_{u'} * \mu_{u*I}) = \delta_p * \mu_{u'*I} =  \mu_{u'*I}.
\end{equation*} 
\end{proof}

%

\begin{theorem}\label{thm: min ideal in CIG2 affine} Suppose that $\left( S_{x}^{\dagger}(\mathcal{G},G), \ast \right)$ is CIG2. Let $I \subseteq S_{x}^{\dagger}(\mathcal{G},G)$ be a minimal left ideal witnessing CIG2. Then all minimal left ideals of $\left(\mathfrak{M}_{x}^{\dagger}(\mathcal{G},G), \ast \right)$ are affinely homeomorphic to $\mathcal{M}(K_{I})$ (in particular, they are Bauer simplices by Fact \ref{fac: props of Bauer simplex}(2)). 
\end{theorem} 

\begin{proof} 

Let $u \in \id(I)$. By	Remark \ref{rem: CIG2 implies CIG1} and Corollary \ref{cor: all ideals in CIG1}, it suffices to show that $\mathfrak{M}(I) * \mu_{u *I} \cong \mathcal{M}(K_{I})$.
For ease of notation, denote the minimal left ideal $\mathfrak{M}(I) * \mu_{u *I}$ as $L$. Let $q:I \to K_{I}$ denote the (continuous) quotient map, and $q_{*}: \mathcal{M}(I) \to \mathcal{M}(K_{I})$ the corresponding push-forward map. Note that $q_{\ast}$ is affine by Fact \ref{meas:facts}(iii). By Proposition \ref{closed:homeomorphism}, we have an affine homeomorphism $\gamma: \mathfrak{M}(I) \to \mathcal{M}(I)$. Let $\Phi := (q_{*} \circ \gamma)|_{L}$. We claim that $\Phi$ is an affine homeomorphism. 
Note that $\Phi$ is the restriction of the composition of two continuous affine maps, hence $\Phi$ itself is a continuous affine map. It suffices to show that $\Phi$ is a bijection (since it is automatically a closed map as $L$ is compact and $\mathcal{M}(K_I)$ is Hausdorff by Fact \ref{meas:facts}(i) as  $K_I$ is compact Hausdorff by CIG2).

~

\noindent \textbf{Claim 1.} $\Phi$ is surjective. 
\begin{proof}
The extreme points of $\mathcal{M}(K
_I)$ are the Dirac measures concentrating on the elements of $K_I$ (see e.g.~\cite[Example 8.16]{simon2011convexity}). By the Krein-Milman theorem, the set 
\begin{equation*} \left\{\sum_{i=1}^{n} r_i\delta_{[u_{i}*I]}: [u_i *I] \in K_{I}, r_i \in \mathbb{R}_{>0}, \sum_{i=1}^{n} r_i = 1, n \in \mathbb{N} \right\}, 
\end{equation*} 
 is dense in $\mathcal{M}(K_{I})$. Fix some $u_1, \ldots ,u_n \in \id(I)$ and $r_1, \ldots ,r_n \in \mathbb{R}_{>0}$ such that $\sum_{i=1}^{n} r_i =1$. It suffices to find some $\mu \in L$ such that $\Phi(\mu) = \sum_{i=1}^{n} r_i \delta_{[u_{i}*I]}$ (as $\Phi$ is a closed map, it will follow that $\Phi(L) = \mathcal{M}(K_I)$). 
 
Let $\lambda := \sum_{i=1}^{n} r_i \delta_{u_{i}} \in \mathfrak{M}_{x}^{\dagger}(\mathcal{G},G)$. Since $\mu_{u*I} \in L$ (by Theorem \ref{thm: CIG1 min ideal}) and $L$ is a left ideal, also  $\lambda * \mu_{u*I} \in L$. By Lemma \ref{convex:1} and Lemma \ref{need4surjective} we have:
\begin{equation*} \lambda *\mu_{u*I} = \left( \sum_{i=1}^{n} r_i \delta_{u_{i}} \right) * \mu_{u*I}=  \sum_{i=1}^{n} r_i (\delta_{u_{i}}  * \mu_{u*I})  = \sum_{i=1}^{n} r_i \mu_{u_i * I},
\end{equation*} 
and as $\gamma$ and $q_{\ast}$ are affine this implies
\begin{equation*} \Phi \left(\lambda * \mu_{u*I} \right) = \Phi \left( \sum_{i=1}^{n} r_i \mu_{u_i * I} \right) = \sum_{i=1}^{n} r_i q_{*}(\tilde{\mu}_{u_i * I}) = \sum_{i=1}^{n} r_i \delta_{[u_i * I]},
\end{equation*} 
where $\tilde{\mu}_{u_i \ast I} \in \mathcal{M}(I)$ is the unique regular Borel probability measure extending $\mu_{u_i \ast I}$, i.e.~$\tilde{\mu}_{u_i \ast I}(X) = h_{u_i \ast I} (X \cap u_i \ast I)$ for any Borel $X \subseteq I$, where $h_{u_i \ast I}$ is the Haar measure on $u_i \ast I$. Hence $\Phi$ is surjective. 
\end{proof}

\noindent \textbf{Claim 2.} $\Phi$ is injective. 
\begin{proof}
 Suppose that $\lambda$ and $\nu$ are in $L$ and $\lambda \neq \nu$. It suffices to find a continuous function $f: K_{I} \to \mathbb{R}$ such that 
\begin{equation*} \int_{K_I} f d(\Phi(\lambda)) \neq \int_{K_I} f d(\Phi(\nu)).
\end{equation*} 
Since $\lambda \neq \nu$, there exists some $\psi(x) \in \mathcal{L}_{x}(\mathcal{G})$ such that $\lambda(\psi(x)) \neq \nu(\psi(x))$. Consider the function $f_{\psi}: I \to \mathbb{R}$ defined via $f_{\psi}(p) := \left( \delta_{p} * \mu_{u * I} \right) (\psi(x))$. This map is continuous since the map $\left( -*\mu_{u*I} \right)(\psi(x)):\mathfrak{M}_{x}^{\dagger}(\mathcal{G},G) \to \mathbb{R}$ is continuous by the ``moreover'' part of Fact \ref{fac: conv is a semigroup in NIP} (and the map $p \in S^{\dagger}_x(\mathcal{G},G) \mapsto \delta_p \in \mathfrak{M}^{\dagger}_x(\mathcal{G},G)$ is continuous). Moreover, $f_{\psi}$ factors through $q$. Indeed, assume that $q(p_1) = q(p_2)$ for some $p_1,p_2 \in I$. Then there exists some $w \in \id(I)$ such that $p_1,p_2 \in w*I$. Then by Lemma \ref{need4surjective} we have:
\begin{equation*} 
f_{\psi}(p_1) = \left( \delta_{p_1} * \mu_{u*I} \right) (\psi(x)) = \mu_{w * I} (\psi(x)) = \left( \delta_{p_2} * \mu_{u*I} \right)(\psi(x)) = f_{\psi}(p_2).
\end{equation*} 
By the universal property of quotient maps, there exists a unique continuous function $f: K_{I} \to \mathbb{R}$ such that $f_{\psi} = f \circ q$. Since $\lambda \in L \subseteq \mathfrak{M}(I)$ (by the proof of Theorem \ref{thm: CIG1 min ideal}), by Lemma \ref{lem: limit} there exists a net of measures $\left(\Av \left(\overline{p}_{j} \right) \right)_{j \in J}$ such that $\overline{p}_{j} = \left(p_{j,1}, \ldots, p_{j,n_j} \right) \in I^{n_j}$ and $\lim_{j \in J} \Av \left(\overline{p}_j \right) = \lambda$ for each $j \in J$.  As $\gamma$ is an affine homeomorphism, we then have $\gamma(\lambda) = \lim_{j \in J} \left(\frac{1}{n_j} \sum_{k=1}^{n_j} \delta_{p_{j,k}} \right)$.
Hence we have the following computation: 
\begin{gather*} \int_{K_I} f d(\Phi(\lambda)) = \int_{K_I} f d \left(q_{*}\left(\gamma(\lambda) \right) \right) = \int_{I} \left( f \circ q \right) d(\gamma(\lambda)) = \int_{I} f_{\psi} d (\gamma(\lambda))\\   
=\int_{S_{x}(\mathcal{G})} f_{\psi} d (\gamma(\lambda))
 = \int_{S_{x}(\mathcal{G})} f_{\psi} d\left( \lim_{j \in J} \left(\frac{1}{n_j} \sum_{k=1}^{n_j} \delta_{p_{j,k}} \right) \right)\\
\overset{\textrm{Fact \ref{meas:facts}(ii)}}{=} \lim_{j \in J} \int_{S_{x}(\mathcal{G})} f_{\psi} d \left(\frac{1}{n_j} \sum_{k=1}^{n_j} \delta_{p_{j,k}} \right) = \lim_{j \in J} \left( \Av(\overline{p}_{j}) * \mu_{u*I}(\psi(x))  \right)\\
 = \left( \left( \lim_{j \in J} \Av(\overline{p}_{j}) \right) * \mu_{u *I} \right)(\psi(x)) = \left(\lambda * \mu_{u*I} \right) (\psi(x)) = \lambda(\psi(x)),
\end{gather*} 
where the last equality holds by Fact \ref{fact:Ellis}(2), as $\mu_{u \ast I}$ is an idempotent in $L$. A similar computation shows that $\int_{K_I} f d(\Phi(\nu))  = \nu(\psi(x)) \neq \lambda(\psi(x))$, so $\Phi$ is injective. 
\end{proof}
\noindent Claims 1 and 2 establish the theorem.
\end{proof} 

\begin{example} \label{ex: CIG2 examples ideal}
\begin{enumerate}

\item Let $G := \left(\mathbb{R}, +, < \right)$. Then the semigroup $S_{x}^{\inv}(\mathcal{G},\mathbb{R})$ is CIG2. 
Indeed, the unique minimal left ideal of $S_{x}^{\inv}(\mathcal{G},\mathbb{R})$ is $I = \{p_{-\infty}, p_{+\infty}\}$, and both elements of $I$ are idempotents (see Example \ref{example:pair}(3)). The ideal subgroups of $I$ are $\{p_{-\infty}\}$ and $\{p_{+ \infty}\}$, both are obviously compact groups under induced topology. We have $\mathfrak{M}(I) = \left\{r\delta_{p_{-\infty}} + (1-r)\delta_{p_{+ \infty}}: r \in [0,1]\right\}$, and if $u = p_{\pm \infty}$ then $\mu_{u * I} = \delta_{p_{\pm \infty}}$.

Now let $\nu \in \mathfrak{M}(I)$, then $\nu = r \delta_{p_{-\infty}} + s \delta_{p_{+\infty}}$ for some $r,s \in [0,1]$ with $r+s=1$. Then $\nu \ast \mu_{p_{\pm \infty} \ast I} = \left( r \delta_{p_{-\infty}} + s \delta_{p_{+\infty}} \right) \ast \mu_{p_{\pm \infty}  \ast I} =  \left( r \delta_{p_{-\infty}} + s \delta_{p_{+\infty}} \right) \ast \delta_{p_{\pm \infty}} = r \left( \delta_{p_{-\infty}} \ast \delta_{p_{\pm \infty}} \right) + s \left( \delta_{p_{+ \infty}} \ast \delta_{p_{\pm \infty}} \right) = r \delta_{p_{-\infty}} + s \delta_{p_{+ \infty}}$. Therefore 
 $\mathfrak{M}(I) * \mu_{p_{\pm \infty} * I} = \mathfrak{M}(I)$, and so $\mathfrak{M}(I) * \mu_{p_{\pm \infty} * I} \cong \mathcal{M}(\{0,1\})$ is a minimal ideal of $\left(\mathfrak{M}_{x}^{\inv}(\mathcal{G},\mathbb{R}), \ast \right)$.

\item Let $G := \left(\mathbb{Z}, +, < \right)$. Then the semigroup $S_{x}^{\inv}(\mathcal{G},\mathbb{Z})$ is CIG2, the unique minimal left ideal of $S_{x}^{\inv}(\mathcal{G},\mathbb{Z})$ is $I = I^{+} \sqcup I^{-}$ and the ideal subgroups of $I$ are $I^{+}$ and $I^{-}$ (see Examples  \ref{example:CIG1} and \ref{ex: CIG2 one}). Both ideal subgroups are compact groups under induced topology, isomorphic to $\hat{\mathbb{Z}}$ as a topological group. 

Let $u^+ \in I^+, u^- \in I^-$ be the identity group elements in $I^-$ and $I^+$, respectively. Then $\mu_{u^{\pm} \ast I}$ is the Haar measure on $I^{\pm} \cong \hat{\mathbb{Z}}$. For every $\nu \in \mathfrak{M}(I)$ we can write $\nu = r \nu^{-} + s \nu^{+}$ for the measures $\nu^-, \nu^+$ defined by $\nu^{-} \left( \varphi(x) \right) = \frac{\nu \left( \varphi(x) \land x <b \right)}{\nu(x<b)}, \nu^{+} \left( \varphi(x) \right) = \frac{\nu \left( \varphi(x) \land x >c \right)}{\nu(x>c)}$ and $b < \mathbb{Z} < c$. We also have $\nu \ast \mu_{u^{\pm} \ast I} = \left( r \nu^{-} + s \nu^{+} \right) \ast \mu_{u^{\pm} \ast I} = r\left( \nu^{-} \ast \mu_{u^{\pm} \ast I} \right) + s \left( \nu^{+} \ast \mu_{u^{\pm} \ast I} \right) = r \mu_{u^{-} \ast I} + s \mu_{u^+ \ast I}$.
Hence we have $\mathfrak{M}(I) * \mu_{u_{\pm \infty} * I} = \left\{ r \mu_{u^{-} \ast I} + s \mu_{u^{+} \ast I} : r + s = 1   \right\} \cong \mathcal{M}(\{0,1\})$ is a minimal ideal of $\left(\mathfrak{M}_{x}^{\inv}(\mathcal{G},\mathbb{R}), \ast \right)$.
\end{enumerate} 
\end{example} 

\begin{fact}\cite{GDP}\label{fact:SL2R} Let $\mathbb{R} \prec \mathcal{R}$ be a saturated real closed field, $G :=\SL_{2}(\mathbb{R})$ and $\mathcal{G} := \SL_{2}\left(\mathcal{R}\right)$. 
Consider the definable subgroups of $\mathcal{G}$ given by
\begin{gather*}
	\mathcal{T} := \left\{\left( \begin{bmatrix}
x & -y \\
y & x \\
\end{bmatrix}  \right) : x^{2} + y^{2} = 1 \right\}, 
\mathcal{H} := \left\{\left( \begin{bmatrix}
b & c \\
0 & b^{-1} \\
\end{bmatrix}  \right) : b \in \mathcal{R}_{>0}, c \in \mathcal{R} \right\}.
\end{gather*}
Let $p_0 := \tp((b,c)/\mathbb{R})$ such that $b > \mathbb{R}$, and $c > \dcl \left(\mathbb{R} \cup \{b\} \right)$. We view $p_0$ as a type in $S_{\mathcal{H}}(\mathbb{R})$ \footnote{As usual, we denote by $S_{\mathcal{H}}(-)$ the space of types concentrating on the definable set $\mathcal{H}$; all of our results can be  modified in an obvious manner to apply to definable groups in an arbitrary theory, as opposed to theories expanding a group.} identifying $(b,c)$ with the matrix $\begin{bmatrix}
b & c \\
0 & b^{-1} \\
\end{bmatrix}$. 
Let $q_0 := \tp((x,y)/\mathbb{R})$, where $y$ is positive infinitesimal and $x > 0$  is the positive square root of $1-y^{2}$. We view $q_0$ as a type in $S_{\mathcal{T}}(\mathbb{R})$ identifying $(x,y)$ with the matrix $\begin{bmatrix}
x & -y \\
y & x \\
\end{bmatrix}$. We let $r_0$ be $\tp(t \cdot h/\mathbb{R}) \in S_{\mathcal{G}}(\mathbb{R})$, where $h \in \mathcal{H}$ and realizes $p_0$ and $t \in \mathcal{T}$ and realizes the unique coheir of $q_0$ over $\mathbb{R} \cup\{h\}$. Then:
\begin{enumerate}

\item $ S_{\mathcal{G}}^{\fs}(\mathcal{R},\mathbb{R}) * r_0$ is a minimal left ideal of $S_{\mathcal{G}}^{\fs}(\mathcal{R},\mathbb{R})$; 
\item any ideal subgroup of $S_{\mathcal{G}}^{\fs}(\mathcal{R},\mathbb{R}) * r_0$ is isomorphic to $\mathbb{Z}/2\mathbb{Z}$; in particular, if we let $r_1$ be the unique element in $ S_{\mathcal{G}}^{\fs}(\mathcal{R},\mathbb{R}) * r_0$ such that $r_1 * r_1 = r_0$ and $r_1 \neq r_0$, then $\{r_0, r_1\}$ is an ideal subgroup. 
\end{enumerate} 
\end{fact} 

\begin{example}\label{ex: SL2R} Let $\mathcal{G} = \SL_{2}(\mathcal{R})$ and $S_{\mathcal{G}}^{\fs}(\mathcal{R},\mathbb{R})$ be the collection of global types concentrated on $\mathcal{G}$ which are finitely satisfiable in $\SL_{2}(\mathbb{R})$. By Fact \ref{fact:SL2R}, $\{r_0,r_1\}$ is an ideal subgroup of $S_{\mathcal{G}}^{\fs}(\mathcal{R},\mathbb{R})$ which is trivially a compact group with the induced topology, and $\frac{1}{2}(\delta_{r_{0}} +\delta_{r_{1}}) $ is the normalized Haar measure on it. By Theorem \ref{thm: CIG1 min ideal},  $\mathfrak{M}\left(S_{\mathcal{G}}^{\fs}(\mathcal{R},\mathbb{R}) *r_0 \right) *  \frac{1}{2}(\delta_{r_{0}} +\delta_{r_{1}}) $ is a minimal left ideal in $\mathfrak{M}_{\mathcal{G}}^{\fs}(\mathcal{R},\mathbb{R})$. Moreover, this minimal left ideal is affinely homeomorphic to $\mathcal{M} \left(K_{S_{\mathcal{G}}^{\fs}(\mathcal{R},\mathbb{R}) *r_0 } \right)$ by Theorem \ref{thm: min ideal in CIG2 affine} (see the notation there), which is a Bauer simplex with infinitely many extreme points (by Remark \ref{prop: many idemp in non def am}).  
\end{example} 
More generally, we have\footnote{We thank the referee for suggesting the following two remarks.}:
\begin{remark}\label{rem: Ref 1}
	If $\mathcal{G}$ is NIP, not definably amenable and $\left(S^{\dagger}_{x} \left( \mathcal{G},G \right), \ast \right)$ is CIG2, then the quotient $K_I$ is infinite for each minimal ideal $I$ in $\left(S^{\dagger}_{x} \left( \mathcal{G},G \right), \ast \right)$, and the minimal ideals in $\left(\mathfrak{M}^{\dagger}_{x} \left( \mathcal{G},G \right), \ast \right)$ are Bauer simplices, each with infinitely many extreme points (by Fact \ref{fac: props of Bauer simplex}(2), Remark \ref{prop: many idemp in non def am} and Theorem \ref{thm: min ideal in CIG2 affine}). 
\end{remark}

\begin{remark}\label{rem: Ref2}
Assume that  $\mathcal{G}$ is NIP and $\left( S^{\fs}_{x} \left( \mathcal{G},G \right), \ast \right)$ is CIG2. Then the following are equivalent:
\begin{enumerate}
	\item $\mathcal{G}$ is definably amenable,
	\item $\left|K_I \right| = 1$ for each minimal left ideal $I$ in $S^{\fs}_{x} \left( \mathcal{G},G \right)$,
\item  $K_I$ is finite for some minimal left ideal $I$ in $S^{\fs}_{x} \left( \mathcal{G},G \right)$.
\end{enumerate}
\end{remark}
\begin{proof}
	(1)$\Rightarrow$(2). By definable amenability and Proposition \ref{fs:I}, $|J|=1$ for every minimal left ideal $J$ in $\left( \mathfrak{M}^{\fs}_{x} \left( \mathcal{G},G \right), \ast \right)$. By Theorem \ref{thm: min ideal in CIG2 affine}, $J$ is affinely homeomorphic to $\mathcal{M}(K_{I})$ for some minimal left ideal $I$ of $\left( S^{\fs}_{x} \left( \mathcal{G},G \right), \ast \right)$, hence also $|K_I|=1$. By Fact \ref{fact:Ellis}(6), we have $\left| K_{I'} \right| = 1$ for every minimal left ideal $I'$ of $\left( S^{\fs}_{x} \left( \mathcal{G},G \right), \ast \right)$.
	
	(2)$\Rightarrow$(3) is trivial.
	
	(3)$\Rightarrow$(1). By Remark \ref{rem: Ref 1} applied for $\dagger=\fs$.
\end{proof}

\begin{remark}
	The implication (1)$\Rightarrow$(2) in Remark \ref{rem: Ref2} does not hold when $\left( S^{\fs}_{x} \left( \mathcal{G},G \right), \ast \right)$ is replaced by $\left( S^{\inv}_{x} \left( \mathcal{G},G \right), \ast \right)$. Indeed, $(\mathbb{Z}, +, <)$ is NIP, definably amenable, CIG2, but $|K_I|=2$ (see Example \ref{ex: CIG2 one}(1)).
\end{remark}

\begin{question}
	It would be interesting to describe minimal left ideals of the semigroup $\left( \mathfrak{M}^{\dagger}_{x}(\mathcal{G},G), \ast \right)$ for some non-definably amenable groups $\mathcal{G}$ where a description of the minimal left ideals/ideal subgroups of $\left( S^{\dagger}_{x}(\mathcal{G},G), \ast \right)$ is known (other than $\SL_{2}(\mathbb{R})$), including certain algebraic groups definable in $\mathbb{Q}_p$ \cite{penazzi2019some, bao2022definably} or in certain dp-minimal fields \cite{jagiella2021topological}.
\end{question}

\begin{question}
	Is the set of extreme points of a minimal left ideal of $\left( \mathfrak{M}^{\dagger}_{x}(\mathcal{G},G), \ast \right)$ always closed, or at least Borel, in a (not necessarily definably amenable) NIP group $\mathcal{G}$? 
\end{question}

\bibliographystyle{plain}
\bibliography{refs}

\end{document}